%% file: main.tex
\documentclass[10pt,oneside,reqno]{amsart}
\usepackage[T1]{fontenc}
\usepackage[utf8]{inputenc}
\usepackage[british]{babel}
\usepackage{bm,blindtext,amsmath,amssymb,amsthm,mathtools,tikz-cd,mathrsfs,cancel,xcolor,wrapfig,comment,array,pbox,booktabs,slashed,graphicx,subfiles}
\usepackage{enumitem}
\graphicspath{ {./images/} }

\usepackage[colorlinks=true,allcolors=blue,hyperindex,breaklinks,hypertexnames=false]{hyperref}

\calclayout

\numberwithin{equation}{section}
\theoremstyle{plain}
\newtheorem{theorem}{Theorem}[section]
\newtheorem{lemma}[theorem]{Lemma}
\newtheorem{prop}[theorem]{Proposition}
\newtheorem{cor}[theorem]{Corollary}

\newcounter{step}[theorem]
\renewcommand{\thestep}{\arabic{step}}
\newenvironment{step}{
    \refstepcounter{step}
    \par\medskip
    \noindent\textbf{Step \thestep.}\ \itshape
}{
    \par\medskip
}

\theoremstyle{definition}
\newtheorem{defn}[theorem]{Definition}
\newtheorem{ex}[theorem]{Example}

\newtheorem{remark}[theorem]{Remark}

\newcommand{\N}{\mathcal{N}}
\newcommand{\Z}{\mathbb{Z}}

\newcommand{\A}{\mathcal{A}}

\newcommand{\M}{\mathcal{M}}

\newcommand{\D}{\mathcal{D}}
\newcommand{\U}{\mathcal{U}}
\newcommand{\HH}{\mathbb{H}}
\newcommand{\GG}{\mathcal{G}}
\newcommand{\R}{\mathbb{R}}
\newcommand{\C}{\mathbb{C}}
\renewcommand{\P}{\mathbb{P}}
\renewcommand{\L}{\mathcal{L}}
\newcommand{\ra}{\rightarrow}

\newcommand{\SW}{\textnormal{SW}}

\newcommand{\GL}{\textnormal{GL}}
\newcommand{\Diff}{\textnormal{Diff}}

\newcommand{\SL}{\textnormal{SL}}
\newcommand{\SU}{\textnormal{SU}}
\newcommand{\SO}{\textnormal{SO}}

\newcommand{\e}{\varepsilon}

\newcommand{\id}{\textnormal{id}}
\newcommand{\diag}{\textnormal{diag}}
\newcommand{\coker}{\textnormal{coker}}
\newcommand{\im}{\textnormal{im}}
\newcommand{\End}{\textnormal{End}}

\newcommand{\Met}{\textnormal{Met}}

\newcommand{\Fix}{\textnormal{Fix}}

\newcommand{\Spin}{\textnormal{Spin}}
\newcommand{\Sym}{\textnormal{Sym}}
\newcommand{\E}{\mathcal{E}}
\newcommand{\ind}{\textnormal{index}}

\newcommand{\G}{\mathcal{G}}
\renewcommand{\O}{\mathcal{O}}

\newcommand{\del}{\partial}
\newcommand{\delbar}{\overline{\partial}}

\title{Existence of multi-monopoles on mapping tori}
\author{Brad Wilson}

\begin{document}
\renewcommand\Re{\operatorname{Re}}
\renewcommand\Im{\operatorname{Im}}

\begin{abstract}
    While the Seiberg--Witten equations have been well-studied on 3-manifolds, their multiple spinor generalisation exhibits some unexpected behaviour. Most notably, the count of these ``multi-monopoles'' does not define a topological invariant. Instead, the count can jump as parameters of the equations cross between certain regions in the parameter space, known as chambers. This wall-crossing phenomenon is related to deep questions about multi-valued harmonic spinors and higher-dimensional gauge theory. However, concrete examples of this behaviour have not been studied, primarily because the existing constructions of multi-monopoles are not rich enough for wall-crossing to be observed. We address this by proving an adiabatic limit theorem, which constructs multi-monopoles for a wide range of parameters on mapping tori. These solutions are obtained by perturbing the fixed points of the monodromy map associated to a family of multi-vortex moduli spaces. We use our theorem to produce the first explicit constructions of multi-monopoles on non-product 3-manifolds in various chambers. 
\end{abstract}

\maketitle

\subfile{chapters/introduction}

\subfile{chapters/multi-monopoles_on_3-manifolds}

\subfile{chapters/multi-monopoles_on_mapping_tori}

\subfile{chapters/vortex_moduli_space}

\subfile{chapters/fixed_points_and_formal_adiabatic_limit}

\subfile{chapters/the_adiabatic_limit}

\subfile{chapters/counting_multi-monopoles}

\subfile{chapters/appendix}

\bibliographystyle{alpha}
\bibliography{ref}

\end{document}

%% file: chapters/introduction.tex
\section{Introduction}\label{s0}

This article provides the first explicit constructions of multi-monopoles on mapping tori. These are a broad class of 3-manifolds of the form
$$\Sigma_f=\mathbb{R}\times\Sigma/\sim,\hspace{4mm}(t+1,z)\sim(t,f(z)),$$
for a closed oriented surface $\Sigma$ and an orientation-preserving diffeomorphism $f$. We begin by outlining the role of multi-monopoles in gauge theory and the foundational challenges surrounding them, before stating our main results.

\subsection{Background}\label{s1.1}

The multi-monopole equations are a multiple spinor generalisation of the Seiberg--Witten equations. They were first investigated by Bryan and Wentworth on K\"{a}hler surfaces \cite{BryanWentworth1996}, but were only later studied on 3-manifolds by Haydys and Walpuski \cite{HaydysWalpuski2020}. They showed that for certain choices of parameters of the equations (a Riemannian metric, among other data), the moduli space of multi-monopoles can be non-compact. This is a fascinating new phenomenon in gauge theory, distinct from the bubbling that occurs for instantons and pseudo-holomorphic curves. Originally discovered by Taubes for flat $\textnormal{PSL}(2,\C)$-connections \cite{Taubes2013}, this behavior is also exhibited by other generalised Seiberg–Witten equations, such as the Kapustin–Witten and Vafa–Witten equations \cite{Taubes2013,Taubes2017,WalpuskiZhang2021}. The parameter space of the multi-monopole equations is believed to consist of disjoint chambers separated by codimension 1 walls where non-compactness occurs. As parameters vary across these walls, the count of multi-monopoles jumps, so it does not define a 3-manifold invariant. Rigorously establishing this wall and chamber picture is a difficult conjecture, although substantial progress has been made towards it \cite{parker2025,hmt2023}.

Understanding this jumping phenomenon is central to several fundamental problems in gauge theory. For instance, Doan and Walpuski used a wall-crossing formula for this situation to construct $\mathbb{Z}_2$-harmonic spinors on 3-manifolds, though their method is non-constructive \cite{Doan-Walpuski2021}. The jumping behavior is also expected to play an important role in defining Casson-like invariants of compact $G_2$-manifolds \cite{Donaldson-Segal2009}. The count of $G_2$ instantons can also jump due to bubbling, and Haydys and Walpuski conjectured that a correction term can be defined using multi-monopoles \cite{Haydys2012,Walpuski2017G2Instantons,DoanWalpuski_2019}. For these reasons, constructions of multi-monopoles exhibiting this behaviour would lead to new existence results for $\mathbb{Z}_2$-harmonic spinors and a path towards computing the conjectured $G_2$ invariants.

Currently, multi-monopoles have only been directly studied on product 3-manifolds of the form $S^1 \times \Sigma$ using complex geometry. For these spaces, the multi-monopole count is only known in the single chamber containing parameters pulled back from the surface $\Sigma$ \cite{Doan_2018,thakar2025modulispacemultimonopolesriemann}. Because this existing construction is restricted to a single chamber, wall-crossing cannot be observed and the broader behavior of these equations has been left largely unexplored.

\subsection{An adiabatic limit theorem}\label{s1.2}

The goal of this paper is to produce examples of multi-monopoles on mapping tori in a variety of chambers. To do this, we execute and extend a program initiated by Salamon to compute the classical Seiberg--Witten invariants of mapping tori \cite{Salamon1999}. A mapping torus $\Sigma_f$ is a surface bundle over the circle, so by choosing a family of metrics $g_t$ on the surface fibres satisfying the periodicity condition $g_{t+1}=f^*g_t$, one obtains a metric on the total space. This also determines a family of vortex moduli spaces, with a monodromy map induced by $f$. In the adiabatic limit, where the metric is scaled in the fibre direction by a small parameter $\e>0$, the Seiberg--Witten equations degenerate to the fixed point equation of this monodromy map. Salamon outlined the formal aspects of this idea, claiming that for sufficiently small $\e>0$, these fixed points can be perturbed to produce unique monopoles. However, a rigorous analytical proof of this claim was not completed.

Our main result (Theorem \ref{intromaintheorem} below) provides this missing analytical foundation and establishes the existence theorem in the more general setting of multi-monopoles. In addition to a spin$^c$ structure $\mathfrak{s}$ and Riemannian metric $g$, the multi-monopole equations on a 3-manifold $Y$ require an auxiliary connection $B$ on a rank $N$ complex vector bundle. This new feature of the equations is responsible for the non-compactness behaviour. After a generic perturbation of these parameters (and an additional perturbing 2-form, which we omit for now) the equations cut out a regular, zero-dimensional moduli space $\mathcal{M}_{Y}(\mathfrak{s},g,B)$ \cite{Doan_2018}. We will sometimes use the term $N$-monopole if the rank $N$ needs to be specified. 

On a mapping torus $\Sigma_f$, this auxiliary connection will be induced by a family of connections $B_t$ on a vector bundle over each surface fibre. This also needs to satisfy a periodicity condition, which we elaborate on in Section \ref{s2}. If $d$ is the degree of the spinor bundle restricted to a fibre, the data $\mathfrak{p}_t=(g_t,B_t)$ determines a multi-vortex moduli space $\M_\Sigma(d,\mathfrak{p}_t)$ for each $t$. For paths $\mathfrak{p}_t$ avoiding the locus where $\M_\Sigma$ is non-compact, these spaces form a compact fibre bundle over $S^1$ with monodromy map
$$\Upsilon_{d,\mathfrak{p}}:\M_\Sigma(d,\mathfrak{p}_0)\ra\M_\Sigma(d,\mathfrak{p}_0).$$
Our main theorem relates fixed points of $\Upsilon_{d,\mathfrak{p}}$ to multi-monopoles on $\Sigma_f$.

\begin{theorem}\label{intromaintheorem}
    Let $\mathfrak{s}$ be a spin$^c$ structure with fibre degree $d$. Suppose $\mathfrak{p}$ is a suitably generic family of parameters such that $\M_\Sigma(d,\mathfrak{p}_t)$ is compact for all $t$. This determines a metric $g_\e=dt^2+\e^2g_t$ for $\e>0$ and a connection $B$ on $\Sigma_f$. Then the fixed points of $\Upsilon_{d,\mathfrak{p}}$ are all non-degenerate. Moreover, for sufficiently small $\e>0$, the moduli space of multi-monopoles is regular and there is an injective map
    \begin{equation}\label{adiabaticlimitiso}
        \Fix_{\mathfrak{s}}(\Upsilon_{d,\mathfrak{p}})\hookrightarrow\M_{\Sigma_f}(\mathfrak{s},g_{\e},B).
    \end{equation}
    Here, $\Fix_{\mathfrak{s}}(\Upsilon_{d,\mathfrak{p}})$ is a subset of fixed points corresponding to $\mathfrak{s}$, see Definition \ref{fixedpointsubset} for a further explanation.
\end{theorem}

The reader should refer to Theorem \ref{maintheorem} for a more precise statement of this result. Its proof is achieved in three parts. The first part is Proposition \ref{fixedpointsande=0}, which identifies the adiabatic moduli space (obtained by formally setting $\e=0$ in the $g_\e$ multi-monopole equations) with $\Fix_{\mathfrak{s}}(\Upsilon_{d,\mathfrak{p}})$. The key geometric insight is that solutions of the adiabatic limit equations correspond to parallel sections of the family of multi-vortex moduli spaces for a certain symplectic connection. These then correspond to fixed points.

The second part is Proposition \ref{transversality}, which constructs the subset of generic parameters for which all fixed points are non-degenerate. The standard Seiberg--Witten transversality argument fails here, since the linearisation of the adiabatic limit equations has an infinite-dimensional cokernel. To overcome this subtle issue, we relate perturbations of the symplectic connection to perturbations of the multi-monopole equations. This allows us to work in terms of parallel sections, which have an elliptic deformation theory compatible with the implicit function theorem.

The third and most difficult part of the proof is Theorem \ref{ezerotoesmall}, which uses Newton's method to iteratively construct a multi-monopole from a solution of the adiabatic limit equations. This requires the linearised operator to be uniformly invertible in certain $\e$-weighted norms, which relies on our compactness assumption and some weighted elliptic estimates. The proof of these estimates is rather complicated, requiring analytic techniques adapted from the work of Dostoglou and Salamon \cite{Dostoglou-Salamon1994,Salamon2000QuantumPF}. Interestingly, their arguments require certain terms in the square of the linearisation to vanish, or have particular dependencies. We prove the required identities in Appendix \ref{appendix:a}, which rely on some remarkable cancellations. These identities are used repeatedly in our proofs, so this framework can likely be a model for other adiabatic limit problems in gauge theory.

\subsection{Constructing multi-monopoles on mapping tori}

The main challenge in applying Theorem \ref{intromaintheorem} to construct new examples lies in understanding the monodromy map and its fixed points. In the classical theory, the vortex moduli space is $\Sym^d\Sigma$ and the monodromy map is $\Sym^df$ \cite{Salamon1999}. Our case is significantly more complicated for two reasons; the multi-vortex moduli spaces have no such explicit description, and the monodromy map depends on the path of parameters $\mathfrak{p}$. Despite these difficulties, for mapping tori with genus 1 fibres we are able to compute the monodromy map explicitly in Section \ref{s5}. We also obtain an explicit topological description of many chambers using complex geometry. In the $N=2$ case, these chambers can be understood in terms of braids, and the monodromy map is given by the permutation associated to the braid. This leads to the theorem below, which offers the first construction of multi-monopoles on non-product 3-manifolds in various chambers. 

\begin{theorem}\label{mmexistence}
    Let $\Sigma_f$ be a mapping torus with genus 1 fibres. Then the following statements hold:
    \begin{itemize}
        \item Let $\mathfrak{s}_1,\dots,\mathfrak{s}_m$ be distinct spin$^c$ structures on $\Sigma_f$, all with $d=0$. For $N\geq 2$, choose non-negative integers $k_1,\dots,k_m$ such that $k=\sum_i k_i$ satisfies $0\leq k\leq N$. Then there is a chamber such that for every $1\leq i\leq m$, there are (at least) $k_i$ $N$-monopoles on $\Sigma_f$ for the spin$^c$ structure $\mathfrak{s}_i$.
        \item Let $\mathfrak{s}$ be any spin$^c$ structure on $\Sigma_f$ with $d>0$. Then there is a chamber such that there are (at least) $Nd$ $N$-monopoles on $\Sigma_f$ for the spin$^c$ structure $\mathfrak{s}$.
    \end{itemize}
\end{theorem}

\begin{remark}\label{mmremark}
    We plan to improve Theorem \ref{intromaintheorem} to a much stronger result, namely that \eqref{adiabaticlimitiso} is a bijection. It can then be used to count multi-monopoles, and the qualification ``at least'' in Theorem \ref{mmexistence} can be improved to ``exactly''. After taking care of orientations, each of these points will be positively oriented, so by selecting an appropriate chamber the multi-monopole counting function $\Spin^c_{d=0}\ra\Z$ can be anything, as long as it is positive and the total count of multi-monopoles does not exceed $N$. In future work, we will also investigate the more subtle problem of counting multi-monopoles on mapping tori with higher genus fibres.
\end{remark}

\subsection{Wall-crossing for mapping tori}

This concludes our statement of results, but we make some further comments on how the wall-crossing phenomenon manifests for mapping tori. In the space of parameters $\mathcal{P}_\Sigma$ on a surface fibre, Doan proved that the locus $\mathcal{W}_\Sigma\subset\mathcal{P}_\Sigma$ where non-compactness occurs is codimension 2 \cite{Doan_2018}. This means that a homotopy (preserving the periodicity condition) between paths in $\mathcal{P}_\Sigma$ will typically pass through $\mathcal{W}_\Sigma$. Such a homotopy corresponds to a path between the induced parameters on $\Sigma_f$, and wall-crossing occurs as the homotopy passes over $\mathcal{W}_\Sigma$. This concept is visualised in Figure \ref{fig:mappingtoruswallcrossing}. Indeed, the set $\mathcal{W}_\Sigma$ seems to have an incredibly rich structure, so Theorem \ref{intromaintheorem} has the capacity to construct multi-monopoles in a large number of chambers.

\begin{figure}[h!]
    \centering
    \includegraphics[width=0.86\textwidth]{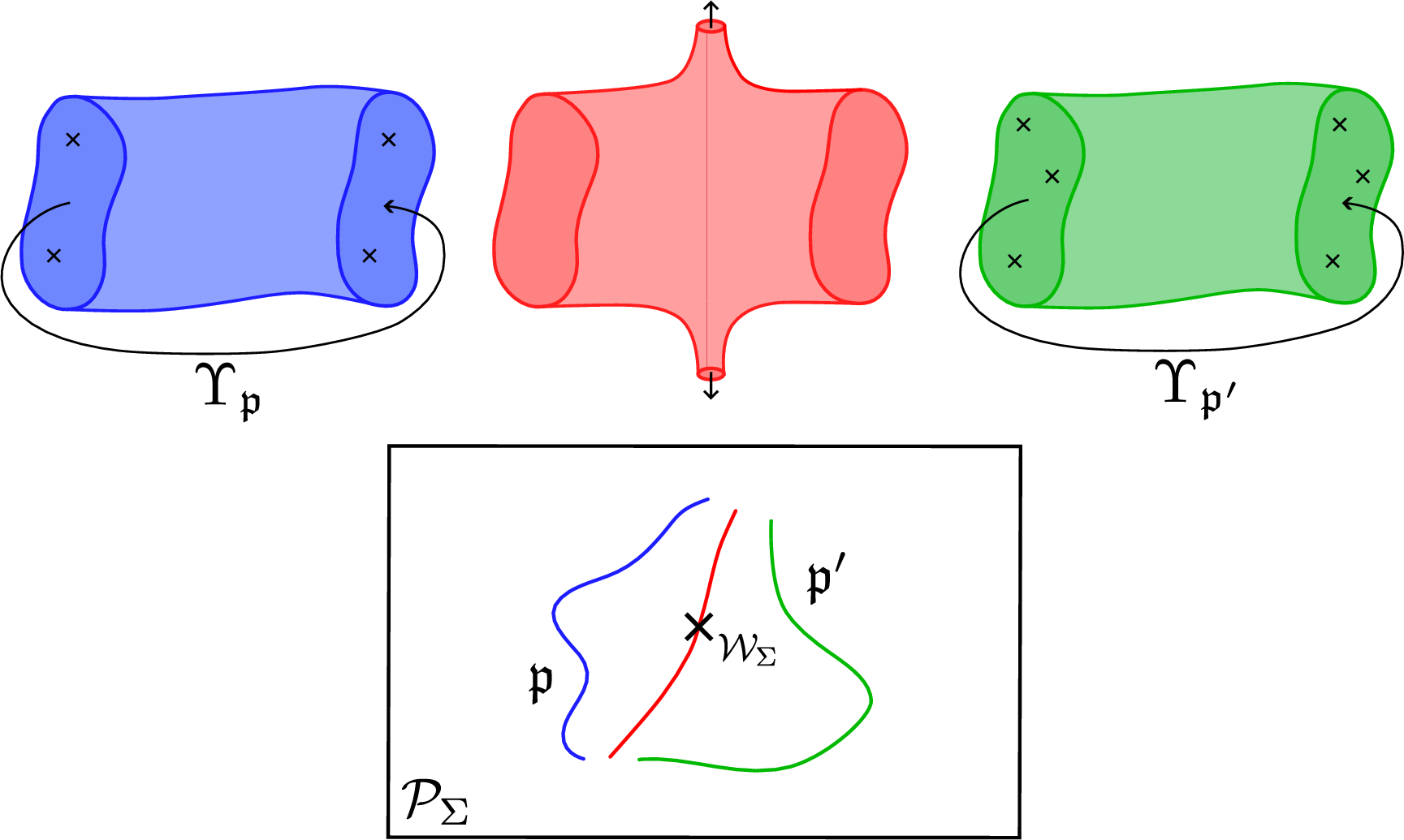}
    \caption{This cartoon illustrates the wall-crossing phenomenon for parameters on a mapping torus. The lower diagram shows three parameters on a mapping torus $\Sigma_f$, visualised as paths in the surface parameter space $\mathcal{P}_\Sigma$. The blue path $\mathfrak{p}$ is homotopic to the green path $\mathfrak{p}'$ (rel the periodicity condition) via the red path. The red path passes through the codimension 2 set $\mathcal{W}_\Sigma$ where $\M_\Sigma$ is non-compact. The pictures above are the families of multi-vortex moduli spaces associated to each of these paths. The red family has a non-compact fibre, which forces the monodromy maps $\Upsilon_{\mathfrak{p}}$ and $\Upsilon_{\mathfrak{p'}}$ to have a different number of fixed points. In the space of parameters on the mapping torus $\Sigma_f$, $\mathfrak{p}$ and $\mathfrak{p}'$ are in different chambers, separated by a wall where the red path lives.}
    \label{fig:mappingtoruswallcrossing}
\end{figure}

\subsection{Acknowledgements}

I am very grateful to my advisor Aleksander Doan for his insights and encouragement with this project. I would also like to thank Ollie Thakar and Calum Crossley for some helpful conversations, and Simon Donaldson for sharing his insights on adiabatic limits. This work was supported by the Engineering and Physical Sciences Research Council [EP/S021590/1]. The EPSRC Centre for Doctoral Training in Geometry and Number Theory (The London School of Geometry and Number Theory), University College London.

%% file: chapters/multi-monopoles_on_3-manifolds.tex
\section{The multi-monopole equations on 3-manifolds}\label{s1}

We briefly review the multi-monopole equations and their deformation theory. The discussion mirrors that of classical Seiberg--Witten theory, except for the non-compactness phenomenon that occurs in the multiple spinor case. For more details, the reader should refer to \cite{Doan_2018,DoanWalpuski2020}.

\subsection{The moduli space of multi-monopoles}

Fix the following data on a closed oriented 3-manifold $Y$:
\begin{itemize}
    \item A spin$^c$ structure $\mathfrak{s}$,
    \item A Riemannian metric $g$,
    \item A unitary connection $B$ on a rank $N$ Hermitian vector bundle $E\ra Y$,
    \item A closed 2-form $\eta\in\Omega^2(Y,i\R)$.
\end{itemize}

Let $S\ra Y$ be the spinor bundle associated to $\mathfrak{s}$, with determinant line bundle $L=\det(S)$. 

\begin{defn}
    With data as above, the multi-monopole equations (or $N$-monopole equations) for a pair $(A,\Psi)\in\A(L)\times\Gamma(S\otimes E)$ are
    \begin{equation}\label{eq:multimonopole}
        \begin{cases}
            \slashed{D}_{AB}\Psi = 0,\\
            \star(F_A + \eta) = \gamma^{-1}\mu(\Psi).
        \end{cases}
    \end{equation}
\end{defn}
Here, $\slashed{D}_{AB}$ is the twisted Dirac operator defined using $B$ and $A$, $\gamma:\Omega^1(i\R)\ra \mathfrak{su}(S)$ is Clifford multiplication, and $\mu(\Psi)$ is the trace-free part of the endomorphism $\Psi\Psi^*$. The gauge group $\G_L=C^\infty(Y,U(1))$ of automorphisms of $L$ acts on a configuration $(A,\Psi)\in\A(L)\times\Gamma(S\otimes E)$ by
$$g\cdot(A,\Psi) = (A - g^{-1}dg, g\Psi).$$
Equation \eqref{eq:multimonopole} is $\G_L$-invariant. A solution $(A,\Psi)$ is called irreducible if its stabiliser is trivial. Otherwise, it is called reducible, which is equivalent to $\Psi=0$. As written, equation \eqref{eq:multimonopole} is not elliptic, even modulo gauge. We must introduce an auxilliary variable $V\in\Omega^0(Y,i\R)$ to obtain an equation which is elliptic modulo gauge:
\begin{equation}\label{eq:modifiedmultimonopole}
    \begin{cases}
        \slashed{D}_{AB}\Psi - V\Psi = 0,\\
        \star(F_{A} + \eta) - dV - \gamma^{-1}\mu(\Psi) = 0.
    \end{cases}
\end{equation}
The gauge group acts trivially on $V$ and this modification introduces no new irreducible solutions. 
\begin{lemma}\cite{Doan_2018}\label{hzero}
    Suppose $(A,V,\Psi)$ is a solution to equation \eqref{eq:modifiedmultimonopole}. Then $V$ is constant and $(A,\Psi)$ is a solution to equation \eqref{eq:multimonopole}. Moreover, if $\Psi\neq 0$ then $V=0$. 
\end{lemma}

In what follows, we implicitly assume that configurations $(A,V,\Psi)$ consist of $L_k^2$ objects and that gauge transformations are $L_{k+1}^2$, omitting the usual subscripts from our notation. 

\begin{defn}
    Let $Y$ and $(\mathfrak{s},g,B,\eta)$ be as above. Then the moduli space of multi-monopoles is the quotient
    $$\mathcal{M}_Y(\mathfrak{s},g,B,\eta) = \{(A,V,\Psi) : \textnormal{\eqref{eq:modifiedmultimonopole}}\}/\G_L.$$
    Similarly, the moduli space of irreducible multi-monopoles is
    $$\mathcal{M}_Y^*(\mathfrak{s},g,B,\eta) = \{(A,\Psi) : \Psi\neq 0,\textnormal{ \eqref{eq:multimonopole}}\}/\G_L.$$
    If any of the parameters are clear from context, they will be omitted from the notation.
\end{defn}

If $B$ and $\eta$ are smooth, all of the moduli spaces with different Sobolev indices are homeomorphic to the moduli space of smooth solutions \cite{Doan_2018}. The local structure of the moduli space near a point $[A,\Psi]\in\mathcal{M}_Y^*$ is controlled by the deformation complex
\begin{equation}\label{deformationcomplex}
    \Omega^0(i\R)\xrightarrow{\mathbb{G}_{A,\Psi}} \Omega^1(i\R)\oplus\Omega^0(i\R)\oplus\Omega^0(S\otimes E)\xrightarrow{\mathbb{S}_{A,\Psi}}\Omega^0(S\otimes E)\oplus\Omega^2(i\R).
\end{equation}
The operator $\mathbb{G}_{A,\Psi}$ is the linearisation of the $\G_{L}$-action at $(A,0,\Psi)$, which is given by
$$\mathbb{G}_{A,\Psi}(v) = (-dv,0,v\Psi).$$
The operator $\mathbb{S}_{A,\Psi}$ is the linearisation of equation \eqref{eq:modifiedmultimonopole} at $(A,0,\Psi)$, after negating the Dirac equation. This is the convention used in \cite{DoanWalpuski2020} and \cite{SalamonBook}, ensuring that the operator $\mathbb{G}^*\oplus \mathbb{S}$ is self-adjoint with respect to the natural inner product (note that \cite{Doan_2018} uses a different convention). Define the quantity
$$\mu(\Psi,\Phi)=\frac{1}{2}\big(\Psi\Phi^* + \Phi\Psi^* - \textnormal{Re}\langle\Psi,\Phi\rangle\big)$$
for spinors $\Psi,\Phi\in\Gamma(S\otimes E)$, which may be viewed as an imaginary-valued 1-form via $\gamma^{-1}$. Then the deformation complex \eqref{deformationcomplex} is equivalent to the one term complex
$$\Omega^1(i\R)\oplus\Omega^0(i\R)\oplus\Omega^0(S\otimes E)\xrightarrow{\mathbb{D}_{A,\Psi}}\Omega^1(i\R)\oplus\Omega^0(i\R)\oplus\Omega^0(S\otimes E),$$
\begin{equation}\label{3dlinearisation}
    \mathbb{D}_{A,\Psi} := \mathbb{G}_{A,\Psi}^*\oplus\mathbb{S}_{A,\Psi} = \begin{pmatrix}
        \star d & -d & - 2\gamma^{-1}\mu(\Psi,\cdot)\\
        -d^* & 0 & -i\textnormal{Im}\langle\Psi,\cdot\rangle\\
        -\gamma(\cdot)\Psi & (\cdot)\Psi & -\slashed{D}_{AB}
    \end{pmatrix}.
\end{equation}
Here, $\mathbb{G}^*$ is the formal $L^2$ adjoint of $\mathbb{G}$. The operator $\mathbb{D}$ is self-adjoint with respect to the $L^2$ inner product
\begin{equation}\label{3dinnerproduct}
    \langle(a,v,\phi),(a',v',\phi')\rangle_{L^2} = -\int_{Y}a\wedge \star a' - \int_{Y}vv'dV_g + \int_{Y}\Re\langle\phi,\phi'\rangle dV_g.
\end{equation}

When $(A,\Psi)$ is a solution of equation \eqref{eq:multimonopole}, the deformation complex \eqref{deformationcomplex} is elliptic of index zero. Associated to this complex are the cohomology groups $H_{A,\Psi}^i$ for $i=0,1,2$. An irreducible solution has $H_{A,\Psi}^0=0$, in which case $H_{A,\Psi}^1 = \ker(\mathbb{D}_{A,\Psi})$ and $H_{A,\Psi}^2 = \coker(\mathcal{D}_{A,\Psi})$ are naturally isomorphic because $\mathbb{D}$ is self-adjoint. 

\begin{defn}
    An irreducible solution $(A,\Psi)$ to equation \eqref{eq:multimonopole} is called \emph{regular} if $H_{A,\Psi}^1=H_{A,\Psi}^2=0$. The moduli space $\M_Y$ is regular if every point is irreducible and regular.
\end{defn}

Near non-regular points $(A,\Psi)$, one settles for constructing a Kuranishi chart. That is, a smooth map $\kappa: H^1_{A,\Psi}\ra H^2_{A,\Psi}$ with $\kappa(0)=0$ and a neighbourhood of $[A,\Psi]\in\M_Y$ modelled by $\kappa^{-1}(0)$. The moduli space $\M_Y$ is Zariski smooth if $\M_Y=\M_Y^*$ and there is a Kuranishi chart around each point with $\kappa=0$. In this case $\M_Y$ is still a smooth manifold, but may have dimension higher than zero (the expected dimension). See \cite{Doan_2018} for more details.

\subsection{Compactness, transversality, and the count of multi-monopoles}

In ordinary Seiberg--Witten theory, there is a codimension $b_1(Y)$ set in the space of closed 2-forms such that if $\eta$ belongs to this set, the moduli space contains reducible solutions. Whenever $b_1>1$, any two choices of metric $g$ and closed 2-form $\eta$ can be joined by a path, which induces a smooth cobordism between the associated moduli spaces. Stokes' theorem then shows that the count of solutions is independent of the parameters, producing a well-defined invariant.

This discussion about $\eta$ extends to the multiple spinor case, although we must now consider the connection $B$ too. The moduli space $\M_Y(\mathfrak{s},g,B,\eta)$ may be non-compact for some choices of $g$ and $B$, and such parameters conjecturally appear in codimension 1. This behaviour spoils the cobordism argument, so the count of solutions is only a chambered invariant. The non-compactness is due to sequences of multi-monopoles $(A_i,\Psi_i)$ where $\|\Psi_i\|_{L^2}$ diverges, which does not occur in the classical case since there is a uniform $L^\infty$ bound on $\Psi$. Haydys and Walpuski proved that in this case, a subsequence converges modulo gauge and rescaling to a \emph{Fueter section} \cite{HaydysWalpuski2020}. These are the ideal limits at the ends of the moduli space.

\begin{defn}
    Let $\mathcal{P}_Y = \Met(Y)\times\A(E)$ and $\mathcal{V}_Y=\ker(d|_{\Omega^2})$. Define the subsets
    $$\mathcal{P}_Y^{cpt} = \{(g,B)\in\mathcal{P}_Y : \M_Y(\mathfrak{s},g,B)\textnormal{ is compact}\},\hspace{4mm}\mathcal{W}_Y=\mathcal{P}_Y\setminus\mathcal{P}_Y^{cpt}.$$ 
\end{defn}

For a fixed $(g,B)\in\mathcal{P}_Y$, it is not enough to choose a generic perturbation $\eta\in\mathcal{V}_Y$ to obtain a regular moduli space. In order to achieve this, the connection $B$ also needs to be varied, although it can be varied an arbitrarily small amount. Recall that a subset of a topological space is called residual if it contains a countable intersection of open dense sets. In particular, residual subsets of complete metric spaces are dense by the Baire category theorem. Doan proved the following transversality result. 

\begin{prop}\cite{Doan_2018}
    Fix a Riemannian metric $g$ on $Y$. Then the subset 
    $$\mathcal{U}^{reg}_Y(g)=\{(B,\eta)\in\mathcal{P}_Y\times\mathcal{V}_Y : \M_Y(\mathfrak{s},g,B,\eta)\textnormal{ is regular}\}$$
    is residual in $\A(E)\times\mathcal{V}_Y$. The subset $\mathcal{U}_{Y}^{reg}=\{(g,B,\eta)\in\mathcal{P}_Y\times\mathcal{V}_Y :(B,\eta)\in\mathcal{U}^{reg}_{Y}(g)\}$ is also residual in $\mathcal{P}_Y\times\mathcal{V}_Y$.
\end{prop}

To count multi-monopoles, the moduli space needs to be oriented. A natural orientation is constructed in \cite[Section 2.6]{Doan_2018}. When $\M_Y$ is regular, this amounts to attaching a sign $\nu(x)=\pm1$ to each point $x\in\M_Y$. 
\begin{defn}\label{countingmonopoles2}
    Suppose $b_1(Y)>1$. For $(g,B,\eta)\in(\mathcal{P}_Y^{cpt}\times\mathcal{V}_Y) \cap \mathcal{U}_Y^{reg}$, define
    $$\SW_Y(\mathfrak{s},g,B) = \sum_{x\in\M_Y(\mathfrak{s},g,B,\eta)}\nu(x)\in\Z.$$
    More generally, for a chamber $C$ (that is, a connected component of $\mathcal{P}_{Y}^{cpt}$) we define $\SW_Y(\mathfrak{s},C)$ by the same formula for any $(g,B,\eta)\in (C\times\mathcal{V}_Y)\cap\mathcal{U}_Y^{reg}$. If $b_1(Y)=1$, we obtain two numbers $\SW_Y^\pm$ for $\eta$ on each side of the codimension 1 wall in $\mathcal{V}_Y$. 
\end{defn}

This definition is well-defined (i.e. independent of the chosen regular point in the chamber $C$) by the cobordism argument outlined above (see Theorem 1.3 and Proposition 2.31 in \cite{Doan_2018} for a proof). It is easy to check that $\SW_Y$ only depends on the gauge equivalence class of $B$.

%% file: chapters/multi-monopoles_on_mapping_tori.tex
\section{The multi-monopole equations on mapping tori}\label{s2}

In \cite{Salamon1999}, Salamon derives an explicit form of the Seiberg--Witten equations on a mapping torus. In this section we adjust this to the multiple spinor case. We begin by reviewing spin$^c$ structures and spin$^c$ connections on mapping tori, before studying the multi-monopole equations and their linearisation in this setting.

\subsection{\texorpdfstring{Spin$^c$ structures on mapping tori}{Spinc structures on mapping tori}}\label{s3.1}

Fix a closed orientable surface $\Sigma$ with a volume form $\omega$. Let $\mathcal{J}(\Sigma,\omega)$ be the space of complex structures compatible with $\omega$ and choose $J\in\mathcal{J}(\Sigma,\omega)$, which equips $\Sigma$ with the structure of a Riemann surface. Then we can form the canonical line bundle $K=\Lambda^{1,0}_J T^*\Sigma$. Interpreting its dual $K^*$ as a principal $U(1)$-bundle, it is isomorphic to the unitary frame bundle of $\Sigma$ and we obtain the canonical spin$^c$ structure
$$K^*\times_{U(1)}\Spin^c(2)\ra\Sigma$$
induced by the map
$$U(1)\ra \Spin^c(2)=(U(1)\times U(1))/\Z_2$$
\begin{equation}\label{eq:spincrep}
    e^{i\theta}\mapsto [e^{i\theta/2},e^{i\theta/2}].
\end{equation}
Precomposing the spin$^c$ representation with this map from $U(1)$ yields
$$U(1)\ra \SO(\HH)$$
$$e^{i\theta}\mapsto (q\mapsto e^{i\theta/2}qe^{i\theta/2}).$$
Splitting $\HH$ as $\C\oplus\C j$, we see that this map fixes the $\C j$ factor and rotates $\C$ by $e^{i\theta}$. This means that the canonical spinor bundle on $\Sigma$ is
$$S_{can}=\underline{\C}\oplus K^*.$$
This construction can be extended to produce a canonical spin$^c$ structure on a mapping torus with fibre $\Sigma$. 
Let $\Diff(\Sigma,\omega)$ denote the space of diffeomorphisms of $\Sigma$ preserving $\omega$. Choose $f\in\Diff(\Sigma,\omega)$ and a smooth family of complex structures $J:\R\ra\mathcal{J}(\Sigma,\omega)$ such that $J_{t+1}=f^*J_t$. Consider the mapping torus
$$\Sigma_f = \R\times\Sigma/\sim,\hspace{4mm}(t+1,z)\sim (t,f(z)),$$
which is equipped with the Riemannian metric
\begin{equation}\label{metric}
    g=dt^2 + \omega(\cdot, J_t\cdot).
\end{equation}
The complex structures $J_t$ also determine a family of canonical line bundles $K_t=\Lambda^{1,0}_{J_t}T^*\Sigma_t$ on $\Sigma$. They satisfy $K_{t+1}=f^*K_t$ and their duals fit together to form the line bundle
$$K_f^*=\{(t,z,\psi) : t\in\R, z\in\Sigma, \psi\in K_t^*\}/\sim,\hspace{5mm}(t+1,z,\psi)\sim (t,f(z),\psi\circ d_z f^{-1})$$
over $\Sigma_f$. The canonical spin$^c$ structure on $\Sigma_f$ is then defined as
$$\mathfrak{s}_{can} = K_f^*\times_{U(1)}\Spin^c(3).$$
Here, the map $U(1)\ra\Spin^c(3)$ is obtained from the map \eqref{eq:spincrep} and the inclusion $\Spin^c(2)\subset \Spin^c(3)$. The associated spinor bundle on $\Sigma_f$ is
$$S_{can} = \underline{\C}\oplus K_f^*.$$
The family of Hermitian inner products on $T\Sigma$ given by
$$\langle\cdot,\cdot\rangle_t = \omega(\cdot,J_t\cdot) - i\omega(\cdot,\cdot)$$
induces a Hermitian inner product on $S_{can}$ which is $\C$-antilinear in the second argument\footnote{Note that the Hermitian inner product used by Salamon \cite{Salamon1999} is $\C$-antilinear in the first argument.}. See \cite{Salamon1999} for a formula defining Clifford multiplication. Other spin$^c$ structures can be obtained from $\mathfrak{s}_{can}$ as follows. Let $L\ra \Sigma$ be a Hermitian line bundle of degree $d$ and choose a unitary lift $\tilde{f}:L\ra L$ of $f:\Sigma\ra\Sigma$. This determines a Hermitian line bundle   
$$L_{\tilde{f}} = (\R\times L)/\sim,\hspace{4mm} (t+1,l)\sim (t,\tilde{f}(l))$$
on $\Sigma_f$. The canonical spin$^c$ structure can be twisted by $L_{\tilde{f}}$ to produce the spin$^c$ structure
$$\mathfrak{s}_{d,\tilde{f}} = (L_{\tilde{f}}\otimes K_f^*)\times_{U(1)}\Spin^c(3).$$
Its spinor bundle is
$$S_{d,\tilde{f}} = \mathfrak{s}_{d,\tilde{f}}\times_{\Spin^c(3)}\HH\simeq L_{\tilde{f}}\oplus (L_{\tilde{f}}\otimes K_f^*),$$
which is similarly equipped with a Hermitian inner product and Clifford multiplication. The determinant line bundle is
$$\det(S_{d,\tilde{f}}) = L_{\tilde{f}}^{\otimes 2}\otimes K_f^*.$$
The following proposition classifies spin$^c$ structures on $\Sigma_f$.

\begin{prop}\label{spincaction}
    Let $L\ra \Sigma$ be a Hermitian line bundle of degree $d$ and let $f\in\Diff(\Sigma,\omega)$. Then,
    \begin{enumerate}
        \item The diffeomorphism $f$ lifts to a unitary isomorphism $\tilde{f}:L\ra L$ and any two such lifts differ by a unitary gauge transformation $u\in\G_L$.
        \item Every spin$^c$ structure on $\Sigma_f$ is isomorphic to $\mathfrak{s}_{d,\tilde{f}}$ for a unitary lift $\tilde{f}:L\ra L$ of $f$. 
        \item Two such lifts $\tilde{f}$ and $\tilde{f}'$ determine isomorphic spin$^c$ structures if and only if $\tilde{f}'=u\tilde{f}$, for $[u]\in\ker(1-f^*)$. Here, $u\in\G_L$ and we are identifying $\pi_0(\G_L)$ with $H^1(\Sigma,\Z)$ via the isomorphism $[u]\mapsto [\frac{u^{-1}du}{2\pi i}]$.
    \end{enumerate}
\end{prop}

\begin{proof}
    A unitary lift $\tilde{f}$ of $f$ is equivalent to a choice of unitary isomorphism $f^*L\simeq L$. These bundles have the same degree since $f^*\omega=\omega$, so such a lift exists. This isomorphism is not canonical, but any two choices differ by a unitary gauge transformation. 
      
    For the second part, recall that $\Spin^c(\Sigma_f)$ is a $H^2(\Sigma_f,\Z)$-torsor, so every spin$^c$ structure is $c\cdot \mathfrak{s}_f$ for some $c\in H^2(\Sigma_f,\Z)$. Every such $c$ is the first Chern class of a line bundle $L_c\ra \Sigma_f$. Over the complement of one surface fibre of $\Sigma_f$, $L_c$ is isomorphic to the pullback of a line bundle $L\ra\Sigma$ of degree $d=\langle c,[\Sigma]\rangle$. Then $L_c\simeq L_{\tilde{f}}$ for some $\tilde{f}:L\ra L$ and $c\cdot \mathfrak{s}_f \simeq \mathfrak{s}_{d,\tilde{f}}$.

    Finally, suppose $\tilde{f}'=u\tilde{f}$ for $u\in\G_L$ which is homotopic to $h(h\circ f)^{-1}$ for some $h\in\G_L$. By the proof of the second part, it suffices to show that $L_{\tilde{f}'}\simeq L_{\tilde{f}}$. To do this, choose a path $h_t\in\G_L$ with $h_0=h$ and $h_1=(h\circ f)u$. Then define an isomorphism $L_{\tilde{f}}\ra L_{\tilde{f}'}$ by $[t,l]\mapsto [t,h_t(l)]$, which is well-defined because
    $$(h\circ f)u\tilde{f}(l) = u\tilde{f}(h(l)) = \tilde{f}'(h(l)).$$
    Conversely, every bundle isomorphism $L_{\tilde{f}}\simeq L_{\tilde{f}'}$ is of this form and defines a path $h_t\in\G_L$ satisfying
    $$\tilde{f}'(h_0(l)) = h_1\tilde{f}(l) = h_1(h_0\circ f)^{-1}\tilde{f}(h(l)).$$
    Taking $u=h_1(h_0\circ f)^{-1}$ completes the proof.
\end{proof}

\subsection{\texorpdfstring{Spin$^c$ connections on mapping tori}{Spinc connections on mapping tori}}\label{canonicalspinc}

There is a canonical spin$^c$ connection on the canonical spinor bundle $S_{can}$. We have $K_t^*\simeq T\Sigma_t$ where $\Sigma_t$ is the fibre of $\Sigma_f$ lying over $t\in S^1$, so $K_f^*$ can be identified with the vertical tangent bundle $T^{V}\Sigma_f$ of the fibration $\Sigma_f\ra S^1$. This is invariant under the Levi-Civita connection $\nabla$ of the metric $g$ defined by equation \eqref{metric}. The canonical spin$^c$ connection on $S_{can}$ is
$$\nabla_{can} = d\oplus \nabla|_{T^{V}\Sigma_f},$$
where $d$ is the trivial connection on $\underline{\C}$. From another point of view, this is the spin$^c$ connection determined by $\nabla|_{T^{V}\Sigma_f}$ on $\det(S_{can})=K_f^*$. As for the twisted spin$^c$ structures, any connection on $\det(S_{d,\tilde{f}})$ is of the form $A_f^{\otimes 2}\otimes\nabla|_{T^{V}\Sigma_f}$ for $A_f\in\A_{\Sigma_f}(L_{\tilde{f}})$. We may write
$$A_f=A(t) + b(t)dt,$$
where $A(t)\in\A_\Sigma(L)$ and $b(t)\in C^{\infty}(\Sigma,i\R)$ satisfy
$$A(t+1)=\tilde{f}^*A(t),\hspace{4mm} b(t+1)=b(t)\circ f$$
for all $t\in\R$. The induced spin$^c$ connection on $S_{d,\tilde{f}}$ is $\nabla_{A_f} = A_f\otimes\nabla_{can}$. 



\subsection{The equations on mapping tori}
We now introduce the Seiberg--Witten equations with multiple spinors on mapping tori, which only requires a small modification to the equations found in \cite{Salamon1999}. The idea is that each of the two equations in \eqref{eq:multimonopole} splits into two more equations, roughly corresponding to the fibre and base directions of the mapping torus. Unlike in the product case, these four equations do not completely decouple and it not possible to choose a temporal gauge as in \cite{Doan_2018}. We continue to fix a closed surface $\Sigma$ with volume form $\omega$ and a diffeomorphism $f\in\Diff(\Sigma,\omega)$. The following definition introduces some notation for spaces of various periodic objects. 

\begin{defn}\label{periodicforms}
    Given a complex vector bundle $V\ra\Sigma$ and a lift $\theta:V\ra V$ of $f:\Sigma\ra\Sigma$, define the following spaces of periodic forms, connections, and complex structures:
    \begin{align*}
        \Omega_{f}^k(\Sigma) &= \{\alpha:\R\ra\Omega^k(\Sigma) : \alpha(t+1) = f^*\alpha(t)\},\\
        \Omega_{\theta}^k(\Sigma,V) &= \{\alpha:\R\ra\Omega^k(\Sigma,L) : \alpha(t+1)=\theta^*\alpha(t)\},\\
        \A_{\theta}(\Sigma,V) &= \{A:\R\ra\A(\Sigma,V) : A(t+1)=\theta^*A(t)\},\\
        \mathcal{J}_f(\Sigma,\omega) &= \{J:\R\ra\mathcal{J}(\Sigma,\omega) : J(t+1)=f^*J(t)\}
    \end{align*}
    Given $J\in\mathcal{J}_f(\Sigma,\omega)$, we also define
    $$\Omega_{\theta,J}^{p,q}(\Sigma,V) = \{\alpha:\R\ra\Omega^{p+q}(\Sigma,V): \alpha(t)\in \Omega^{p,q}_{J(t)}(\Sigma,V), \hspace{1mm}\alpha(t+1)=\theta^*\alpha(t)\}.$$
\end{defn}

\begin{defn}\label{swdata}
    Let $\Sigma$, $\omega$, $f$ be as above, and fix a rank $N$ Hermitian vector bundle $E\ra\Sigma$. For a unitary lift $\varphi:E\ra E$, define the parameter space
    $$\mathcal{P}_\varphi = \mathcal{J}_f(\Sigma,\omega)\times\A_{\varphi}(\Sigma,E).$$

\end{defn}



A parameter $(J,B)\in\mathcal{P}_\varphi$ induces a Riemannian metric by equation \eqref{metric}, as well as a unitary connection $B_\varphi$ on the Hermitian vector bundle 
$$E_\varphi = (\R\times E)/\sim,\hspace{3mm}(t+1,e)\sim(t,\varphi(e))$$
on $\Sigma_f$. Along with a lift $\tilde{f}:L\ra L$, this is the data required to write down the multi-monopole equations \eqref{eq:multimonopole} on $\Sigma_f$. The space of spinors on $\Sigma_f$ decomposes as
$$\Omega^0(E_{\varphi}\otimes S_{d,\tilde{f}}) = \Omega^0(E_{\varphi}\otimes L_{\tilde{f}})\oplus\Omega^0(E_{\varphi}\otimes L_{\tilde{f}}\otimes K_f^*).$$
We will often view such a spinor as a pair
\begin{equation}\label{spinordecomposition}
    \begin{matrix}
        (\Phi,\Psi)\in\Omega_{\varphi\otimes\tilde{f}}^0(\Sigma, E\otimes L)\oplus\Omega^{0,1}_{\varphi\otimes\tilde{f},J}(\Sigma, E\otimes L).
    \end{matrix}
\end{equation}
With this understood, Clifford multiplication $\gamma:\Gamma(T\Sigma_f)\ra \End(E_{\varphi}\otimes S_{d,\tilde{f}})$ is given by
$$\gamma(\xi + \tau\del_t)\begin{pmatrix}
    \Phi\\
    \Psi
\end{pmatrix} = \begin{pmatrix}
    -i\tau\Phi - \sqrt{2}\Psi(\xi)\\
    i\tau\Psi + \langle\cdot,\xi\rangle_t\Phi/\sqrt{2}
\end{pmatrix}$$
for $\xi\in T_z\Sigma$ and $\tau\in\R$. From this we obtain the Dirac operator, in the same way as \cite{Salamon1999}.

\begin{lemma}\label{diracoperator}
    Consider the spin$^c$ connection on $\det(S_{d,\tilde{f}})$ determined by $ A_f=A(t)+b(t)dt$ for $A\in\A_{\tilde{f}}(\Sigma,L)$ and $b\in\Omega^0_f(\Sigma,i\R)$, as described in Section \ref{canonicalspinc}. Along with $\mathfrak{p}=(J,B)\in\mathcal{P}_{\varphi}$, this determines the twisted Dirac operator $\slashed{D}_{ A_fB_\varphi}$ appearing in \eqref{eq:multimonopole}. Explicitly, it is given by
    $$\slashed{D}_{A_fB_\varphi}\begin{pmatrix}
        \Phi\\
        \Psi
    \end{pmatrix} = \begin{pmatrix}
        -i\nabla_t & \sqrt{2}\delbar_{A,\mathfrak{p}}^*\\
        \sqrt{2}\delbar_{A,\mathfrak{p}} & i\nabla_t
    \end{pmatrix}\begin{pmatrix}
        \Phi\\
        \Psi
    \end{pmatrix}.$$
    Here, $\delbar_{A,\mathfrak{p}}:=d_{B\otimes A}^{0,1}$ is the Dolbeault operator determined by the unitary connection $B\otimes A$ on $E\otimes L$ and the complex structure $J$, which implicitly depends on $t$. The operator $\nabla_t$ is defined by
    $$\nabla_t\Phi = \dot{\Phi} + b\Phi, \hspace{5mm} \nabla_t\Psi = \dot{\Psi} +b\Psi + \frac{1}{2}\Psi\circ J\dot{J}.$$
\end{lemma}

\begin{defn}\label{configurationspaces}
    For $J\in\mathcal{J}_f(\Sigma,\omega)$, define the affine configuration space
    \begin{equation}
        \mathcal{C}_f = \A_{\tilde{f}}(\Sigma,L)\times\Omega^0_{\varphi\otimes\tilde{f}}(\Sigma,E\otimes L)\times\Omega^0_f(\Sigma,i\R)\times\Omega^0_f(\Sigma,i\R)\times\Omega^{0,1}_{\varphi\otimes\tilde{f},J}(\Sigma,E\otimes L),
    \end{equation}
    and its tangent space
    \begin{equation}
        W_f = \Omega^1_{f}(\Sigma,i\R)\oplus\Omega^0_{\varphi\otimes\tilde{f}}(\Sigma,E\otimes L)\oplus\Omega^0_{f}(\Sigma,i\R)\oplus\Omega^0_{f}(\Sigma,i\R)\oplus\Omega^{0,1}_{\varphi\otimes\tilde{f},J}(\Sigma,E\otimes L).
    \end{equation}
\end{defn}
Although $\mathcal{C}_f$ and $W_f$ depend on $J, \tilde{f}$, and $\varphi$, this data will be omitted from the notation. Note that we have swapped the order of $\Phi$ and $b$, despite the ordering suggested by section \ref{s2}. 

\begin{prop}\label{swequations}\cite{Salamon1999}
    Fix a spin$^c$ structure $\mathfrak{s}_{d,\tilde{f}}$ and let $\mathfrak{p}\in\mathcal{P}_\varphi$. Also let $\eta = \eta_2 - \eta_1\wedge dt$ for $\eta_1\in\Omega^1_f(\Sigma,i\R)$ and $\eta_2\in\Omega^2_f(\Sigma,i\R)$. A tuple
    $(A,\Phi,V,b,\Psi)\in\mathcal{C}_{f}$
    solves the modified Seiberg--Witten equation with multiple spinors \eqref{eq:modifiedmultimonopole} on $\Sigma_f$ if
    \begin{equation}\label{eq:SW1}
        \begin{cases}
            -i\nabla_t\Phi + \sqrt{2}\delbar_{A,\mathfrak{p}}^*\Psi - V\Phi = 0, \\
            i\nabla_t\Psi + \sqrt{2}\delbar_{A,\mathfrak{p}}\Phi - V\Psi = 0,\\
            \star_t(\dot{A}-db+\frac{\alpha_t}{2} + \eta_1) - dV - i\sqrt{2}\textnormal{Im}\langle\Psi,\Phi\rangle = 0,\\
            \star_t(F_{A} + \eta_2) + \frac{i\kappa_t}{2} - \dot{V} - \frac{i}{2}|\Phi|^2 + \frac{i}{2}|\Psi|^2 = 0.
        \end{cases}
    \end{equation}
    Here, $\star_t$ is the Hodge star operator for the metric $\omega(\cdot,J_t\cdot)$ on $\Sigma$ and $\kappa_t$ is its Gauss curvature. The 1-form $\alpha\in\Omega^1_f(\Sigma,i\R)$ is defined by $\alpha_t\circ J=\dot{\nabla}+\frac{1}{2}J\nabla\dot{J}$.
\end{prop}

We will consider perturbations $\eta = \eta_2 - \eta_1\wedge dt$ of the form
\begin{equation}\label{perturbation}
    \eta_1 = -\frac{\alpha}{2}-i\sigma,\hspace{4mm}\eta_2 = i\Big(\frac{\kappa}{2} -\tau\Big)\omega
\end{equation}
for $\sigma\in\Omega^1_f(\Sigma)$ and $\tau\in\Omega^0_f(\Sigma)$. Indeed, every 2-form $\eta$ can be written in this way.

\begin{lemma}
    The perturbation $\eta$ defined by \eqref{perturbation} is closed if and only if $\dot{\tau} + \star_t d\sigma=0$.
\end{lemma}

\begin{proof}
    The curvature of $\nabla_f$ is $F_{\nabla_f} = -\frac{i\kappa}{2}\omega - \frac{\alpha}{2}\wedge dt$ \cite{Salamon1999}, so the Bianchi identity implies
    $$0=dF_{\nabla_f} = -\frac{i\dot{\kappa}}{2}dt\wedge\omega - \frac{d\alpha}{2}\wedge dt.$$
    The exterior derivative of $\eta$ is
    $$d\eta = i\Big(\frac{\dot{\kappa}_t}{2} - \dot{\tau_t}\Big)dt\wedge\omega + \Big(\frac{d\alpha_t}{2} + id\sigma_t\Big)\wedge dt = -i(\dot{\tau} + \star_t d\sigma)dt\wedge\omega,$$
    which vanishes if and only if $\dot{\tau} + \star_t d\sigma=0$. 
\end{proof}

It will be easier to work with the variables $\sigma,\tau$ than $\eta_1,\eta_2$, so we define the following perturbation space.

\begin{defn}\label{perturbationspace}
    Then for $J\in\mathcal{J}_f(\Sigma,\omega)$, define the spaces
    $$\mathcal{V}_f = \{(\sigma,\tau)\in\Omega^1_f(\Sigma)\times\Omega^0_f(\Sigma) : \dot{\tau} + \star_td\sigma=0\},$$
    $$\mathcal{V}_f^{\pm} = \{(\sigma,\tau)\in\mathcal{V}_f : \pm(d-\bar{\tau})<0\},$$
    where $\bar{\tau} := \int_\Sigma\frac{\tau\omega}{2\pi}\in\R$ is necessarily constant. 
\end{defn}



Following Salamon \cite{Salamon1999}, we choose to rescale our variables to
\begin{equation}\label{rescalings}
    \omega^{new}=\frac{1}{2\e^2}\omega^{old},\hspace{4mm}\tau^{new}=2\e^{2}\tau,\hspace{4mm}\Phi^{new} = \sqrt{2}\e\Phi^{old},\hspace{4mm}\Psi^{new} = \e^{-1}\Psi^{old},
\end{equation}
for $\e>0$. This preserves the product $\tau\omega$ and thus the perturbation $\eta$. Geometrically, the metric $g=dt^2+\omega(\cdot,J_t\cdot)$ becomes $g_\e=dt^2 + 2\e^{2}\omega(\cdot,J_t\cdot)$, which degenerates in the fibre direction as $\e\ra 0$. After rescaling, equation \eqref{eq:SW1} becomes
\begin{equation}\label{eq:epsilonSW}
    \begin{cases}
        -i\nabla_t\Phi + \delbar_{A,\mathfrak{p}}^*\Psi - V\Phi=0,\\
        i\nabla_t\Psi + \e^{-2}\delbar_{A,\mathfrak{p}}\Phi - V\Psi=0,\\
        \star_t(\dot{A}-db - i\sigma) - dV - i\textnormal{Im}\langle\Psi,\Phi\rangle = 0,\\
        \e^{-2}(\star_tF_A -\frac{i}{2}|\Phi|^2 + i\tau)  - \dot{V} + \frac{i}{2}|\Psi|^2=0.
    \end{cases}
\end{equation}
Here, the Hodge star operator and Hermitian inner products are determined by the unscaled volume form $\omega$. 

\subsection{The linearisation}

Let $\Xi=(A,\Phi,0,b,\Psi)\in\mathcal{C}_f$ be an irreducible solution to the rescaled multi-monopole equations \eqref{eq:epsilonSW}. A short computation shows that the formal adjoint of the linearisation $\mathbb{G}_{\Xi}$ of the gauge action with respect to the $L^2$ inner product induced by $g_\e$ (with the appropriate rescalings) is
\begin{equation}
    \mathbb{G}_{\Xi}^{*_\e}(\xi) = \e^{-2}(-d^{*_t}a - i\textnormal{Im}\langle\Phi,\phi\rangle) + \dot{c} - i\textnormal{Im}\langle\Psi,\psi\rangle,
\end{equation}
for $\xi=(a,\phi,v,c,\psi)\in W_{f}$. Taking care with the signs, one can also linearise the rescaled multi-monopole equations \eqref{eq:epsilonSW} and write down the operator $\mathbb{G}^{*_\e}_\Xi\oplus\mathbb{S}_{\Xi}$ as in \eqref{3dlinearisation}:
$$\D_{\e}(\Xi):W_{f}\ra W_{f},$$
\begin{equation}\label{linearisation5x5}
    \mathcal{D}_{\e}(\Xi) = \begin{pmatrix}
        \star_t\del_t & -i\textnormal{Im}\langle\Psi,\cdot\rangle & -d & -\star_t d & -i\textnormal{Im}\langle\cdot,\Phi\rangle \\
        -\langle\Psi,(\cdot)^{0,1}\rangle & i\nabla_t & (\cdot)\Phi & i(\cdot)\Phi & -\delbar_{A,\mathfrak{p}}^*\\
        -\e^{-2}d^{*_t} & -\e^{-2}i\textnormal{Im}\langle\Phi,\cdot\rangle & 0 & \del_t & -i\textnormal{Im}\langle\Psi,\cdot\rangle \\
        \e^{-2}\star_t d & -\e^{-2}i\textnormal{Re}\langle\Phi,\cdot\rangle & -\del_t & 0 & i\textnormal{Re}\langle\Psi,\cdot\rangle \\
        -\e^{-2}(\cdot)^{0,1}\Phi & -\e^{-2}\delbar_{A,\mathfrak{p}} & (\cdot)\Psi & -i(\cdot)\Psi & -i\nabla_t
    \end{pmatrix}.
\end{equation}
Here, in order to view this as a map on $W_f$, we swapped the second and fourth rows and columns compared to equation \eqref{3dlinearisation}. 
The operator $\D_\e$ is elliptic and self-adjoint with respect to the $L^2$ metric \eqref{3dinnerproduct} determined by $g_\e$. Consider the block matrix decomposition
\begin{equation}\label{blocklinearisation}
    \D_{\e}(\Xi) = \begin{pmatrix}
        N_{b,\Psi} & G_{A,\Phi} & S_{A,\Phi}^*\\
        \e^{-2}G_{A,\Phi}^* & 0 & L_{\Psi}^*\\
        \e^{-2}S_{A,\Phi} & L_{\Psi} & M_{b,\Psi}
    \end{pmatrix},
\end{equation}
given by
$$G_{A,\Phi} = \begin{pmatrix}
    -d\\
    (\cdot)\Phi
\end{pmatrix},\hspace{5mm}S_{A,\Phi} = \begin{pmatrix}
    \star_t d & -i\textnormal{Re}\langle\Phi,\cdot\rangle\\
    -(\cdot)^{0,1}\Phi & -\delbar_{A,\mathfrak{p}}
\end{pmatrix},$$
$$N_{b,\Psi} = \begin{pmatrix}
    \star_t\del_t & -i\textnormal{Im}\langle\Psi,\cdot\rangle\\
    -\langle\Psi,(\cdot)^{0,1}\rangle & i\nabla_t
\end{pmatrix},\hspace{5mm} L_{\Psi} = \begin{pmatrix}-\del_t \\ (\cdot)\Psi\end{pmatrix},\hspace{5mm} M_{b,\Psi} = \begin{pmatrix} 0 & i\textnormal{Re}\langle\Psi,\cdot\rangle\\
-i(\cdot)\Psi & -i\nabla_t\end{pmatrix}.$$
The subscripts indicate which parts of $\Xi=(A,\Phi,0,b,\Psi)$ each operator depends on. Some useful identities between these operators are proven in Appendix \ref{appendix:a}.

%% file: chapters/vortex_moduli_space.tex
\section{The multi-vortex moduli space}\label{s3}

By considering multi-monopoles on $S^1\times\Sigma$ that are invariant in the $S^1$ direction, equation \eqref{eq:multimonopole} reduces to a system of equations on $\Sigma$ which we will call the multi-vortex equations. The classical vortex equations have been extensively studied (for example, see \cite{Bradlow1990,Garcia-Prada1994}). In that case, the moduli space can be identified with a symmetric product of $\Sigma$. In contrast, the multi-vortex moduli space lacks such an explicit description, although it is understood in some cases \cite{Doan_2018,thakar2025modulispacemultimonopolesriemann}. Throughout this section, fix a closed oriented surface $\Sigma$ with volume from $\omega$, a degree $d$ Hermitian line bundle $L\ra\Sigma$, and a rank $N$ Hermitian vector bundle $E\ra\Sigma$.

\subsection{A dimensional reduction}

Let $Y=S^1\times\Sigma$ and choose the data $(\mathfrak{s},g,B,\eta)$ of equation \eqref{eq:multimonopole} to be pulled back from $\Sigma$, as was considered in \cite{Doan_2018}. Viewing $S^1\times\Sigma$ as a mapping torus with $f=\id$, this is equivalent to choosing trivial lifts $\tilde{f}=\id_L$, $\varphi=\id_E$, a constant path $\mathfrak{p}=(J,B)\in\mathcal{P}_{\varphi}$, and a closed $S^1$-invariant 2-form $\eta$. For such choices of data, a solution $(A,b,\Phi,\Psi)$ to equation \eqref{eq:SW1} is said to be $S^1$-invariant if
$$\dot{A}=0,\hspace{2mm} b=0,\hspace{2mm} \dot{\Phi}=0,\hspace{2mm} \dot{\Psi}=0.$$
\begin{defn}
    Define the spaces
    $$\mathcal{C}_\Sigma = \A_\Sigma(L)\times\Gamma(\Sigma,E\otimes L)\times\Omega^{0,1}(\Sigma,E\otimes L),$$
    $$\mathcal{C}_\Sigma^* = \{(A,\Phi,\Psi)\in\mathcal{C}_\Sigma : (\Phi,\Psi)\neq(0,0)\},$$
    $$\mathcal{P}_\Sigma = \mathcal{J}(\Sigma,\omega)\times\A(E).$$
\end{defn}
\begin{defn}\label{multivortex}
    Let $\mathfrak{p}=(J,B)\in\mathcal{P}_\Sigma$ and $\tau\in\Omega^0(\Sigma)$. The multi-vortex equations on $\Sigma$ are the following system of partial differential equations for $(A,\Phi,\Psi)\in\mathcal{C}_\Sigma$:
    \begin{equation}\label{eq:vortex}
    \begin{cases}
        \delbar_{A,\mathfrak{p}}^*\Psi = 0,\\
        \delbar_{A,\mathfrak{p}}\Phi = 0,\\
        \langle\Psi,\Phi\rangle=0,\\
        \star_tF_A -\frac{i}{2}|\Phi|^2 + \frac{i}{2}|\Psi|^2 + i\tau=0.
    \end{cases}
    \end{equation}
\end{defn}
It is clear that an $S^1$-invariant multi-monopole $(A,0,\Phi,\Psi)$ solves the multi-vortex equations. Note that our convention differs from that of \cite{Doan_2018}, who writes the multi-vortex equations in terms of $\alpha=\Phi$ and $\beta=\Psi^*$. The gauge group $\G_L=C^\infty(\Sigma,U(1))$ of unitary automorphisms of $L$ acts on $\mathcal{C}_\Sigma$ by
$$g\cdot(A,\Phi,\Psi) = (A-g^{-1}dg,g\Phi,g\Psi).$$
Just like the multi-monopole equations, the multi-vortex equations are $\G_L$-invariant. Reducible and irreducible solutions are similarly defined. Consider the number
$$\bar{\tau} = \int_{\Sigma}\frac{\tau\omega}{2\pi}\in\R.$$
Integrating the last equation of \eqref{eq:vortex}, we see that
$$0 = \int_{\Sigma} (-F_A + i\tau\omega - \frac{i}{2}|\Phi|^2\omega +\frac{i}{2}|\Psi|^2\omega) = 2\pi i(d-\bar{\tau}) - \frac{i}{2}\|\Phi\|_{L^2}^2 + \frac{i}{2}\|\Psi\|_{L^2}^2,$$
so $d-\bar{\tau}<0$ implies $\Phi\neq 0$ and $d-\bar{\tau}>0$ implies $\Psi\neq 0$. In particular, there are no reducible solutions in both of these cases. When $d-\bar{\tau}=0$ we have $|\Phi|=|\Psi|$, which allows for the possibility of a reducible solution with $\Phi=\Psi=0$.


\begin{defn}
    Let $\mathfrak{p}\in\mathcal{P}_\Sigma$ and choose $\tau\in \Omega^0(\Sigma)$ with $d-\overline{\tau}<0$. Define the moduli space of multi-vortices $\M_\Sigma(d,\mathfrak{p},\tau)$ as the space of $\G_L$-equivalence classes of solutions to equation \eqref{multivortex}.
\end{defn}

The moduli space $\M_\Sigma$ has the structure of a real analytic space, given by the Kuranishi maps induced by the relevant deformation complex \cite{Doan_2018}. As we have seen, $S^1$-invariant multi-monopoles correspond to multi-vortices. In fact, for data pulled back from $\Sigma$, every multi-monopole on $S^1\times\Sigma$ becomes $S^1$-invariant after a gauge transformation and there is an isomorphism of the moduli spaces. 

\begin{theorem}{\cite{Doan_2018}}\label{S1invariantstuff}
    Let $\pi:S^1\times\Sigma\ra\Sigma$ be the projection to the second factor. Suppose the data $(\mathfrak{s},g,B,\eta)$ on $Y=S^1\times\Sigma$ is pulled back from $\Sigma$, with $g,B$ determined by $\mathfrak{p}\in\mathcal{P}_\Sigma$ and $\eta=-i\tau\omega$ for $\tau\in\Omega^0(\Sigma)$ with $d-\bar{\tau}\neq 0$ (c.f. equation \eqref{perturbation}). Then every multi-monopole differs from an $S^1$-invariant multi-monopole by a gauge transformation in $\G_{\pi^*L}$. Moreover, any two $S^1$-invariant multi-monopoles differ by a gauge transformation in $\G_L$ and there is an isomorphism of real analytic spaces
    $$\M^*_{Y}(\mathfrak{s},g,B,\eta)\simeq \M_{\Sigma}(d,\mathfrak{p},\tau).$$
\end{theorem}

\subsection{Framed vortices and symplectic quotients}

Unless otherwise specified, we will always assume the convention $d-\bar{\tau}<0$. Every result has an analogous statement when $d-\bar{\tau}>0$, with the roles of $\Psi$ and $\Phi$ swapped. We ignore the case $d-\bar{\tau}=0$ because then $\M_\Sigma$ is either empty or contains reducibles. 

\begin{defn}
    Let $\mathfrak{p}\in\mathcal{P}_\Sigma$ and $\tau\in\Omega^0(\Sigma)$ with $d-\bar{\tau}<0$. The framed multi-vortex equations on $\Sigma$ are the following system of equations for $(A,\Phi)\in\A_\Sigma(L)\times\Gamma(\Sigma,E\otimes L)$:
    \begin{equation}\label{eq:framedvortex}
        \begin{cases}
            \delbar_{A,\mathfrak{p}}\Phi = 0 \textnormal{ and }\Phi\neq 0,\\
            \star F_A - \frac{i}{2}|\Phi|^2 +i\tau = 0.
        \end{cases}
    \end{equation}
    Also define the solution space and moduli space of the framed multi-vortex equations by
    $$\mathcal{Z}_\Sigma(d,\mathfrak{p},\tau) = \{(A,\Phi)\textnormal{ solving equation \eqref{eq:framedvortex}}\},$$
    $$\N_\Sigma(d,\mathfrak{p},\tau) = \mathcal{Z}_\Sigma(d,\mathfrak{p},\tau)/\G_L.$$
\end{defn}

The moduli space $\N_\Sigma$ is naturally identified with the subspace of $\M_\Sigma$ where $\Psi=0$. It is equipped with the structure of a real analytic space, induced by the elliptic deformation complex
\begin{equation}\label{framedvortexcomplex}
    \Omega^0(i\R)\xrightarrow{G_{A,\Phi}} \Omega^1(i\R)\oplus\Gamma(E\otimes L) \xrightarrow{S_{A,\Phi}} \Omega^0(i\R)\oplus\Omega^{0,1}(E\otimes L)
\end{equation}
at each solution $(A,\Phi)$ to equation \eqref{eq:framedvortex}. The operators $G_{A,\Phi}$ and $S_{A,\Phi}$ are the linearisation of the gauge action and equation \eqref{eq:framedvortex}, respectively. Explicitly, they are
$$G_{A,\Phi} = \begin{pmatrix}
    -d\\
    (\cdot)\Phi
\end{pmatrix},\hspace{6mm} S_{A,\Phi} = \begin{pmatrix}
    \star d & -i\textnormal{Re}\langle\Phi,\cdot\rangle\\
    -(\cdot)^{0,1}\Phi & -\delbar_{A,p}
\end{pmatrix}.$$
Note that these agree with the operators by the same name appearing in the block matrix \eqref{blocklinearisation}, evaluated at a particular $t$. Although $A(t),\Phi(t)$ may not solve equation \eqref{eq:framedvortex} in that case, we can see that the deformation theory of the 3-dimensional equations on $\Sigma_f$ has a contribution from the deformation theory the multi-vortex equations on each surface fibre. We will make this precise in section \ref{s3.5}.

Recall that the classical vortex moduli space can be realised as a symplectic quotient. Similarly, $\M_\Sigma$ and $\N_\Sigma$ can be realised as symplectic quotients, with the moment maps for the $\G_L$ action being the curvature equations in \eqref{eq:vortex} and \eqref{eq:framedvortex}, respectively. The complex gauge group $\G_L^{\C}=C^\infty(\Sigma,\C^*)$ acts on $\mathcal{C}_\Sigma$ by
$$g\cdot(A,\Phi,\Psi) = (A + \overline{g}^{-1}\del\overline{g} - g^{-1}\delbar g, g\Phi,g\Psi).$$
This action preserves the first three multi-vortex equations \eqref{multivortex}, although the moment map equation is only preserved by the real gauge group. It also restricts to an action on $\A_\Sigma(L)\times\Gamma(\Sigma,E\otimes L)$, where the first of the framed multi-vortex equations is preserved. This suggests a Hitchin--Kobayashi correspondence for these spaces, which was proven by Doan \cite{Doan_2018}.

\begin{defn}
    If $E\ra\Sigma$ is a Hermitian vector bundle and $J$ is a complex structure on $\Sigma$, a unitary connection $B\in\A(E)$ is equivalent to a holomorphic structure $\delbar_B$ on $E$. The resulting holomorphic bundle on $(\Sigma,J)$ will be denoted by $\E_B$. 
\end{defn}

\begin{theorem}{\cite{Doan_2018}}\label{vortexmodulispaces}
Fix $\mathfrak{p}=(J,B)\in\mathcal{P}_\Sigma$ and $\tau\in\Omega^0(\Sigma)$ with $d-\bar{\tau}<0$. Then up to diffeomorphism, the moduli spaces $\M_\Sigma(d,\mathfrak{p},\tau)$ and $\mathcal{N}_\Sigma(d,\mathfrak{p},\tau)$ only depend on $d\in\mathbb{Z}$, the complex structure $J\in\mathcal{J}(\Sigma,\omega)$, and the holomorphic bundle $\mathcal{E}_B = (E, \delbar_B)$. The moduli space of multi-vortices $\M_\Sigma(d,\mathfrak{p},\tau)$ admits the following descriptions:
\begin{enumerate}
    \item $\G_L$-orbits of solutions to equation \eqref{eq:vortex}.
    \item $\mathcal{G}_L^{\C}$-orbits of triples $(A, \Phi, \Psi)$ with $\Phi \neq 0$ solving $\bar{\partial}_{A,\mathfrak{p}} \Phi = 0$, $\delbar_{A,\mathfrak{p}}^*\Psi=0$, $\langle\Phi,\Psi\rangle=0$.
    \item Isomorphism classes of triples $(\mathcal{L}, \alpha,\beta)$, where $\mathcal{L} \to \Sigma$ is a holomorphic line bundle of degree $d$, $\alpha \in H^0(\Sigma, \mathcal{E}_B \otimes \mathcal{L})$ is non-zero and $\beta\in H^0(\Sigma,\E_B^*\otimes \L^*\otimes K)$ satisfies $\alpha\beta=0$. Two pairs $(\L,\alpha,\beta)$, $(\L',\alpha',\beta')$ are isomorphic if there is a holomorphic isomorphism $\L\ra\L'$ identifying the sections.
\end{enumerate}
Similarly, the moduli space of framed multi-vortices $\N_\Sigma(d,\mathfrak{p},\tau)$ can be described as:
\begin{enumerate}
    \item $\G_L$-orbits of solutions to equation \eqref{eq:framedvortex}.
    \item $\mathcal{G}_L^{\C}$-orbits of pairs $(A,\Phi)$ with $\Phi \neq 0$ solving $\bar{\partial}_{A,\mathfrak{p}} \Phi = 0$.
    \item Isomorphism classes of pairs $(\mathcal{L}, \alpha)$, where $\mathcal{L} \to \Sigma$ is a holomorphic line bundle of degree $d$ and $\alpha \in H^0(\Sigma, \mathcal{E}_B \otimes \mathcal{L})$ is non-zero.
\end{enumerate}
Moreover, $\N_\Sigma$ is always compact.
\end{theorem}

\begin{remark}
    The second and third interpretations of $\N_\Sigma$ and $\M_\Sigma$ in Theorem \ref{vortexmodulispaces} are independent of $\tau$, so long as $d-\bar{\tau}<0$. For this reason, we will omit $\tau$ from the notation when viewing the spaces in this way. When using the third interpretation, we will conflate $\Sigma$ with the Riemann surface $(\Sigma,J)$ and denote the moduli spaces by $\N_\Sigma(d,\E)$ and $\M_{\Sigma}(d,\E)$ for $\E=\E_B$. 
\end{remark}

Note that the $\Psi$ appearing in equation \eqref{eq:vortex} is an anti-holomorphic section of $\E_B^*\otimes \L_A^*\otimes K$, so to pass between the second and third interpretations of $\M_\Sigma$ we must set $\beta = \Psi^*$. By conjugating, we work purely in terms of holomorphic data on $\Sigma$, allowing us to use techniques from complex geometry to study $\M_\Sigma$. Indeed, the deformation complex consists of complex linear maps, equipping $\N_\Sigma$ and $\M_\Sigma$ with natural complex analytic structures \cite{Doan_2018}. The smooth locus of $\N_\Sigma$ is K\"{a}hler, and it has been shown that $\N_\Sigma$ is a projective variety \cite{thakar2025modulispacemultimonopolesriemann}.

\subsection{Generic behaviour}

Despite the compactness of $\N_\Sigma$, one should not expect $\M_\Sigma$ to be compact for every $\mathfrak{p}\in\mathcal{P}_\Sigma$. If there is a point $[A,\Phi,\Psi]\in\M_\Sigma$ with $\Psi\neq 0$, we see that $[A,\Phi,\lambda\Psi]\in\M_{\Sigma}$ for every $\lambda\in\C$. Such points are therefore the source of non-compactness in $\M_\Sigma$ and correspond to Fueter sections on $S^1\times\Sigma$ \cite{Doan_2018}. It follows that $\M_\Sigma$ is non-compact if and only if some fibre of the projection map 
$$\pi:\M_\Sigma(d,\mathfrak{p},\tau)\ra\N_\Sigma(d,\mathfrak{p},\tau),\hspace{4mm}\pi([A,\Phi,\Psi])=[A,\Phi]$$
is non-trivial. The generic behaviour of $\M_\Sigma$ and $\N_\Sigma$ is summarised by the following theorem.

\begin{theorem}{\cite{Doan_2018}}\label{smoothnesscompactness}
    For every complex structure $J\in\mathcal{J}(\Sigma,\omega)$, there is a residual subset $\A^{reg}_\Sigma(J)\subset\A_\Sigma(E)$ such that for every $B\in\A^{reg}_\Sigma(J)$ and $\mathfrak{p}=(J,B)$, we have:
    \begin{enumerate}
        \item $\mathcal{N}_\Sigma(d,\mathfrak{p},\tau)$ is a compact, smooth, K\"{a}hler manifold of the expected dimension 
        $$Nd - (N-1)(g-1),$$
        where $g$ is the genus of $\Sigma$ and $N$ is the rank of $E$. 
        \item $\mathcal{M}_{\Sigma}(d,\mathfrak{p},\tau)$ is compact, Zariski smooth and biholomorphic to $\mathcal{N}_\Sigma(d,\mathfrak{p},\tau)$. This is equivalent to there being no Fueter sections on $S^1\times\Sigma$ for the pulled back data. 
    \end{enumerate}
\end{theorem}

\begin{defn}
    The subset of regular parameters is defined by
    $$\mathcal{P}_{\Sigma}^{reg} = \{(J,B)\in\mathcal{P}_\Sigma : B\in\A_\Sigma^{reg}(J)\}\subset \mathcal{P}_\Sigma.$$
\end{defn}

We briefly review the K\"{a}hler structure of $\mathcal{N}_\Sigma$. Firstly, consider the symplectic form
\begin{equation}\label{eq:symplecticform}
    \Omega((a,\phi), (a',\phi')) = -\int_\Sigma a\wedge a' - \int_\Sigma\textnormal{Im}\langle \phi,\phi'\rangle\omega
\end{equation}
on $\A_\Sigma(L)\times\Gamma(E\otimes L)$. This is compatible with the complex structure $(a,\phi)\mapsto (\star_J a,i\phi)$ and the standard $L^2$ metric, so in fact $\Omega$ is a K\"{a}hler form. For $\mathfrak{p}\in\mathcal{P}_\Sigma$, define the space
$$\mathcal{X}_{\mathfrak{p}} = \{(A,\Phi)\in\A(L)\times\Gamma(E\otimes L) : \Phi\neq 0,\hspace{1mm}\delbar_{A,\mathfrak{p}}\Phi = 0\}.$$
By definition, the regular parameters $\mathcal{P}_{\Sigma}^{reg}$ are those such that the linearisation of $\delbar_{A,\mathfrak{p}}\Phi=0$ is surjective, which is equivalent to the smoothness of $\mathcal{N}_\Sigma(d,\mathfrak{p},\tau)$ \cite[Corollary 7.2]{Doan_2018}. This linearisation is complex linear, so $\mathcal{X}_\mathfrak{p}$ is a smooth infinite-dimensional complex manifold for $\mathfrak{p}\in\mathcal{P}_\Sigma^{reg}$. The restriction $\Omega_\mathfrak{p}$ of $\Omega$ to the tangent space of $\mathcal{X}_\mathfrak{p}$ is non-degenerate, so $\mathcal{X}_\mathfrak{p}$ is a smooth infinite-dimensional K\"{a}hler manifold. This K\"{a}hler structure on $\mathcal{X}_{\mathfrak{p}}$ descends to the K\"{a}hler quotient $\mathcal{X}_{\mathfrak{p}}/\G_L^{\C}$, which is isomorphic to $\mathcal{N}_\Sigma(d,\mathfrak{p},\tau)$ by Proposition \ref{vortexmodulispaces}.

%% file: chapters/fixed_points_and_formal_adiabatic_limit.tex
\section{Fixed points and the adiabatic limit}\label{s3.5}

By formally setting $\e=0$ in the rescaled multi-monopole equation \eqref{eq:epsilonSW}, we obtain
\begin{equation}\label{eq:epsilonzeroSW1}
    \begin{cases}
        -i\nabla_t\Phi + \delbar_{A,\mathfrak{p}}^*\Psi - V\Phi=0,\\
        \delbar_{A,\mathfrak{p}}\Phi=0,\\
        \star_t(\dot{A}-db - i\sigma) - dV - i\textnormal{Im}\langle\Psi,\Phi\rangle = 0,\\
        \star_tF_A -\frac{i}{2}|\Phi|^2 + i\tau=0,
    \end{cases}
\end{equation}
for $(A,\Phi,V,b,\Psi)\in\mathcal{C}_{f}$ and $\mathfrak{p}\in\mathcal{P}_\varphi$. We will call this the \emph{adiabatic limit equation}. Geometrically, it is obtained as the metric on $\Sigma_f$ shrinks in the fibre direction, or equivalently, as the base is stretched out. The gauge action on $\mathcal{C}_f$ preserves the adiabatic limit equation, so we can define its moduli space of solutions.
\begin{defn}
    Let $\mathfrak{s}_{d,\tilde{f}}$ be a spin$^c$ structure on $\Sigma_f$ and choose $\mathfrak{p}\in\mathcal{P}_{\varphi}$ and $\eta\in\mathcal{V}_f^+$. The solution space of the adiabatic limit equation is
    $$\A_0(\mathfrak{s}_{d,\tilde{f}},\mathfrak{p},\eta) = \{\Xi\in\mathcal{C}_f : \eqref{eq:epsilonzeroSW1}\}.$$
    The periodic gauge group is
    $$\G_f = \{g:\R\ra \G_L : g(t+1)=g(t)\circ f\}.$$
    The adiabatic moduli space is the quotient
    $$\M_0(\mathfrak{s}_{d,\tilde{f}},\mathfrak{p},\eta) = \A_0(\mathfrak{s}_{d,\tilde{f}},\mathfrak{p},\eta)/\G_{f}.$$
\end{defn}
This section explores several different geometric interpretations of equation \eqref{eq:epsilonzeroSW1} and its linearisation. In particular, we show that points in $\M_0(\mathfrak{s}_{d,\tilde{f}},\mathfrak{p},\eta)$ correspond to fixed points of a symplectomorphism of a multi-vortex moduli space.

\subsection{The geometry of the adiabatic limit equation}

Note that the second and fourth equations of \eqref{eq:epsilonzeroSW1} are just the framed multi-vortex equations \eqref{eq:framedvortex} for $A(t),\Phi(t)$. One way to understand the other two equations is to rewrite them in the form
\begin{equation}\label{hodgeparalleltransport}
    G_{A,\Phi}(V) + S_{A,\Phi}^*\begin{pmatrix}b\\\Psi\end{pmatrix} + \begin{pmatrix}\star(\dot{A}-i\sigma)\\i\dot{\Phi}\end{pmatrix}=0.
\end{equation}
By the Hodge decomposition theorem for the elliptic complex \eqref{framedvortexcomplex}, the variables $V,b,$ and $\Psi$ are uniquely determined by $A,\Phi,$ and $\sigma$ for every $t$. The following analogue of Lemma \ref{hzero} shows that $V=0$.

\begin{lemma}\label{hequalszero}
    If $(A,\Phi,V,b,\Psi)\in\A_0(\mathfrak{s}_{d,\tilde{f}},\mathfrak{p},\eta)$ for $\mathfrak{p}\in\mathcal{P}_\varphi$ and $\eta=(\sigma,\tau)\in \mathcal{V}_f^+$, then $V=0$.
\end{lemma}

\begin{proof}
    Note that $d\dot{A} = \del_t F_A = i\Re\langle\dot{\Phi},\Phi\rangle\omega - i\dot{\tau}\omega$, where the first equality follows from the Bianchi identity on $\Sigma_f$ and the second from differentiating $F_A=\frac{i}{2}|\Phi|^2\omega - i\tau\omega$. Applying $G_{A,\Phi}^*$ to equation \eqref{hodgeparalleltransport} and using $\dot{\tau} + \star d\sigma=0$, we obtain
    \begin{align*}
        0 &= G_{A,\Phi}^*G_{A,\Phi}(V) + G_{A,\Phi}^*\begin{pmatrix}
            \star(\dot{A} - i\sigma)\\
            i\dot{\Phi}
        \end{pmatrix}\\
        &= d^*dV -i\Im\langle\Phi,h\Phi\rangle + \star d\dot{A} -i\star d\sigma + i\Im\langle\Phi,i\dot{\Phi}\rangle\\
        &= (\Delta + |\Phi|^2)V +i\Re\langle\dot{\Phi},\Phi\rangle - i\dot{\tau} - i\star d\sigma - i\Re\langle\Phi,\dot{\Phi}\rangle\\
        &= (\Delta + |\Phi|^2)V
    \end{align*}
    for each $t$. Since $(A(t),\Phi(t))$ solves the framed multi-vortex equations \eqref{eq:framedvortex} and $d-\overline{\tau}<0$, we have $\Phi(t)\neq 0$ for all $t$. Then $V(t)=0$ for all $t$ because $\Delta_t + |\Phi(t)|^2$ is invertible. 
\end{proof}

From now on, we will need to assume that $\M_\Sigma(d,\mathfrak{p}_t,\tau_t)$ is compact for all $t$, so we define the following families of regular parameters.

\begin{defn}
    For a family of complex structures $J\in\mathcal{J}_f(\Sigma,\omega)$ and a unitary lift $\varphi:E\ra E$ of $f\in\Diff(\Sigma,\omega)$, define
    $$\A_\varphi^{reg}(J)=\{B\in\A_\varphi(E) : B_t\in\A_\Sigma^{reg}(J_t) \textnormal{ for all }t\}\subset \A_\varphi(E),$$
    $$\mathcal{P}_\varphi^{reg} = \{(J,B)\in\mathcal{P}_{\varphi} : (J_t,B_t)\in\mathcal{P}_\Sigma^{reg}\textnormal{ for all }t\}\subset\mathcal{P}_\varphi.$$
\end{defn}

By Theorem \ref{smoothnesscompactness}, choosing $\mathfrak{p}\in\mathcal{P}_\varphi^{reg}$ ensures that $\M_\Sigma(d,\mathfrak{p}_t,\tau_t)$ is compact and biholomorphic to the smooth K\"{a}hler manifold $\N_\Sigma(d,\mathfrak{p}_t,\tau_t)$ for every $t$. In this case, the multi-vortex moduli spaces fit together to form a smooth fibre bundle.

\begin{defn}\label{familymoduli}
    Choose unitary lifts $\tilde{f}:L\ra L$ and $\varphi:E\ra E$ of $f$. Let $\mathfrak{p}\in\mathcal{P}_\varphi^{reg}$ and $\eta=(\sigma,\tau)\in\mathcal{V}_f^+$. Define the map 
    $$F_{\tilde{f}}:\mathcal{Z}_\Sigma(d,\mathfrak{p}_0,\tau_0)\ra\mathcal{Z}_\Sigma(d,\mathfrak{p}_1,\tau_1),\hspace{4mm}(A,\Phi) \mapsto (\tilde{f}^*A,(\tilde{f}\otimes\varphi)^*\Phi),$$
    which descends to
    $$F:\N_\Sigma(d,\mathfrak{p}_0,\tau_0)\ra\N_\Sigma(d,\mathfrak{p}_1,\tau_1)$$
    at the level of moduli spaces, which is independent of $\tilde{f}$. Also define the fibre bundles
    $$\mathcal{Z}_{\tilde{f}}=\mathcal{Z}_{\tilde{f}}(d,\mathfrak{p},\tau) := \frac{\{(t,z) : z\in\mathcal{Z}_\Sigma(d,\mathfrak{p}_t,\tau_t)\}}{(1,F_{\tilde{f}}(z))\sim(0,z)}\ra S^1,$$
    $$\N_f=\N_f(d,\mathfrak{p},\tau):= \frac{\{(t,x) : x\in\N_\Sigma(d,\mathfrak{p}_t,\tau_t)\}}{(1,F(x))\sim(0,x)}\ra S^1.$$
\end{defn}

We will often view $\mathcal{Z}_{\tilde{f}}\ra\N_f$ as a principal $\G_L$-bundle, where the projection map is the quotient by the fibre-wise $\G_L$-action. Since two of the adiabatic limit equations are the multi-vortex equations, for $\mathfrak{p}\in\mathcal{P}_\varphi^{reg}$ and $\eta\in\mathcal{V}_f^+$ there is a well-defined map
\begin{equation}\label{downstairsadiabaticcorrespondence}
    \M_0(\mathfrak{s}_{d,\tilde{f}},\mathfrak{p},\eta)\ra\Gamma(\N_f(d,\mathfrak{p},\tau))
\end{equation}
given by $[A,\Phi,0,b,\Psi]\mapsto [A,\Phi]$. For the classical Seiberg--Witten equations, Salamon proved that there is a symplectic connection on $\N_f$ for which sections in the image of \eqref{downstairsadiabaticcorrespondence} are parallel \cite[Theorem 5.1]{Salamon1999}. As we now explain, this generalises to the multi-monopole setting for regular paths of parameters. 

\begin{theorem}{\cite{Salamon1999}}\label{paralleltransporttheorem}
    Suppose $\mathfrak{p}:[0,1]\ra \mathcal{P}_\Sigma^{reg}$ is given by $\mathfrak{p}=(J,B)$ and that $\sigma:[0,1]\ra\Omega^1(\Sigma)$, $\tau:[0,1]\ra\Omega^0(\Sigma)$ solve $\dot{\tau}_t + \star_t d\sigma_t=0$. Then for each $t\in[0,1]$ there is a symplectomorphism
    $$P_t:\N_\Sigma(d,\mathfrak{p}_0,\tau_0)\ra\N_\Sigma(d,\mathfrak{p}_t,\tau_t)$$
    defined by $[A(0),\Phi(0)]\mapsto[A(t),\Phi(t)]$, where
    \begin{equation}\label{paralleltransportequations}
        i\dot{A} = \Re\langle\Psi,\Phi\rangle - \sigma,\hspace{3mm}i\dot{\Phi} = \delbar_{A,\mathfrak{p}}^*\Psi,
    \end{equation}
    and $\Psi(t)\in\Omega^{0,1}_{J_t}(\Sigma,E\otimes L)$ is the unique solution to the elliptic equation
    \begin{equation}\label{paralleltransportpsi}
        \delbar_{A,\mathfrak{p}}\delbar_{A,\mathfrak{p}}^*\Psi + \frac{1}{2}\langle\Psi,\Phi\rangle\Phi = \frac{1}{2}(\del_{A,\mathfrak{p}}\Phi)\circ\dot{J} + (\sigma + \dot{B})^{0,1}\Phi.
    \end{equation}
    If $\mathfrak{p}_0=\mathfrak{p}_0$, $\tau(0)=\tau(1)$, and $\int_0^1\sigma(s)ds=0$ then $P_1$ is Hamiltonian. 
\end{theorem}

The proof of Theorem \ref{paralleltransporttheorem} is essentially the same as \cite[Theorem 5.1]{Salamon1999}, so we only outline the differences. The main idea is to construct a universal symplectic connection on a universal bundle. In our case this is the fibre bundle $\mathcal{X}\ra\mathcal{P}_\Sigma^{reg}\times\Omega^1(\Sigma)$ with fibre
$$\mathcal{X}_{\mathfrak{p},\sigma}=\{(A,\Phi)\in\A_\Sigma(L)\times\Gamma(E\otimes L) : \delbar_{A+i\sigma,\mathfrak{p}}\Phi=0,\hspace{1mm}\Phi\neq 0\}$$
over $(\mathfrak{p},\sigma)$. The assumption $\mathfrak{p}\in\mathcal{P}_\Sigma^{reg}$ guarantees that $\mathcal{X}_{\mathfrak{p},\sigma}$ is a complex Banach manifold. This is because the linearisation of $\delbar_{A+i\sigma,\mathfrak{p}}\Phi=0$ is surjective if and only if $\N_\Sigma(d,\mathfrak{p},\tau)$ is regular (use \cite[Corollary 7.2]{Doan_2018} and that $\sigma$ only adds lower order terms). In the classical case this holds without any extra assumptions, see \cite[Lemma 5.2]{Salamon1999}. 

Equation \eqref{eq:symplecticform} defines a closed symplectic connection 2-form on the univeral fibre bundle $\mathcal{X}$. The horizontal subspace at $(A,\Phi,\mathfrak{p},\sigma)$ is the $\Omega$-complement of the vertical tangent space $T_{(A,\Phi)}\mathcal{X}_{\mathfrak{p},\sigma}$. Salamon derives equations for the horizontal lift of a path $(\mathfrak{p},\sigma)$ in the base, although in our case differentiating $\delbar_{A+i\sigma,\mathfrak{p}}$ produces an extra term involving $\dot{B}\in\Omega^1(\End(E))$. Importantly, the moment map equation is preserved along such lifts. Equations \eqref{paralleltransportequations} and \eqref{paralleltransportpsi} are recovered by the substitutions
$$A'(t)=A(t)-i\sigma'(t),\hspace{4mm}\sigma'(t)=\int_0^t\sigma(s)ds.$$
This identifies $\mathcal{X}_\mathfrak{p}$ and $\mathcal{X}_{\mathfrak{p},\sigma}$ as symplectic manifolds, so we obtain a family of $\G_L$-equivariant symplectomorphisms $P_t:\mathcal{X}_{\mathfrak{p}_0}\ra \mathcal{X}_{\mathfrak{p}_t}$. Since the moment map equation is preserved, $P_t$ descends to a symplectomorphism at the level of moduli spaces, as stated in Theorem \ref{paralleltransporttheorem}.

Now suppose $\mathfrak{p}\in\mathcal{P}_\varphi^{reg}$ and $(\sigma,\tau)\in\mathcal{V}_f^+$. Then Theorem \ref{paralleltransporttheorem} equips the symplectic fibre bundle $\N_f(d,\mathfrak{p},\tau)\ra S^1$ with a flat symplectic connection. 

\begin{defn}
    Define the symplectomorphism
    \begin{equation}\label{monodromymap}
        \Upsilon_{d,\mathfrak{p},\eta}:=P_1^{-1}\circ F: \N_\Sigma(d,\mathfrak{p}_0,\tau_0)\ra\N_\Sigma(d,\mathfrak{p}_0,\tau_0),
    \end{equation}
    which is the monodromy map of the flat symplectic connection on $\N_f$. 
\end{defn}

It is a standard fact that evaluation at $t=0$ induces a bijective correspondence between parallel sections of $\N_f(d,\mathfrak{p},\tau)$ and fixed points of $\Upsilon_{d,\mathfrak{p},\eta}$, which can be further partitioned as follows.

\begin{defn}\label{fixedpointsubset}
    Let $\mathfrak{p}$ and $\eta$ be as above. Define the subset $\Fix_{\tilde{f}}(\Upsilon_{d,\mathfrak{p},\eta})\subset\Fix(\Upsilon_{d,\mathfrak{p},\eta})$ to consist of those fixed points for which the corresponding section of $\N_f(d,\mathfrak{p},\tau)$ lifts to a section of $\mathcal{Z}_{\tilde{f}}(d,\mathfrak{p},\tau)$.
\end{defn}

These fixed point subsets can then be identified with the adiabatic limit moduli spaces.

\begin{lemma}\label{correspondences}
    Let $\tilde{f}:L\ra L$ and $\varphi:E\ra E$ be unitary lifts of $f$. Let $\mathfrak{p}\in\mathcal{P}^{reg}_{\varphi}$ and $\eta\in\mathcal{V}_f^+$. Then there is a bijective correspondence
    $$\M_0(\mathfrak{s}_{d,\tilde{f}},\mathfrak{p},\eta)\simeq \Fix_{\tilde{f}}(\Upsilon_{d,\mathfrak{p},\eta})$$
    given by $[A,\Phi,0,b,\Psi]\mapsto [A(0),\Phi(0)]$.
\end{lemma}

\begin{proof}
    The proof goes via the correspondence between fixed points and parallel sections mentioned above. Suppose $(A,\Phi,0,b,\Psi)\in\A_0(\mathfrak{s}_{d,\tilde{f}},\mathfrak{p},\eta)$ and choose a gauge in which $b=0$ (by solving the equation $b-g^{-1}\dot{g}=0$). Lemma \ref{hequalszero} shows that the adiabatic limit equation \eqref{hodgeparalleltransport} is equivalent to \eqref{paralleltransportequations}. The second order equation \eqref{paralleltransportpsi} is recovered by applying $S_{A,\Phi}$, so $[A,\Phi]$ is a parallel section of $\N_f(d,\mathfrak{p},\tau)$. The periodicity condition on the lift $(A,\Phi)$ shows that it is a section of $\mathcal{Z}_{\tilde{f}}(d,\mathfrak{p},\tau)$. Conversely, uniqueness of $\Psi$ shows that it satisfies the required periodicity condition.
\end{proof}

\begin{remark}
    Two lifts $\tilde{f},\tilde{f}'$ of $f$ to $L\ra\Sigma$ determine the same subsets $\Fix_{\tilde{f}}(\Upsilon)=\Fix_{\tilde{f}'}(\Upsilon)$ if and only if the spin$^c$ structures $\mathfrak{s}_{d,\tilde{f}}$ and $\mathfrak{s}_{d,\tilde{f}'}$ are isomorphic. The proof of this is the same as that of \cite[Lemma 7.1]{Salamon1999}. Refer also to Lemma \ref{spincaction}.
\end{remark}

\subsection{Equivalence of deformations}

In addition to the correspondence of Lemma \ref{correspondences}, we also need to compare the deformations of these objects. We begin by understanding deformations of parallel sections. For $\mathfrak{p}\in\mathcal{P}_\varphi^{reg}$ and $\eta=(\sigma,\tau)\in\mathcal{V}_f^+$, a section $(A,\Phi)$ lies in the image of the map
$$\A_0(\mathfrak{s}_{d,\tilde{f}},\mathfrak{p},\eta)\ra \Gamma(\mathcal{Z}_{\tilde{f}}(d,\mathfrak{p},\tau))$$
if and only if
\begin{equation}\label{adiabaticlimitequations}
    \widetilde{\mathcal{F}}(A,\Phi):=\pi_{A,\Phi}\begin{pmatrix}
        \star(\dot{A}-i\sigma)\\
        i\dot{\Phi}
    \end{pmatrix} = 0.
\end{equation}
Here, $\pi_{A,\Phi}$ is the projection to $H_{A,\Phi}^1=\ker(G_{A,\Phi}^*\oplus S_{A,\Phi})$ for each $t$. The Hodge decomposition is unique, so \eqref{adiabaticlimitequations} is equivalent to equation \eqref{hodgeparalleltransport}. For a section $X$ of $\mathcal{Z}_{\tilde{f}}$, let $\mathcal{H}_X\ra S^1$ be the vector bundle with fibre $H_{X(t)}^1$ over $t\in S^1$. Then $\widetilde{\mathcal{F}}$ is an equivariant section of the $\G_f$-equivariant vector bundle with fibre $\Gamma(\mathcal{H}_X)$ over $X\in\Gamma(\mathcal{Z}_{\tilde{f}})$. This descends to a section $\mathcal{F}$ of the bundle with fibre $\Gamma(s^*\mathcal{H})$ over $s\in\Gamma(\N_f)$, where $\mathcal{H}$ is the vertical tangent bundle of $\N_f\ra S^1$. A section of $\N_f$ is parallel if and only if $\mathcal{F}(s)=0$. For each lift $X\in\Gamma(\mathcal{Z}_{\tilde{f}})$ of $s=[X]\in\Gamma(\N_f)$, there is an isomorphism $\mathcal{H}_X\simeq s^*\mathcal{H}$. The linearisation of $\mathcal{F}$ can then be viewed as a map $\Gamma(\mathcal{H}_X)\ra\Gamma(\mathcal{H}_X)$, as is done in the following lemma. 

\begin{lemma}
    Suppose $\mathfrak{p}\in\mathcal{P}_\varphi^{reg}$ and $\eta=(\sigma,\tau)\in\mathcal{V}_f^+$. Then the linearisation of $\mathcal{F}$ at $[X]\in\mathcal{F}^{-1}(0)$ is the operator
    \begin{equation}\label{adiabaticlinearisation}
        \D_0(X) := \pi_{X}N_{Y}|_{\Gamma(\mathcal{H}_{X})} : \Gamma(\mathcal{H}_{X})\ra\Gamma(\mathcal{H}_{X}).
    \end{equation}
    Here, $Y$ is the solution of equation \eqref{hodgeparalleltransport} uniquely determined by $X$.
\end{lemma}

\begin{proof}
    Let $X=(A,\Phi)$, $Y=(b,\Psi)$, and $\Xi=(X,0,Y)$. Suppose $(a,\phi)\in\Gamma(\mathcal{H}_{X})$ and choose $\Xi_s=(A_s,\Phi_s,0,b_s,\Psi_s)$ solving equation \eqref{hodgeparalleltransport} for all $s$ and satisfying
    $$\Xi_0=\Xi,\hspace{3mm}\del_sA_s=a,\hspace{3mm}\del_s\Phi_s=\phi,$$
    Then up to terms of order $s^2$, we have
    \begin{align*}
        0 &= S_{A_s,\Phi_s}^*\begin{pmatrix}
            b_s\\
            \Psi_s
        \end{pmatrix} + \begin{pmatrix}
            \star(\dot{A}_s - i\sigma)\\
            i\dot{\Phi}_s
        \end{pmatrix}\\
        &= S_{A,\Phi}^*\begin{pmatrix}
            b\\
            \Psi
        \end{pmatrix} + \begin{pmatrix}
            \star(\dot{A} + s\dot{a} - i\sigma) - i\Im\langle\Psi_s,s\phi\rangle\\
            i\dot{\Phi} + is\dot{\phi} + ib_s s\phi - \langle\Psi_s,sa^{0,1}\rangle
        \end{pmatrix} + O(s^2).
    \end{align*}
    This can be used to compute the derivative of the first equation of \eqref{adiabaticlimitequations}:
    \begin{align*}
        \D_0(A,\Phi)\begin{pmatrix}
            a\\
            \phi
        \end{pmatrix} &= \pi_{A,\Phi}\frac{d}{ds}\bigg|_{s=0}\begin{pmatrix}
            \star(\dot{A} + s\dot{a} - i\sigma) - i\Im\langle\Psi_s,s\phi\rangle\\
            i\dot{\Phi} + is\dot{\phi} + ib_s s\phi - \langle\Psi_s,sa^{0,1}\rangle
        \end{pmatrix}\\
        &= \pi_{A,\Phi}\begin{pmatrix}
            \star\dot{a} - i\Im\langle\Psi,\phi\rangle\\
            i\nabla_t\phi - \langle\Psi,a^{0,1}\rangle
        \end{pmatrix}\\
        &= \pi_{A,\Phi}N_{b,\Psi}\begin{pmatrix}
            a\\
            \phi
        \end{pmatrix}.
    \end{align*}
\end{proof}

Now consider the operator
\begin{equation}
    \D_0(\Xi) = \begin{pmatrix}
        N_{Y} & G_{X} & S_{X}^*\\
        G_{X}^* & 0 & L_Y^*\\
        S_{X} & 0 & 0
    \end{pmatrix}:W_f\ra W_f.
\end{equation}
This is the linearisation of the adiabatic limit equations \eqref{eq:epsilonzeroSW1} at $\Xi=(X,0,Y)\in\A_0(\mathfrak{s}_{d,\tilde{f}},\mathfrak{p},\eta)$, modulo gauge. The tangent space at $[\Xi]$ of $\M_0(\mathfrak{s}_{d,\tilde{f}},\mathfrak{p},\eta)$ is $\ker\D_0(\Xi)$. The tangent space at $[X]$ to the image of \eqref{downstairsadiabaticcorrespondence} (the space of parallel sections) is $\ker\D_0(X)$.

\begin{prop}\label{kernelidentification}
    Suppose $\mathfrak{p}\in\mathcal{P}_\varphi^{reg}$ and $\eta\in\mathcal{V}_f^+$. Then for all $\Xi=(X,0,Y)\in\mathcal{A}_0(\mathfrak{s}_{d,\tilde{f}},\mathfrak{p},\eta)$, the projection $\pi_X(x,v,y)=\pi_X(x)$ induces an isomorphism
    $$\pi_X:\ker\D_0(\Xi)\ra\ker\D_0(X),$$
    which is the derivative at $[\Xi]$ of the map \eqref{downstairsadiabaticcorrespondence}. 
\end{prop}

\begin{proof}
    To see that $\pi_{X}:\ker\D_0(\Xi)\ra\ker\D_0(X)$ is well-defined, suppose $\xi=(x,v,y)\in\ker\D_0(\Xi)$. We omit subscripts from the notation, writing $\pi$ instead of $\pi_X$, $G$ instead of $G_X$, and so on. First note that
    $$0 = G^*Nx + G^*Gv = -L^*Sx + G^*Gv = G^*Gv$$
    by Lemma \ref{operatoridentities}, which implies $v=0$ because $G^*G$ is invertible for all $t$ (see also Lemma \ref{hequalszero}). Since $Sx=0$ there is $w$ such that $x=x_0+Gw$, where $x_0=\pi(x)$. By Lemma \ref{operatoridentities}, we have
    $$0 = Nx + S^*y = Nx_0 - S^*(Lw + y).$$
    Applying $\pi$ to this shows that $x_0\in\ker\D_0(X)$. To see that $\pi$ is an isomorphism, let $x_0\in\ker\D_0(X)$. Then $x = x_0 + Gw$ satisfies $Sx=0$, so we need to solve the equations
    \begin{align*}
        Nx_0 + S^*(Lw + y) &= 0,\\
        G^*Gw + L^*y &= 0,
    \end{align*}
    for unique $w,y$ (as before, $v=0$). Since $\pi(Nx_0)=0$, there is a solution $y'$ to $Nx_0 + S^*y'=0$ for each $t$. This solution is unique because $S$ is surjective for all $t$ (since $\mathfrak{p}\in\mathcal{P}_\varphi^{reg}$). It remains to solve the equation
    $$G^*Gw + L^*(y' + Lw) = 0.$$
    A short computation shows that $G^*G + L^*L = \Delta_{\Sigma_f} + |\Phi|^2 + |\Psi|^2$, which is invertible for all $t$. Then there is a unique solution $w$ to the above equation and we define $y = y' + Lw$. 
\end{proof}

\begin{remark}\label{infinitecoker}
    We will tend to continue working with $\D_0(X)$ rather than $\D_0(\Xi)$. The reason is that $\D_0(X)$ is self-adjoint and elliptic, while $\D_0(\Xi)$ is neither (some derivatives disappeared in the adiabatic limit). Indeed, the proof of Proposition \ref{kernelidentification} fails for the adjoint operator $\D_0(\Xi)^*$. This has to do with the non-vanishing remainder term $R_{\Xi}$ in the identity \eqref{identity2}. 
\end{remark}

\begin{lemma}\label{fixedpointsande=0}
    Let $\tilde{f}:L\ra L$ and $\varphi:E\ra E$ be unitary lifts of $f$. Let $\mathfrak{p}\in\mathcal{P}^{reg}_{\varphi}$ and $\eta\in\mathcal{V}_f^+$. Then non-degenerate solutions to the adiabatic limit equations correspond to non-degenerate fixed points via the isomorphism of Lemma \ref{correspondences}.
\end{lemma}

\begin{proof}
    If $s_\e$ is a smooth family of parallel sections passing through $s_0$, the derivative $\delta s_0=\del_\e s_\e|_{\e=0}$ is a section of $s^*\mathcal{H}\ra S^1$, where $\mathcal{H}\ra\N_f$ is the vertical tangent bundle. The fixed points $s_\e(0)$ solve $\Upsilon s_\e(0)=s_\e(0)$, which can be differentiated to obtain $(\delta s_0)(0)\in\ker(d_{s_0(0)}\Upsilon - 1)$. This shows that there is an isomorphism $\ker\D_0(X)\simeq \ker(d_{[X]}\Upsilon_{d,\mathfrak{p},\eta} - 1)$. Lemma \ref{kernelidentification} finishes the proof. 
\end{proof}

\subsection{Transversality in the adiabatic limit}
It remains to prove that for a generic perturbation, every fixed point of the monodromy map $\Upsilon$ is non-degenerate. We choose to work in the setting of parallel sections, since the standard Sard--Smale theorem cannot be applied to the adiabatic limit equations themselves. This is because there is an infinite-dimensional cokernel (see Remark \ref{infinitecoker}), which is not the case for deformations of parallel sections. We begin by proving the following technical lemma. 

\begin{lemma}\label{preliminarydensity}
    Fix $J\in\mathcal{J}_f(\Sigma,\omega)$ and define $\mathcal{U}_J\subset\A_\varphi^{reg}(J)\times\mathcal{V}_f^+$ as the set of $(B,\eta)$ such that
    $$\frac{1}{2}(\del_{A,\mathfrak{p}}\Phi)\circ\dot{J} + (\sigma+\dot{B})^{0,1}\Phi\neq 0$$
    for all $t$ and all $[A,\Phi]\in\Gamma(\mathcal{N}_{f}(d,\mathfrak{p},\tau))$. Here, $\mathfrak{p}=(J,B)$ and $\eta=(\sigma,\tau)$. Then $\mathcal{U}_J$ is an open dense subset of $\A_\varphi^{reg}(J)\times\mathcal{V}_f^+$ in the $C^\infty$ topology. 
\end{lemma}

\begin{proof}
    For $(B,\eta)\in\A_\varphi^{reg}(J)\times\mathcal{V}_f^+$, $t\in[0,1]$, and $(A,\Phi)\in\mathcal{Z}_{\Sigma}(d,\mathfrak{p}_t,\tau_t)$, define the quantity
    \begin{equation}\label{righthandside}
        \theta_t^{B,\eta}(A,\Phi) := \frac{1}{2}(\del_{A,\mathfrak{p}_t}\Phi)\circ\dot{J}_t + (\sigma_t+\dot{B}_t)^{0,1}\Phi,
    \end{equation}
    where $\mathfrak{p}=(J,B)$ and $\eta=(\sigma,\tau)$. To see that $\mathcal{U}_J$ is open, define
    $$\Theta(B,\eta) = \inf_{(t,[A,\Phi])\in\N_f(d,\mathfrak{p},\tau)}\big\|\theta_t^{B,\eta}(A,\Phi)\big\|_{L^2(\Sigma,J_t)}.$$
    Since $B\in\A_\varphi^{reg}(J)$, $\N_f$ is compact so this function is continuous and the infimum is achieved. Then $\mathcal{U}_J=\Theta^{-1}(0,\infty)$ is open. To show that $\mathcal{U}_J$ is dense, suppose $(B,\eta)\in\A_\varphi^{reg}(J)\times\mathcal{V}_f^+$. Note that if there is a smooth $\zeta\in\Omega^{0,1}(\Sigma)$ such that
    \begin{equation}\label{badperturbation}
        \theta_t^{B,\sigma+\zeta,\tau}(A,\Phi) = \theta_t^{B,\eta}(A,\Phi) + \zeta\Phi=0,
    \end{equation}
    it is determined pointwise by the equation
    $$\zeta = -\frac{\langle\theta_t^{B,\eta}(A,\Phi),\Phi\rangle}{|\Phi|^2}$$
    on the set $\{\Phi\neq 0\}$. Since $\Phi$ is a holomorphic section, this set is all but a finite number of points, so $\zeta$ is determined everywhere by continuity. Then for each $t$ there is a map
    $$\N_\Sigma(d,\mathfrak{p}_t,\tau_t) \dashrightarrow \Omega^{0,1}(\Sigma),\hspace{5mm}[A,\Phi] \mapsto -\frac{\langle\theta_t^{B,\eta}(A,\Phi),\Phi\rangle}{|\Phi|^2},$$
    defined for any $[A,\Phi]$ for which there is $\zeta$ solving equation \eqref{badperturbation}. Let $\mathcal{K}_t$ be the complement of the image of this map, which is dense in $\Omega^{0,1}(\Sigma)$ because $\N_\Sigma$ is finite-dimensional. Since $(\sigma+\zeta,\tau)$ must solve the equation $\dot{\tau} + \star d(\sigma + \zeta)=0$, we require $d\zeta=0$. The intersection $\mathcal{K}_{t,cl} = \mathcal{K}_t\cap\Omega^{0,1}_{cl}(\Sigma)$ is still a dense subset of the closed $0,1$-forms $\Omega^{0,1}_{cl}(\Sigma)$. Now define
    $$\mathcal{S} = \{\zeta\in\Omega^{0,1}_f(\Sigma) : \zeta(t)\in\Omega^{0,1}_{cl}(\Sigma)\setminus\mathcal{K}_{t,cl} \textnormal{ for all }t\}\subset\Omega^{0,1}_{f,cl}(\Sigma),$$
    which is a dense subset of $\Omega^{0,1}_f(\Sigma)$ because any path $\zeta$ can be perturbed to avoid the family of finite-dimensional subsets $\mathcal{K}_{t,cl}$. Then for $\zeta\in\mathcal{S}$ arbitrarily close to $0$ (in the $C^\infty$ topology), the point $(B,\sigma+\zeta,\tau)\in\mathcal{U}_J$ is arbitrarily close to $(B,\sigma,\tau)$. This shows that $\mathcal{U}_J$ is dense in $\A_\varphi(E)\times\mathcal{V}_f^+$.

\end{proof}

\begin{prop}\label{transversality}
    Fix $J\in\mathcal{J}_f(\Sigma,\omega)$ and let $\mathfrak{s}_{d,\tilde{f}}$ be a spin$^c$ structure on $\Sigma_f$. Then there exists a dense subset
    $$\mathcal{U}_{\varphi}^{reg}(J)\subset\A_\varphi(E)\times\mathcal{V}_f^+$$
    such that for all $(B,\eta)\in \mathcal{U}_\varphi^{reg}(J)$, the following statements hold for $\mathfrak{p}=(J,B)$ and $\eta=(\sigma,\tau)$:
    \begin{enumerate}
        \item $\mathfrak{p}\in\mathcal{P}_\varphi^{reg}$, so that $\N_\Sigma(d,\mathfrak{p}_t,\tau_t)$ is regular for all $t$.
        \item Every fixed point in $\Fix_{\tilde{f}}(\Upsilon_{d,\mathfrak{p},\eta})$ is non-degenerate.
    \end{enumerate}
    Moreover, $\mathcal{U}_{\varphi}^{reg}:=\{(J,B,\eta):(B,\eta)\in\mathcal{U}_\varphi^{reg}(J)\}$ is a dense subset of $\mathcal{P}_{\varphi}\times\mathcal{V}_f^+$. 
\end{prop}

\begin{proof}
    With $J$ fixed and $k\geq 1$, define the spaces
    $$\widetilde{\bm{\mathcal{B}}}_k = \{(B,\eta,X) : (B,\eta)\in\mathcal{U}_{J,k},\hspace{1mm} X\in\Gamma_{k+1}(\mathcal{Z}_{\tilde{f}}(d,\mathfrak{p},\tau)),\hspace{1mm}\mathfrak{p}=(J,B)\},$$
    $$\widetilde{\bm{\mathcal{H}}}_k = \{(B,\eta,X,x) : (B,\eta,X)\in\bm{\mathcal{B}}_k,\hspace{1mm}x\in\Gamma_{k}(\mathcal{H}_X)\}.$$
    Here, additional subscripts $k$ indicate the $L_k^2$ completion of a space, and we are defining
    $$\mathcal{U}_{J,k}=\mathcal{U}_J\cap\A_{\varphi,k}^{reg}(J)\times\mathcal{V}_{f,k}^+,$$
    which remains an open dense subset by Lemma \ref{preliminarydensity}. The projection $\widetilde{\bm{\mathcal{H}}}_k\ra\widetilde{\bm{\mathcal{B}}}_k$ is a $\G_{f,k+2}$-equivariant Hilbert bundle. For $X=(A,\Phi)$, define a smooth equivariant section of this bundle by
    $$\widetilde{\bm{\mathcal{F}}}(B,\sigma,\tau,X) = \pi_{X}\begin{pmatrix}
        \star(\dot{A}-i\sigma)\\
        i\dot{\Phi}
    \end{pmatrix}.$$
    This is a version of the section $\widetilde{\mathcal{F}}$ of equation \eqref{adiabaticlimitequations} where the parameter $B$ and perturbation are allowed to vary. As in the fixed parameter case, it descends to a section $\bm{\mathcal{F}}$ of the bundle $\bm{\mathcal{H}}_k\ra\bm{\mathcal{B}}_k$, where
    $$\bm{\mathcal{B}}_k = \{(B,\eta,[X]) : (B,\eta)\in\mathcal{U}_{J,k},\hspace{1mm} [X]\in\Gamma_{k+1}(\mathcal{N}_{f}(d,\mathfrak{p},\tau)),\hspace{1mm}\mathfrak{p}=(J,B)\},$$
    $$\bm{\mathcal{H}}_k = \{(B,\eta,[X],x) : (B,\eta,X)\in\bm{\mathcal{B}}_k,\hspace{1mm}x\in\Gamma_{k}([X]^*\mathcal{H})\},$$
    and $\mathcal{H}\ra \N_f$ is the vertical tangent bundle. The tangent space to $\bm{\mathcal{B}}_k$ at $(B,\sigma,\tau,[X])$ is the space of all tuples
    $$(\beta,\zeta,\delta,x)\in\Omega^{1}_{\varphi,k}(\End(E))\oplus\mathcal{V}_{f,k}^+\oplus\Omega^1_{f,k}(i\R)\oplus\Gamma_{\varphi,k}(E\otimes L)$$
    such that
    $$S_Xx + (i\delta,\beta^{0,1}\Phi)=0,\hspace{3mm}G_X^*x=0.$$
    Define the linear map $\nu_X:\Omega^0_{f,k}(\Sigma)\oplus \Omega_{\varphi,k}^1(\End(E))\ra\Omega^1_{f,k}(i\R)\oplus\Gamma_{\varphi,k}(E\otimes L)$ as the unique solution to the equation
    $$S_X\nu_X\begin{pmatrix}
        \delta\\
        \beta
    \end{pmatrix}  + \begin{pmatrix}
        i\delta\\
        \beta^{0,1}\Phi
    \end{pmatrix}=0$$
    which is $L^2$ orthogonal to $\ker(S_{X})$. When $B,\tau$ are varied in the direction $(\beta,\delta)$, this is the normal component of the infinitessimal deformation of $\Gamma(\N_f(d,\mathfrak{p},\tau))$ at $[X]$. Subtracting this produces the isomorphism
    $$T_{(B,\sigma,\tau,[X])}\bm{\mathcal{B}}_k\simeq\Omega^{1}_{\varphi,k}(\End(E))\oplus\mathcal{V}_{f,k}^+\oplus\Omega^1_{f,k}(i\R)\oplus\Gamma_{k}(\mathcal{H}_X)$$
    $$(\beta,\zeta,\delta,x)\mapsto (\beta,\zeta,\delta,x - \nu_X(\delta,\beta))$$
    With this identification made, the linearisation of $\bm{\mathcal{F}}$ at a zero $(B,\sigma,\tau,[X])$ is
    $$d\bm{\mathcal{F}}_{(B,\sigma,\tau,[X])}:\Omega^{1}_{\varphi,k}(\End(E))\oplus\mathcal{V}_{f,k}^+\oplus\Omega^1_{f,k}(i\R)\oplus\Gamma_k(\mathcal{H}_X)\ra\Gamma_{k}(\mathcal{H}_X),$$
    $$(\beta,\zeta,\delta,x) \mapsto \D_0(X)x + \pi_X\begin{pmatrix}
        -i\star\zeta\\
        \langle\Psi,\beta^{0,1}\rangle
    \end{pmatrix}.$$
    Now suppose $x'=(a',\phi')\in\Gamma_k(\mathcal{H}_X)$ is $L^2$ orthogonal to the image of $d\bm{\mathcal{F}}$. This means that for all $(\beta,\zeta,\delta,a,\phi)$ we have
    \begin{equation}\label{orthogonaltoimage}
        \int_0^1\int_\Sigma\bigg\langle N_{b,\Psi}\begin{pmatrix}
            a\\
            \phi
        \end{pmatrix} + \begin{pmatrix}
            -i\star\zeta\\
            \langle\Psi,\beta^{0,1}\rangle
        \end{pmatrix}, \begin{pmatrix}
            a'\\
            \phi'
        \end{pmatrix}\bigg\rangle\omega\wedge dt=0,
    \end{equation}
    $$G_X^*x'=0,\hspace{3mm}S_Xx'=0.$$

    Suppose for contradiction that there is $t\in[0,1]$ and $z_0\in\Sigma$ such that $\phi'(t,z_0)\neq 0$. In the gauge such that $b=0$, $\Psi(t)$ solves equation \eqref{paralleltransportpsi}. Since $(B,\eta)\in\mathcal{U}_{J,k}$, the right hand side of \eqref{paralleltransportpsi} is non-zero, so $\Psi(t)$ is not identically zero for all $t$. Moreover, $\{\Psi(t)=0\}\subset\Sigma$ is dense by the unique continuation principle for nonlinear elliptic PDE (for example, see \cite{Aronszajn1957}). Then by the continuity of $\phi'$, we can find $z$ arbitrarily close to $z_0$ such that $\Psi(t,z)\neq 0$ and $\phi'(t,z)\neq 0$. Now vary only $\beta$ in equation \eqref{orthogonaltoimage} to get
    \begin{equation}\label{varyingbeta}
        \int_{0}^{1}\int_{\Sigma}\big\langle\langle\Psi,\beta^{0,1}\rangle,\phi'\big\rangle\omega\wedge dt = 0.
    \end{equation}
    On the other hand, let $U$ be a small neighbourhood of $(t,z)$ in $\Sigma_f$ where both $\Psi$ and $\phi'$ are non-zero. Let $\rho:\Sigma_f\ra[0,1]$ be a smooth bump function supported in a slightly smaller neighbourhood $U'$. Then for $\beta = \rho\langle\Psi,\phi'\rangle$, the left hand side of \eqref{varyingbeta} is
    $$\int_{U'}\rho|\Psi|^2|\phi'|^2\omega\wedge dt > 0,$$
    which is a contradiction. Now that $\phi'=0$, the equation $S_Xx'=0$ reduces to $(a')^{0,1}\Phi = 0$. This implies $a'=0$ because $\delbar_{A,\mathfrak{p}}\Phi=0$ means that $\Phi$ only vanishes at discrete points. This proves that $\bm{\mathcal{F}}$ is a submersion, so $\bm{\mathcal{F}}^{-1}(0)$ is a smooth submanifold of $\bm{\mathcal{B}}_k$. The derivative at $X$ of the projection $\bm{\mathcal{B}}_k\ra \mathcal{U}_{J,k}$ is Fredholm, and there is a natural identification of its cokernel with $\coker(\D_0(X))$. By the Sard--Smale theorem, for all $k$ the set of regular values is residual in $\mathcal{U}_{J,k}$ and thus dense in $\A_{\varphi,k}^{reg}(J)\times\mathcal{V}_{f,k}^+$ by Lemma \ref{preliminarydensity}. The argument used in \cite[Proposition 2.19]{Doan_2018} shows that this statement also holds in the $C^\infty$ topology, which defines the dense subset $\mathcal{U}_\varphi^{reg}(J)$. The preimage of a regular value $(B,\eta)\in\mathcal{U}_\varphi^{reg}(J)$ under the projection is the space of parallel sections of $\N_f(d,\mathfrak{p},\tau)$, which is zero-dimensional and regular because the obstructions vanish. The multi-vortex moduli spaces $\N_\Sigma(d,\mathfrak{p}_t,\tau_t)$ are regular by construction, and all fixed points of $\Upsilon_{d,\mathfrak{p},\eta}$ are non-degenerate by Proposition \ref{fixedpointsande=0}. The density of $\mathcal{U}_\varphi^{reg}(J)$ for a fixed $J$ also implies the density of $\mathcal{U}_\varphi^{reg}$.
\end{proof}

\begin{remark}
    Perturbing the family of parameters $\mathfrak{p}$ is certainly necessary. For example, $\Upsilon$ fixes every point in $\N_\Sigma$ for constant data pulled back from $\Sigma$ (this is the content of Theorem \ref{S1invariantstuff}). However, this is inconsequential because a perturbation $\mathfrak{p}'$ of $\mathfrak{p}$ can be made so small that $\mathfrak{p}'$ still lies in the same chamber as $\mathfrak{p}$. 
\end{remark}

%% file: chapters/the_adiabatic_limit.tex
\section{Approximating the adiabatic limit}\label{s4}

We have seen that fixed points of the monodromy map $\Upsilon$ correspond to solutions of the adiabatic limit equation \eqref{eq:epsilonzeroSW1}. Moreover, these fixed points become non-degenerate after an arbitrarily small perturbation. In this section we prove Theorem \ref{intromaintheorem}, which perturbs a fixed point to a unique multi-monopole for a metric sufficiently small in the fibre direction. This result is restated below with more precise notation.

\begin{theorem}\label{maintheorem}
    Let $d\in\Z$ and choose lifts $\tilde{f}$ and $\varphi$ of $f$ as in the previous sections. Then there is a dense subset $\U^{reg}_{\varphi}\subset \mathcal{P}_{\varphi}^{reg}\times\mathcal{V}_f^+$ such that for $(\mathfrak{p},\eta)\in\U^{reg}_{\varphi}$, the fixed points of $\Upsilon_{d,\mathfrak{p},\eta}$ are all non-degenerate. Moreover, for $(\mathfrak{p},\eta)\in\U_\varphi^{reg}$ there is $\e_0>0$ such that for all $\e\in(0,\e_0)$, the moduli space $\M_{\Sigma_f}(\mathfrak{s}_{d,\tilde{f}},g_\e,B_\varphi,\eta)$ is regular and there is an injective map
    \begin{equation}\label{adiabaticlimitiso}
        \Fix_{\tilde{f}}(\Upsilon_{d,\mathfrak{p},\eta})\hookrightarrow\M_{\Sigma_f}(\mathfrak{s}_{d,\tilde{f}},g_\e,B_\varphi,\eta).
    \end{equation}
\end{theorem}

Having already constructed the set $\mathcal{U}^{reg}_{\varphi}$, the proof will be completed by Theorem \ref{ezerotoesmall}. The main analytic difficulty is in proving weighted $L^p$ elliptic estimates for $\D_\e$. A calculation analogous to \cite[Lemma 4.2]{Dostoglou-Salamon1994} suffices for $p=2$, but significantly more care is required for $p>2$. In that case, we adapt the techniques used in the appendix of \cite{Salamon2000QuantumPF} to our situation. 

\subsection{Sobolev norms and moduli spaces}

We begin by defining spaces of $W^{1,p}$ solutions, along with weighted Sobolev norms which are adapted to the geometry of our situation. 

\begin{defn}
    Let $W^{1,p}_f$ be the completion of $W_{f}$ with respect to the $W^{1,p}$ Sobolev norm, and define $\mathcal{C}_f^{1,p} = \Xi_0 + W_f^{1,p}$ for some smooth reference configuration $\Xi_0\in\mathcal{C}_f$. Let $\G_f^{2,p}$ be the space of $W^{2,p}$ gauge transformations on $\Sigma_f$. For $(\mathfrak{p},\eta)\in\mathcal{U}_\varphi^{reg}$ and $\e>0$, define the solution spaces
    $$\A_{0}^{1,p}(\mathfrak{s}_{d,\tilde{f}},\mathfrak{p},\eta) = \{\Xi\in \mathcal{C}^{1,p}_{f} : \eqref{eq:epsilonzeroSW1}\},$$
    $$\A_{\e}^{1,p}(\mathfrak{s}_{d,\tilde{f}},\mathfrak{p},\eta) = \{\Xi\in \mathcal{C}^{1,p}_{f} : \eqref{eq:epsilonSW}\},$$
    and the moduli spaces
    $$\M_0(\mathfrak{s}_{d,\tilde{f}},\mathfrak{p},\eta) = \A^{1,p}_0(\mathfrak{s}_{d,\tilde{f}},\mathfrak{p},\eta)/\GG_f^{2,p},$$
    $$\M_{\e}(\mathfrak{s}_{d,\tilde{f}},\mathfrak{p},\eta) = \A^{1,p}_{\e}(\mathfrak{s}_{d,\tilde{f}},\mathfrak{p},\eta)/\GG_f^{2,p}.$$
\end{defn}

\begin{defn}
    Fix $\Xi = (X,0,Y)\in\A_0^{1,p}(\mathfrak{s}_{d,\tilde{f}},\mathfrak{p},\eta)$ and let $\xi = (x,v,y)\in W_f^{1,p}$. Then for $\e>0$, define the following $\Xi$-dependent norms:
    \begin{equation*}
        \|\xi\|_{0,p,\e}^p = \int_0^1\Big(\|x\|_{L^p(\Sigma_t)}^p + \e^p\|v\|_{L^p(\Sigma_t)}^p + \e^p\|y\|_{L^p(\Sigma_t)}^p\Big)dt,
    \end{equation*}

    \begin{equation*}
        \|\xi\|_{\infty,\e} = \|x\|_{L^\infty(\Sigma_f)} + \e\|v\|_{L^\infty(\Sigma_f)} + \e\|y\|_{L^\infty(\Sigma_f)},
    \end{equation*}
    
    \begin{align*}
        \|\xi\|_{1,p,\e}^p &= \int_{0}^{1}\Big(\|x\|_{L^p(\Sigma_t)}^p + \|G^*_{X}x\|_{L^p(\Sigma_t)}^p + \|S_{X}x\|_{L^p(\Sigma_t)}^p + \e^p\|N_{Y}x\|_{L^p(\Sigma_t)}^p\\
        &\hspace{3mm}+ \e^p\|G_Xv\|_{L^p(\Sigma_t)}^p + \e^{2p}\|L_Yv\|_{L^p(\Sigma_t)}^p\\
        &\hspace{3mm}+ \e^p\|S_{X}^*y\|_{L^p(\Sigma_t)}^p + \e^{2p}\|L_Y^*y\|_{L^p(\Sigma_t)}^p + \e^{2p}\|M_Yy\|_{L^p(\Sigma_t)}^p\Big)dt.
    \end{align*}
    Here, the norm $L^p(\Sigma_t)$ is the $L^p$ norm on $\Sigma$ with respect to the metric $g_t=\omega(\cdot,J_t\cdot)$. 
\end{defn}

If $\mathfrak{p}$ and $\eta$ are smooth and $\Xi$ is a $W^{1,p}$ solution to either the adiabatic limit equations or the $\e$-dependent multi-monopole equations, there is a gauge transformation $g\in\G_f^{2,p}$ such that $g^*\Xi$ is smooth. This follows from the same arguments used for the classical Seiberg--Witten moduli space, see for example \cite{MorganSW}. So although we will work with $W^{1,p}$ solutions, modulo gauge we recover the moduli space of smooth solutions. Also note that if $\mathfrak{p}=(J,B)$, we have
$$\M_\e(\mathfrak{s}_{d,\tilde{f}},\mathfrak{p},\eta) = \M_{\Sigma_f}(\mathfrak{s}_{d,\tilde{f}},g_\e,B_\varphi,\eta),$$
where $g_\e = dt^2 + \e^2\omega(\cdot,J_t\cdot)$ is the rescaled metric and $B_\varphi$ is the connection induced by $B$. 

\subsection{Preliminary estimates}

In this section, we prove several technical lemmas in preparation for the proof of the weighted elliptic estimate for $\D_\e$ (Proposition \ref{modifiedellipticestimate}). Our first estimates take place in a single fibre $\Sigma$ of the mapping torus. These are used to obtain local estimates for various components of $\D_{\e}$, which are combined into a global estimate in Proposition \ref{technicallemma}.

\begin{defn}
    Let $\mathfrak{p}=(J,B)\in\mathcal{P}_\Sigma$ and define
    $$W_{\Sigma} = \Omega^1(\Sigma,i\R)\oplus\Gamma(\Sigma,E\otimes L)\oplus\Omega^0(\Sigma,i\R)\oplus\Omega^0(\Sigma,i\R)\oplus\Omega^{0,1}_J(\Sigma,E\otimes L).$$
    Let $X\in\A(\Sigma,L)\times\Gamma(\Sigma,E\otimes L)$ and define the operators
    $$P_{\mathfrak{p},X}: W_{\Sigma}\ra W_{\Sigma},\hspace{4mm}I_J: W_{\Sigma}\ra  W_{\Sigma},$$
    $$P_{\mathfrak{p},X} = \begin{pmatrix}
        0 & G_X & S_X^*\\
        G_X^* & 0 & 0\\
        S_X & 0 & 0
    \end{pmatrix},\hspace{4mm} I_J = \begin{pmatrix}
        \star_J & 0 & 0 & 0 & 0\\
        0 & i & 0 & 0 & 0\\
        0 & 0 & 0 & 1 & 0\\
        0 & 0 & -1 & 0 & 0\\
        0 & 0 & 0 & 0 & -i
    \end{pmatrix}.$$
\end{defn}

From now on, fix a choice of smooth reference parameters 
$$\mathfrak{p}_0=(J_0,B_0)\in\mathcal{P}_\Sigma,\hspace{4mm}X_0=(A_0,\Phi_0)\in\mathcal{A}(\Sigma,L)\times \Gamma(\Sigma,E\otimes L).$$ 

\begin{lemma}\label{fibreellipticestimate}
    Fix $p\geq 2$. Then there exists a constant $c>0$ such that if $\mathfrak{p}=(J,B)\in\mathcal{P}_\Sigma^{reg}$ and $X\in\mathcal{Z}_{\Sigma}(d,\mathfrak{p},\tau)$ satisfy
    \begin{equation}\label{fibreuniformbound}
        \|J\|_{C^1(\Sigma,J_0)} + \|B-B_0\|_{L^\infty(\Sigma,J_0)} + \|X-X_0\|_{L^\infty(\Sigma,J_0)}\leq C,
    \end{equation}
    for some $C>0$, then for every $\xi=(x,v,y)\in W_{\Sigma}$ we have
    $$\|v\|_{W^{1,p}(\Sigma,J)}^p + \|y\|_{W^{1,p}(\Sigma,J)}^p\leq c\|G_Xv+S_X^*y\|_{L^p(\Sigma,J)}^p,$$
    $$\|x\|_{W^{1,p}(\Sigma,J)}^p \leq c\Big(\|G_X^*x\|_{L^p(\Sigma,J)}^p + \|S_Xx\|_{L^p(\Sigma,J)}^p + \|\pi_X(x)\|_{L^p(\Sigma,J)}^p\Big),$$
    $$\|\xi\|_{W^{1,p}(\Sigma,J)}^p \leq c\Big(\|P_{\mathfrak{p},X}\xi\|^p_{L^p(\Sigma,J)} + \|\pi_X\xi\|_{L^p(\Sigma,J)}^p\Big).$$
    Moreover, $c$ only depends on $C$, $p$, $\mathfrak{p}_0$, and $X_0$.
\end{lemma}

\begin{lemma}\label{LpL2exchange}
    Fix $p\geq 2$. Then for every $\delta>0$ sufficiently small, there exists a constant $c_\delta>0$ such that if $\mathfrak{p}=(J,B)\in\mathcal{P}_\Sigma^{reg}$ and $X\in\A(\Sigma,L)\times\Gamma(\Sigma,E\otimes L)$ satisfy \eqref{fibreuniformbound} for some $C>0$, we have
    \begin{equation}
        \|\xi\|_{L^p(\Sigma,J)}^p\leq \delta\|P_{\mathfrak{p},X}\xi\|_{L^p(\Sigma,J)}^p + c_\delta\|\xi\|_{L^2(\Sigma,J)}^p
    \end{equation}
    for every $\xi\in W_{\Sigma}$. Moreover, $c_\delta$ only depends on $\delta$, $C$, $p$, $\mathfrak{p}_0$, and $X_0$.
\end{lemma}

The inequalities in Lemma \ref{fibreellipticestimate} are just the standard elliptic estimates for $G_X^*\oplus S_X$ and its adjoint $G_X+S_X^*$ (which is injective because $\mathfrak{p}$ is regular). The main content of the lemma is that $c$ only has the stated dependencies. The basic idea is to start with the elliptic estimate for a fixed $\mathfrak{p}_1,X_1$, then prove that the estimate still holds for $\mathfrak{p},X$ in a $\delta$-neighbourhood with a uniform constant. Then the Arzela--Ascoli theorem implies that the set of $\mathfrak{p},X$ satisfying the bound \eqref{fibreuniformbound} is totally bounded, so it can be covered by finitely many such $\delta$-neighbourhoods. For a more detailed explanation, refer to Appendix B and Lemma C.1 in \cite{Salamon2000QuantumPF}. The proof of Lemma \ref{LpL2exchange} relies on a version of Morrey's inequality and a weaker version of Lemma \ref{fibreellipticestimate} (see \cite[Lemma B.2]{Salamon2000QuantumPF} for more details). 

\begin{prop}\label{Yzerocase}
    Fix $p\geq 2$. Suppose $\mathfrak{p}=(J,B):\R\ra \mathcal{P}_\Sigma^{reg}$, $\tau:\R\ra \Omega^{0}(\Sigma)$ and $X:\R\ra \A(\Sigma,L)\times\Gamma(\Sigma,E\otimes L)$ be continuously differentiable functions such that $X(t)\in\mathcal{Z}_\Sigma(d,\mathfrak{p}_t,\tau_t)$ and $\bar{\tau}_t = \frac{1}{2\pi}\int_\Sigma\tau_t\omega$ is constant with $d-\bar{\tau}<0$. Also suppose there is $C>0$ such that
    \begin{equation}\label{uniformbound1}
        \sup_{t\in\R}\big(\|J(t)\|_{C^2(\Sigma)} + \|\del_t J\|_{C^1(\Sigma)} + \|B(t)-B_0\|_{C^1(\Sigma)} + \|\del_t B\|_{L^\infty(\Sigma)}\big)<C
    \end{equation}
    \begin{equation}\label{uniformbound2}
        \sup_{t\in\R}\big(\|X(t)-X_0\|_{C^1(\Sigma)} + \|\del_t X\|_{L^\infty(\Sigma)}\big)<C.
    \end{equation}
    Then there exist constants $\e_0>0$ and $c>0$ such that for every $\xi=\xi(s)\in C_0^\infty(\R,W_\Sigma)$ and every $\e\in (0,\e_0)$, we have
    \begin{equation}\label{Yzeroinequality}
        \int_{\R}\Big(\|\del_{s}\xi\|_{L^p(\Sigma_{\e s})}^p + \|P_{\e s}\xi\|_{L^p(\Sigma_{\e s})}^p\Big)ds\leq c\int_{\R}\Big(\|D_{\e s}\xi\|_{L^p(\Sigma_{\e s})}^p + \e^p\|\pi_{\e s}\xi\|_{L^p(\Sigma_{\e s})}^p\Big)ds.
    \end{equation}
    Here, $D_{\e s} := P_{\e s} + I_{\e s}\del_s$ for $P_{\e s} = P_{\mathfrak{p}(\e s),X(\e s)}$ and $I_{\e s} = I_{J(\e s)}$. The space $C_0^\infty(\R,W_\Sigma)$ consists of smooth paths $\xi:\R\ra W_\Sigma$ with compact support. The operator $\pi_{\e s}=\pi_{X(\e s)}$ is $L^2$ orthogonal projection to $H^1_{X(\e s)}$ with respect to the metric $g_{\e s}=\omega(\cdot,J_{\e s}\cdot)$, and $\|\cdot\|_{L^p(\Sigma_{\e s})}$ denotes the $L^p$-norm with respect to $g_{\e s}$. Moreover, $c$ only depends on $p$, $C$, and the reference parameters $\mathfrak{p}_0$ and $X_0$.
\end{prop}

The proof of Proposition \ref{Yzerocase} depends on the following four lemmas. Their proofs are modelled on those of Lemmas C.3, C.4, C.5, and C.6 in Appendix C of \cite{Salamon2000QuantumPF}. For this reason we will not include fully detailed explanations of some inequalities, instead emphasising how the arguments differ for our equations. We will use the letter $c$ for a positive constant that can vary from line to line, depending only on $p$, the reference parameters $\mathfrak{p}_0$,$X_0$, and the upper bound $C$ of \eqref{uniformbound1} and \eqref{uniformbound2}.

\begin{lemma}\label{step1}
    Fix $p\geq 2$. Suppose $\mathfrak{p}$, $\tau$, and $X$ are as in the statement of Proposition \ref{Yzerocase}. Then there exist constants $\e_0>0$ and $c>0$ such that, for every $\xi\in C_0^\infty(\R,W_\Sigma)$ and $\e\in(0,\e_0)$, we have
    \begin{equation}
        \int_{\R}\|\xi\|_{L^2(\Sigma_{\e s})}^pds \leq c\int_{\R}\Big(\|D_{\e s}\xi\|_{L^2(\Sigma_{\e s})}^p + \|\pi_{\e s}\xi\|_{L^2(\Sigma_{\e s})}^p\Big)ds.
    \end{equation}
\end{lemma}

\begin{proof}
    Extend the operators $P,I,$ and $Q:=\diag(-\star_0\star_{J(\e s)},1,1,1,1)$ to maps from the $W^{1,2}$ version of $W_\Sigma$ to the $L^2$ version. Here, $\star_0$ is the Hodge star operator determined by the reference complex structure $J_0$. These operators satisfy the commutator identities
    $$IP+PI=0,\hspace{3mm}QI+I^*Q=0,\hspace{3mm}QP-P^*Q=0,$$
    so the inequality follows from a simpler version of \cite[Proposition A.2]{Salamon2000QuantumPF} with 1-dimensional domain $\R$ instead of $\R^2$, and the uniform bounds \eqref{uniformbound1} and \eqref{uniformbound2}.
\end{proof}

\begin{lemma}\label{step2}
    Fix $p\geq 2$. Suppose that $\tau\in\Omega^0(\Sigma)$ satisfies $d-\bar{\tau}<0$ and the bound \eqref{fibreuniformbound} holds for $\mathfrak{p}=(J,B)\in\mathcal{P}_\Sigma^{reg}$ and $X=(A,\Phi)\in\mathcal{Z}_{\Sigma}(d,\mathfrak{p},\tau)$. Then there exists a constant $c>0$ such that for every $\xi\in C_0^\infty(\R,W_\Sigma)$, we have
    \begin{equation}\label{lemma6.8inequality}
        \int_\R\|\del_s\xi\|_{L^p(\Sigma,J)}^pds\leq c\int_\R\|D\xi\|_{L^p(\Sigma,J)}^pds
    \end{equation}
    for $D = I\del_s + P$ with $I=I_J$ and $P=P_{\mathfrak{p},X}$. 
\end{lemma}

\begin{proof}
    Decompose $\xi=(x,v,y)$ as $\xi=\xi_0+\xi_1$, where $\xi_0=\pi_X\xi=(x_0,0,0)$ and $\xi_1=(x-x_0,v,y)$. Then if $x_0=(a_0,\phi_0)$ we have
    $$D\xi_0=I\del_s\xi_0=\begin{pmatrix}
        \star_J\del_s a_0\\
        i\del_s\phi_0
    \end{pmatrix},$$
    which implies
    \begin{equation}
        \int_{\R}\|\del_s\xi_0\|_{L^p(\Sigma,J)}^p ds\leq c\int_{\R}\|D\xi_0\|_{L^p(\Sigma,J)}^p ds. 
    \end{equation}
    for a constant $c$ depending continuously on $J$. By \eqref{fibreuniformbound}, this can be replaced by a constant depending on $C$ and the reference complex structure $J_0$. To obtain a similar estimate for $\xi_1$, patch together the local elliptic estimates for $D$ over $[k-1,k+1]\times\Sigma$, $k\in\Z$, to get the global estimate
    \begin{equation}\label{cylinderinequality}
        \|\xi_1\|_{W^{1,p}(\R\times\Sigma)}^p \leq c\big(\|D\xi_1\|_{L^p(\R\times\Sigma)}^p + \|\xi_1\|_{L^p(\R\times\Sigma)}^p\big).
    \end{equation}
    Since $\pi_X\xi_1=0$, it follows from Lemmas \ref{LpL2exchange} and \ref{step1} that
    $$\|\xi_1\|_{L^p(\R\times\Sigma)}^p \leq c\int_\R\big(\delta\|D\xi_1\|_{L^p(\Sigma,J)}^p + \delta\|\del_s\xi_1\|_{L^p(\Sigma,J)}^p + c_\delta\|D\xi_1\|_{L^p(\Sigma,J)}^p\big)ds.$$
    Then for $\delta$ small enough, the $\del_s\xi_1$ term can be absorbed into the left hand side of \eqref{cylinderinequality}, proving \eqref{lemma6.8inequality} for $\xi_1$. Now note that there is a constant $c>0$, only depending on $p,\mathfrak{p}_0,X_0,$ and $C$, such that 
    \begin{equation}\label{projectionLp}
        \|\pi_X\xi\|_{L^p(\Sigma,J)}\leq c\|\xi\|_{L^p(\Sigma,J)}.
    \end{equation}
    To see this, write $x=x_0+Gv_0+S^*y_0$, then apply Lemma \ref{fibreellipticestimate} to get
    $$\|x-x_0\|_{L^p(\Sigma,J)}\leq \|Gv_0\|_{L^p(\Sigma,J)} + \|S^*y_0\|_{L^p(\Sigma,J)}\leq c\|x\|_{L^p(\Sigma,J)},$$
    followed by the reverse triangle inequality. Since $X$ is independent of $s$ and $\pi_X$ is $I$-linear, we have
    \begin{equation}\label{lemma6.8identities}
        D\xi_0 = \pi_XD\xi,\hspace{4mm} D\xi_1 = (1-\pi_X)D\xi.
    \end{equation}
    Finally, \eqref{lemma6.8inequality} can be obtained by bounding the $\del_s\xi_0$ and $\del_s\xi_1$ terms separately, followed by using the identities \eqref{lemma6.8identities} and \eqref{projectionLp} to bound $\|D\xi_0\|_{L^p}$ and $\|D\xi_1\|_{L^p}$ by $\|D\xi\|_{L^p}$.
\end{proof}

\begin{lemma}\label{step3}
    Fix $p\geq 2$ and suppose $\mathfrak{p}=\mathfrak{p}(\e s)$, $\tau=\tau(\e s)$, and $X=X(\e s)$ are as in the statement of Proposition \ref{Yzerocase}. Then there exist constants $\e_0>0$ and $c>0$ such that for every $\e\in(0,\e_0)$ and $\xi\in C_0^\infty(\R,W_\Sigma)$ satisfying
    \begin{equation}\label{vanishingprojection}
        \pi_{\e s}\xi(s)=0,
    \end{equation}
    for all $s\in\R$, we have
    \begin{equation}\label{step3inequality}
        \int_{\R}\|\del_{s}\xi\|_{L^p(\Sigma_{\e s})}^p ds\leq c\int_{\R}\|D_{\e s}\xi\|_{L^p(\Sigma_{\e s})}^p ds.
    \end{equation}
\end{lemma}

\begin{proof}
    One first proves an error bound of the form
    \begin{equation}\label{differencebound}
        \|(D_{\e s} - D_{\e s_1})\xi\|_{L^p(\Sigma_{\e s})}^p \leq c|s-s_1|^p\e^p\Big(\|\xi\|_{W^{1,p}(\Sigma,J_0)}^p + \|\del_s\xi\|_{L^p(\Sigma,J_0)}^p\Big),
    \end{equation}
    for all $\xi\in C_0^\infty(\R, W_\Sigma)$, $s,s_1\in\R$, and $\e>0$. This follows directly from writing out the difference terms on the left hand side and applying the $C^1$ bounds \eqref{uniformbound1} and \eqref{uniformbound2}. Lemma \ref{fibreellipticestimate} and the condition \eqref{vanishingprojection} show that 
    \begin{align*}
        \|\xi\|_{W^{1,p}(\Sigma,J_0)}^p &\leq c\Big(\|\del_s\xi\|_{L^p(\Sigma_{\e s})}^p + \|D_{\e s}\xi\|_{L^p(\Sigma_{\e s})}^p\Big)
    \end{align*}
    for all $s\in\R$ and $\e>0$. Combining this with Lemma \ref{step2} and \eqref{differencebound}, we find that any $\xi\in C_0^\infty(\R,W_\Sigma)$ supported in a ball of radius $r$ centered at $s_1$ satisfies
    $$\int_{\R}\|\del_s\xi\|_{L^p(\Sigma_{\e s})}^pds \leq c(1+(r\e)^p)\int_{\R}\|D_{\e s}\xi\|_{L^p(\Sigma_{\e s})}^p + c(r\e)^p\int_{\R}\|\del_s\xi\|_{L^p(\Sigma_{\e s})}^p.$$
    Taking $\e_0$ sufficiently small and $r\e<\e_0$, the rightmost term can be absorbed into the left hand side. Then \eqref{step3inequality} holds for any $\xi\in C_0^\infty(\R,W_\Sigma)$ satisfying \eqref{vanishingprojection} and supported in a ball of radius $r<\e_0/\e$. To prove inequality \eqref{step3inequality} in general, smooth bump functions can be used to patch these local estimates together to get
    $$\int_{\R}\|\del_s\xi\|_{L^p(\Sigma_{\e s})}^p ds \leq c\int_{\R}\Big(\|D_{\e s}\xi\|_{L^p(\Sigma_{\e s})}^p + \|\xi\|_{L^p(\Sigma_{\e s})}^p\Big)ds.$$
    The $\|\xi\|_{L^p}$ term is a result of factoring out derivatives of the bump functions. This can be removed by applying Lemma \ref{LpL2exchange}:
    $$\int_{\R}\|\del_s\xi\|_{L^p(\Sigma_{\e s})}^p ds \leq c(1+2^p\delta)\int_{\R}\|D_{\e s}\xi\|_{L^p(\Sigma_{\e s})}^p ds + c_{\delta}\int_{\R}\|\xi\|_{L^2(\Sigma_{\e s})}^pds + c\delta\int_{\R}\|\del_s\xi\|_{L^p(\Sigma_{\e s})}^p ds.$$
    Choosing $\delta>0$ sufficiently small, the last term can be absorbed into the left hand side. The $\|\xi\|_{L^2}$ term can also be absorbed by applying Lemma \ref{step1} and \eqref{vanishingprojection}.
\end{proof}

\begin{lemma}\label{step4}
    Fix $p\geq 2$. Suppose $\mathfrak{p}=\mathfrak{p}(\e s)$, $\tau=\tau(\e s)$, and $X=X(\e s)$ are as in the statement of Proposition \ref{Yzerocase}. Then there exist constants $\e_0>0$ and $c>0$ such that for every $\xi\in C_0^\infty(\R,W_\Sigma)$ and $\e\in(0,\e_0)$, we have
    \begin{equation}\label{c4inequality}
        \int_{\R}\|\del_s\xi\|_{L^p(\Sigma_{\e s})}^p ds \leq c\int_{\R}\Big(\|D_{\e s}\xi\|_{L^p(\Sigma_{\e s})}^p + \e^p\|\xi\|_{L^p(\Sigma_{\e s})}^p\Big)ds.
    \end{equation}
\end{lemma}

\begin{proof}
    Decompose $\xi=(x,v,y)$ as
    $$\xi=\xi_0 + \xi_1,\hspace{3mm}\xi_0 = (x_0,0,0),\hspace{3mm}\xi_1 = (x_1,v,y),$$
    where $x_0:=\pi_{\e s}(x)$ and $x_1:=x-x_0$. Let $x=(a,\phi)$ and $x_i=(a_i,\phi_i)$ for $i=0,1$. The required bound for $\xi_1$ follows immediately from Lemma \ref{step3}, so the first step is to prove the lemma for $\xi_0$. Since $P_{\e s}\xi_0=0$, we have $D_{\e s}\xi_0=(\tilde{x}_0,0,0)$ for
    $$\tilde{x}_0 = \begin{pmatrix}
        \star_{\e s}\del_s a_0\\
        i\del_s\phi_0
    \end{pmatrix}.$$
    The bound \eqref{uniformbound1} shows that
    \begin{equation}\label{1dcz}
        \int_{\R}\|\del_s x_0\|_{L^p(\Sigma_{\e s})}^pds \leq c\int_{\R}\|\tilde{x}_0\|_{L^p(\Sigma_{\e s})}^pds,
    \end{equation}
    which proves a stronger version of \eqref{c4inequality} for $\xi_0$. Let $G_{\e s}=G_{X(\e s)}$ and $S_{\e s} = S_{X(\e s)}$. Then using $\del_sG_{\e s}^*x_0=0$ and $\del_sS_{\e s}x_0=0$, we obtain the following identities:
    \begin{align*}
        G_{\e s}^*\tilde{x}_0 &= \e\big(\dot{\star}_{\e s}da_0 - i\Re\langle\dot{\Phi}(\e s),\phi_0\rangle\big),\\
        S_{\e s}\tilde{x}_0 &= \e\begin{pmatrix}
            -\dot{\star}_{\e s}d\star_{\e s}a_0 - \star_{\e s}d\dot{\star}_{\e s}a_0 + i\Im\langle\dot{\Phi}(\e s),\phi_0\rangle\\
            -ia_0^{0,1}\dot{\Phi}(\e s) - \frac{1}{2}(\dot{\star}_{\e s}a_0)\Phi + i\dot{A}(\e s)\phi_0 + i\dot{B}(\e s)\phi_0 + \star_{\e s}\dot{\star}_{\e s}d_{AB}\phi_0
        \end{pmatrix}.
    \end{align*}
    Combining these with two applications of Lemma \ref{fibreellipticestimate}, we can bound the non-harmonic part of $\tilde{x}_0$:
    \begin{align}
        \nonumber\|\tilde{x}_0 - \pi_{\e s}(\tilde{x}_0)\|_{L^p(\Sigma_{\e s})}^p &\leq c\Big(\|G_{\e s}^*\tilde{x}_0\|_{L^p(\Sigma_{\e s})}^p + \|S_{\e s}\tilde{x}_0\|_{L^p(\Sigma_{\e s})}^p\Big)\\
        \nonumber&\leq c\e^p\|x_0\|_{W^{1,p}(\Sigma_{\e s})}^{p}\\
        \label{x0inequality}&\leq c\e^p\|x_0\|_{L^p(\Sigma_{\e s})}^p.
    \end{align}
    Now let $v_1$ and $y_1=(c_1,\psi_1)$ be such that
    $$x_1 = x - x_0 = G_{\e s}v_1 + S_{\e s}^*y_1.$$
    Then if $\tilde{x}_1$ is the first component of $D_{\e s}\xi_1$, we have
    $$\pi_{\e s}(\tilde{x}_1) = \e\pi_{\e s}\begin{pmatrix}
        \dot{\star}_{\e s}dc_1 - i\Im\langle\psi_1,\dot{\Phi}(\e s)\rangle\\
        \dot{\Phi}(\e s)(v_1 + ic_1) - \langle\dot{A}(\e s)^{0,1},\psi_1\rangle
    \end{pmatrix}.$$
    Then apply Lemma \ref{fibreellipticestimate} to get
    \begin{align}
        \nonumber\|\pi_{\e s}(\tilde{x}_1)\|_{L^p(\Sigma_{\e s})}^p &\leq c\e^p\Big(\|v_1\|_{W^{1,p}(\Sigma_{\e s})} + \|y_1\|_{W^{1,p}(\Sigma_{\e s})}\Big)\\
        \nonumber&\leq c\e^p\|G_{\e s}v_1 + S_{\e s}^*y_1\|_{L^p(\Sigma_{\e s})}\\
        \label{x1inequality}&=c\e^p\|x_1\|_{L^p(\Sigma_{\e s})}^p.
    \end{align}
    Letting $\tilde{x} = \tilde{x}_0 + \tilde{x}_1$ be the first component of $D_{\e s}\xi$, we can write
    $$\tilde{x}_0 = \pi_{\e s}(\tilde{x}) - \pi_{\e s}(\tilde{x}_1) + (\tilde{x}_0 - \pi_{\e s}(\tilde{x}_0)),$$
    $$\tilde{x}_1 = (\tilde{x} - \pi_{\e s}(\tilde{x})) - (\tilde{x}_0 - \pi_{\e s}(\tilde{x}_0)) + \pi_{\e s}(\tilde{x}_1).$$
    As in the proof of \cite[Lemma C.6]{Salamon2000QuantumPF}, each of these terms can be bounded using the inequalities \eqref{1dcz}, \eqref{x0inequality}, and \eqref{x1inequality}, resulting in
    $$\|\tilde{x}_0\|^p_{L^p(\Sigma_{\e s})} + \|\tilde{x}_1\|^p_{L^p(\Sigma_{\e s})} \leq c\Big(\|\tilde{x}\|^p_{L^p(\Sigma_{\e s})} + \e^p\|x\|^p_{L^p(\Sigma_{\e s})}\Big).$$
    Then we have
    $$\|D_{\e s}\xi_0\|_{L^p(\Sigma_{\e s})}^p + \|D_{\e s}\xi_1\|_{L^p(\Sigma_{\e s})}^p \leq c\Big(\|D_{\e s}\xi\|_{L^p(\Sigma_{\e s})}^p + \e^p\|\xi\|_{L^p(\Sigma_{\e s})}^p\Big),$$
    so \eqref{c4inequality} follows from \eqref{1dcz} and Lemma \ref{step3} applied to $\xi_1$. 
\end{proof}

\begin{proof}[Proof of Proposition \ref{Yzerocase}]
    It follows from Lemma \ref{step4} that
    $$\int_{\R}\|P_{\e s}\xi\|_{L^p(\Sigma_{\e s})}^pds \leq c\int_{\R}\Big(\|D_{\e s}\xi\|_{L^p(\Sigma_{\e s})}^p + \e^p\|\xi\|_{L^p(\Sigma_{\e s})}^p\Big)ds.$$
    This can be combined with Lemmas \ref{step4} and \ref{fibreellipticestimate} to get
    \begin{align*}
        \int_{\R}&\Big(\|\del_s\xi\|_{L^p(\Sigma_{\e s})}^p + \|P_{\e s}\xi\|_{L^p(\Sigma,\tilde{J})}^p\Big)ds\\
        &\leq c\int_{\R}\Big(\|D_{\e s}\xi\|_{L^p(\Sigma_{\e s})}^p + \e^p\|\xi\|_{L^p(\Sigma_{\e s})}^p\Big)ds\\
        &\leq c\int_{\R}\Big(\|D_{\e s}\xi\|_{L^p(\Sigma_{\e s})}^p + \e^p\|P_{\e s}\xi\|_{L^p(\Sigma_{\e s})}^p + \e^p\|\pi_{\e s}\xi\|_{L^p(\Sigma_{\e s})}^p\Big)ds.
    \end{align*}
    For $\e>0$ sufficiently small, we obtain the desired inequality \eqref{Yzeroinequality}.
\end{proof}

\begin{prop}\label{technicallemma}
    Fix $p\geq 2$ and let $\Xi=(X,0,Y)\in\A_0(\mathfrak{s}_{d,\tilde{f}},\mathfrak{p},\eta)$. Then there exist constants $\e_0>0$ and $c>0$ such that
    \begin{align}\label{keyestimate}
        \nonumber\int_{0}^{1}&\Big(\|G^*_{X}x\|_{L^p(\Sigma_t)}^p + \|S_{X}x\|_{L^p(\Sigma_t)}^p + \e^p\|N_{Y}x\|_{L^p(\Sigma_t)}^p + \e^p\|G_Xv\|_{L^p(\Sigma_t)}^p\\
        &\nonumber\hspace{3mm}+ \e^{2p}\|L_Yv\|_{L^p(\Sigma_t)}^p + \e^p\|S_{X}^*y\|_{L^p(\Sigma_t)}^p + \e^{2p}\|L_Y^*y\|_{L^p(\Sigma_t)}^p + \e^{2p}\|M_Yy\|_{L^p(\Sigma_t)}^p\Big)dt\\
        &\hspace{8mm}\leq c\e^p\|\D_\e\xi\|_{0,p,\e}^p + c\e^p\int_{0}^{1}\|\pi_X(x)\|_{L^p(\Sigma_t)}^pdt
    \end{align}
    for all $\xi\in W^{1,p}_{f}$ and $\e\in(0,\e_0)$. The constant $c$ can be chosen independently of $\e$. 
\end{prop}

\begin{proof}
    For $s\in[0,1/\e]$, define the following rescaled quantities
    \begin{align*}
        x'(s) &= x(\e s), & X'(s) &= X(\e s),\\
        v'(s) &= \e v(\e s), & V'(s) &= \e V(\e s),\\
        y'(s) &= \e y(\e s), & Y'(s) &= \e Y(\e s),\\
        \xi' &= (x',v',y'), & \Xi' &= (X',V',Y'),
    \end{align*}
    $$\D'_{\e}(\Xi') = \begin{pmatrix}
        N_{Y'} & G_{X'} & S_{X'}^*\\
        G_{X'}^* & 0 & L_{Y'}^*\\
        S_{X'} & L_{Y'} & M_{Y'}
        \end{pmatrix}.$$

    These objects can be viewed as living on a stretched mapping torus of length $1/\e$, and satisfy the uniform bounds \eqref{uniformbound1} and \eqref{uniformbound2} with $C$ independent of $\e$. By substituting $t=\e s$, \eqref{keyestimate} is equivalent to
    \begin{align}
        \nonumber\int_{0}^{1/\e}&\Big(\|G^*_{X'}x'\|_{L^p(\Sigma_{\e s})}^p + \|S_{X'}x'\|_{L^p(\Sigma_{\e s})}^p + \|N_{Y'}x'\|_{L^p(\Sigma_{\e s})}^p + \|G_{X'}v'\|_{L^p(\Sigma_{\e s})}^p + \|L_{Y'}v'\|_{L^p(\Sigma_{\e s})}^p\\
        \label{rescaledinequality}&\hspace{5mm}+ \|S_{X'}^*y'\|_{L^p(\Sigma_{\e s})}^p + \|L_{Y'}^*y'\|_{L^p(\Sigma_{\e s})}^p + \|M_{Y'}y'\|_{L^p(\Sigma_{\e s})}^p\Big)ds\\
        &\leq c\int_{0}^{1/\e}\Big(\|\D'_\e\xi'\|_{L^p(\Sigma_{\e s})}^p + \e^p\|\pi_{X'}(x')\|_{L^p(\Sigma_{\e s})}^p\Big)ds.
    \end{align}
    Lemma \ref{Yzerocase} proves this for $Y=Y'=0$, since $\D'_{\e} = D_{\e s}$ in that case. Then we just need to control the error terms involving $N$, $L$, $L^*$, and $M$. For example, it follows from the definition of $N$ that we have a pointwise inequality
    $$|N_{Y'}x' - N_{0}x'| \leq c\e|x'|.$$
    Integrating and using the elliptic estimate \eqref{fibreellipticestimate}, we obtain
    \begin{align*}
        \int_{0}^{1/\e}&\|N_{Y'}x'-N_0x'\|_{L^p(\Sigma_{\e s})}^pds\\
        &\leq c\e^p\int_{0}^{1/\e}\Big(\|G_{X'}^*\tilde{x}\|_{L^p(\Sigma_{\e s})}^p + \|S_{X'}x'\|_{L^p(\Sigma_{\e s})}^p + \|\pi_{X'}(x')\|_{L^p(\Sigma_{\e s})}^p\Big)ds.
    \end{align*}
    Similar error estimates hold for $L$, $L^*$, and $M$. After bounding the left hand side of \eqref{rescaledinequality} by Lemma \ref{Yzerocase} and the error estimates, we can choose $\e>0$ small enough to absorb the the terms involving $G$ and $S$ into the left hand side. This proves the proposition.
\end{proof}

\subsection{Weighted elliptic estimates}

Having obtained our preliminary estimates, we can now prove the weighted elliptic estimates for $\D_\e$. Throughout, fix $p\geq 2$, $(\mathfrak{p},\eta)\in\mathcal{U}_{\varphi}^{reg}$, and let $\Xi=(X,0,Y)\in\A_0(\mathfrak{s}_{d,\tilde{f}},\mathfrak{p},\eta)$. 

\begin{prop}\label{modifiedellipticestimate}
    There exist constants $\e_0>0$ and $c>0$ such that
    \begin{equation}\label{elliptic1}
        \|\xi\|_{1,p,\e} \leq c\big(\e\|\D_\e\xi\|_{0,p,\e} + \|\pi_X\xi\|_{L^p}\big),
    \end{equation}
    \begin{equation}\label{elliptic2}
        \|\xi - \pi_X\xi\|_{1,p,\e} \leq c\e\big(\|\D_\e\xi\|_{0,p,\e} + \|\pi_X\xi\|_{L^p}\big),
    \end{equation}
    for all $\xi\in W^{1,p}_{f}$ and $0<\e\leq \e_0$. The constant $c$ can be chosen independently of $\e$.
\end{prop}

\begin{proof}
    The inequality \eqref{elliptic1} follows directly from Proposition \ref{technicallemma} and Lemma \ref{fibreellipticestimate}. In what follows, we omit subscripts from the operators $N,G,S,L,M,$ and $\pi$. To prove inequality \eqref{elliptic2}, let $\xi=(x,v,y)$ and let $u$ be a solution to the equation
    $x-x_0=\Delta u,$
    where $x_0=\pi(x)$ and $\Delta = GG^* + S^*S$. Let $v_0=G^*u$ and $y_0=Su$, which are uniquely determined by $x$ and are independent of the choice of $u$. Then by the Calderon--Zygmund inequality, there is a constant $c>0$ such that
    \begin{align}
        \|Gv_0\|_{L^p(\Sigma_t)} + \|S^*y_0\|_{L^p(\Sigma_t)} &= \|GG^*u\|_{L^p(\Sigma_t)} + \|S^*Su\|_{L^p(\Sigma_t)}\nonumber \\
        &\leq c\|\Delta u\|_{L^p(\Sigma_t)}\nonumber \\
        &= c\|x-x_0\|_{L^p(\Sigma_t)}.\label{czinequality1}
    \end{align}
    A priori, the constant depends on $X(t)$ and $\mathfrak{p}(t)$, but as described in the previous section, the dependence on $t$ can be removed using the bounds \eqref{uniformbound1}, \eqref{uniformbound2}, and the Arzela-Ascoli theorem. The following identities are obtained from Lemma \ref{operatoridentities}.
    \begin{equation}\label{nidentity}
        N(x-x_0) = S^*Lv_0 - GL^*y_0 - S^*My_0 + Ry_0
    \end{equation}
    \begin{align*}
        G^*GL^*y_0 &= G^*Ry_0 - G^*Nx,\\
        SS^*(Lv_0 + My_0) &= SRy_0 - SN x + R^*x_0.
    \end{align*}
    Using Lemma \ref{fibreellipticestimate} and the explicit form of $R$ given in equation \eqref{identity2}, we have
    \begin{align*}
        \|GL^*y_0\|_{L^p(\Sigma_t)} + \|S^*(Lv_0 + My_0)\|_{L^p(\Sigma_t)} &\leq c\big(\|Nx\|_{L^p(\Sigma_t)} + \|Ry_0\|_{L^p(\Sigma_t)} + \|R^*x_0\|_{L^p(\Sigma_t)} \big)\\
        &\leq c\big(\|Nx\|_{L^p(\Sigma_t)} + \|y_0\|_{W^{1,p}(\Sigma_t)} + \|x_0\|_{W^{1,p}(\Sigma_t)}\big)\\
        &\leq c\big(\|Nx\|_{L^p(\Sigma_t)} + \|x-x_0\|_{L^p(\Sigma_t)} + \|x_0\|_{L^p(\Sigma_t)}\big)\\
        &\leq c\big(\|Nx\|_{L^p(\Sigma_t)} + \|x\|_{L^p(\Sigma_t)}\big).
    \end{align*}
    Now use \eqref{nidentity} to get the bound
    $$\|N(x-x_0)\|_{L^p(\Sigma_t)}\leq c\big(\|Nx\|_{L^p(\Sigma_t)} + \|x\|_{L^p(\Sigma_t)}\big).$$
    Then \eqref{elliptic2} follows from Proposition \ref{technicallemma} and Lemma \ref{fibreellipticestimate}.
\end{proof}

\begin{lemma}\label{D0vsDepsilon}
    There is a constant $c>0$ such that for all $\e>0$ and $\xi\in W^{1,p}_{f}$, we have
    $$\|\pi_X\D_\e\xi - \D_0\pi_X\xi\|_{L^p}\leq c\|\xi-\pi_X\xi\|_{0,p,\e}.$$
\end{lemma}

\begin{proof}
    Decompose $x = x_0 + Gv_0 + S^*y_0$, where $x_0=\pi(x)$. Then we have
    \begin{align*}
        \|\pi\D_\e\xi - \D_0\pi\xi\|_{L^p} &= \|\pi N(Gv_0 + S^*y_0)\|_{L^p}\\
        &= \|\pi Ry_0\|_{L^p}\\
        &\leq c\|y_0\|_{W^{1,p}}\\
        &\leq c\|S^*y_0\|_{L^p}\\
        &\leq c\|\xi - \pi\xi\|_{0,p,\e}.
    \end{align*}
    The first equality follows from the definitions of $\D_\e$ and $\D_0$. The second line follows from Lemma \ref{operatoridentities}. The third line holds because $Ry_0$ only involves derivatives of $y_0$ in the $\Sigma$ direction, and the final two lines come from Lemma \ref{fibreellipticestimate} and inequality \eqref{czinequality1}.
\end{proof}

\begin{lemma}\label{modifiedellipticestimate2}
    Suppose $\Xi=(X,0,Y)$ is a non-degenerate solution to the adiabatic limit equations, so that $\D_0=\D_0(X)$ is injective. Then there exist constants $c>0$ and $\e_0>0$ such that
    $$\|\xi\|_{1,p,\e}\leq c\big(\e\|\D_\e\xi\|_{0,p,\e} + \|\pi_X\D_\e\xi\|_{L^p}\big),$$
    $$\|\xi - \pi_X\xi\|_{1,p,\e}\leq c\e\|\D_\e\xi\|_{0,p,\e},$$
    for all $\xi\in W^{1,p}_{f}$ and all $0<\e\leq\e_0$.
\end{lemma}

\begin{proof}
    Since $(\mathfrak{p},\eta)\in\mathcal{U}_\varphi^{reg}$, $\D_0:W^{1,p}(\mathcal{H}_X)\ra L^p(\mathcal{H}_X)$ is a surjective index zero Fredholm operator. Then it is also injective and satisfies a bound of the form
    $$\|x_0\|_{1,p}\leq c\|\D_0 x_0\|_p.$$ 
    Applying this to $\pi\xi$ and using Lemma \ref{D0vsDepsilon}, we have
    \begin{align*}
        \|\pi\xi\|_{W^{1,p}} &\leq c\|\D_0\pi\xi\|_{L^p}\\
        &\leq c(\|\pi\D_\e\xi\|_{L^p} + \|\xi-\pi\xi\|_{0,p,\e})\\
        &\leq c(\|\pi\D_\e\xi\|_{L^p} + \e\|\D_\e\xi\|_{0,p,\e} + \e\|\pi\xi\|_{L^p})
    \end{align*}
    Choosing $\e_0>0$ such that $c\e_0<1$, we can absorb the last term into the left hand side to get
    \begin{equation}\label{randominequality}
        \|\pi\xi\|_p\leq c(\|\pi\D_\e\xi\|_{L^p} + \e\|\D_\e\xi\|_{0,p,\e}).
    \end{equation}
    The $W^{1,p}$ norm on $\Gamma(\mathcal{H}_X)$ is equivalent to the norm $\|x\|_{L^p} + \|Nx\|_{L^p}$, so by Lemma \ref{modifiedellipticestimate} we have 
    \begin{align*}
        \|\xi\|_{1,p,\e} &\leq \|\xi-\pi\xi\|_{1,p,\e} + c\|\pi\xi\|_{W^{1,p}}\\
        &\leq c(\e\|\D_\e\xi\|_{0,p,\e} + \e\|\pi\xi\|_{L^p} + \|\pi\D_\e\xi\|_{L^p} + \e\|\D_\e\xi\|_{0,p,\e})\\
        &\leq c(\e\|\D_\e\xi\|_{0,p,\e} + \|\pi\D_\e\xi\|_{L^p} + \e\|\pi\xi\|_{L^p}).
    \end{align*}
    Once again, $\e\|\pi\xi\|_{L^p}$ can be absorbed into the left hand side, proving the first inequality. To prove the second inequality, apply Lemma \ref{modifiedellipticestimate} and inequality \eqref{randominequality} as above, followed by $\|\pi\D_\e\xi\|_{L^p}\leq c\|\D_\e\xi\|_{0,p,\e}$.
\end{proof}

The bound \eqref{projectionLp} and $\ind(\D_\e)=0$ imply the following corollary of Lemma \ref{modifiedellipticestimate2}.

\begin{cor}\label{uniforminvertibility}
    Suppose $(\mathfrak{p},\eta)\in\mathcal{U}_{\varphi}^{reg}$. Then every $\Xi\in\A_0(\mathfrak{s}_{d,\tilde{f}},\mathfrak{p},\eta)$ is non-degenerate and $\D_{\e}(\Xi)$ is uniformly invertible for sufficiently small $\e>0$. That is, there exist constants $c>0$ and $\e_0>0$ such that for all $\e\in(0,\e_0)$ we have
    $$\|\xi\|_{1,p,\e}\leq c\|\D_{\e}(\Xi)\xi\|_{0,p,\e}$$
    for all $\xi\in W^{1,p}_f$. The constant $c$ does not depend on $\Xi$ or $\e$. 
\end{cor}

\subsection{Constructing the perturbation map}

We now use our estimates to perturb a solution of the adiabatic limit equations to a genuine multi-monopole for $\e>0$ sufficiently small. A sequence of perturbations is iteratively constructed via Newton's method. We also prove that the perturbation map is injective, from which Theorem \ref{maintheorem} follows. This argument requires a bound on the quadratic error term, which can only be obtained in terms of the weighted $\infty,\e$-norm. In order to work with this, we need the following lemma.

\begin{lemma}\label{Linfinityestimate}
    For every $p>3$ there exists a constant $c>0$ such that
    $$\|\xi\|_{\infty,\e}\leq c\e^{-1/p}\|\xi\|_{1,p,\e}$$
    for all $\xi\in W^{1,p}_{f}$ and $0<\e\leq 1$. The constant $c$ can be chosen independently of $\e$. 
\end{lemma}

\begin{proof}
    Let $\xi'$ be the rescaled variable introduced in the proof of Proposition \ref{technicallemma}. Then one can check that $\e^{-1/p}\|\xi\|_{1,p,\e}$ is equivalent to the standard $1,p$-norm on the rescaled mapping torus of length $1/\e$, which we denote by $\|\xi'\|$. For $p>3$ there is the Sobolev inequality $\|\xi'\|_{L^\infty}\leq c\|\xi'\|$. The constant $c$ is independent of $\e$ because this inequality is obtained by gluing together many local inequalities, each of which is independent of $\e$. Finally, we have 
    $$\|\xi\|_{\infty,\e}=\|\xi'\|_{L^\infty} \leq c\e^{-1/p}\|\xi\|_{1,p,\e}.$$
\end{proof}

\begin{defn}
    Let $\mathfrak{p}\in\mathcal{P}_\varphi$ and $\eta\in\mathcal{V}_f^+$. For $\e>0$ define the Seiberg--Witten map
    $$\SW_{\mathfrak{p},\eta,\e}:\mathcal{C}_f\ra W_f,$$
    $$\SW_{\mathfrak{p},\eta,\e}(A,\Phi,V,b,\Psi) = \begin{pmatrix}
        \star_t(\dot{A}-db-i\sigma) - dV - i\Im\langle\Psi,\Phi\rangle\\
        -i\nabla_t\Phi + \delbar_{A,\mathfrak{p}}^*\Psi - V\Phi\\
        0\\
        \e^{-2}(\star_tF_A - \frac{i}{2}|\Phi|^2 + i\tau) - \dot{V} + \frac{i}{2}|\Psi|^2\\
        i\nabla_t\Psi + \e^{-2}\delbar_{A,\mathfrak{p}}\Phi - V\Psi
    \end{pmatrix}.$$
\end{defn}

\begin{theorem}\label{ezerotoesmall}
    Suppose $(\mathfrak{p},\eta)\in\mathcal{U}_\varphi^{reg}$ and $p>3$. Then there is $\e_0>0$ such that for each $\e\in(0,\e_0)$ there is a smooth gauge-equivariant map
    $$\mathcal{T}_\e : \A_0^{1,p}(\mathfrak{s}_{d,\tilde{f}},\mathfrak{p},\eta)\ra \A_{\e}^{1,p}(\mathfrak{s}_{d,\tilde{f}},\mathfrak{p},\eta),$$
    such that $\Xi_\e = \mathcal{T}_\e(\Xi_0)$ satisfies the Coulomb gauge fixing condition
    \begin{equation}\label{gaugefixingcondition}
        \mathbb{G}_{\Xi_0}^{*_\e}(\Xi_\e - \Xi_0) = \e^{-2}G_{X_0}^{*}(X_{\e}-X_0) + L_{Y_0}^*(Y_{\e}-Y_0)=0
    \end{equation}
    for $\Xi_\e=(X_\e,0,Y_\e)$ and $\Xi_0=(X_0,0,Y_0)$. Moreover, there is a quadratic estimate
    \begin{equation}\label{quadraticconvergence}
        \|\Xi_\e - \Xi_0\|_{1,p,\e}\leq c\e^2,
    \end{equation}
    where the $1,p,\e$-norm is determined by $\Xi_0$ and the constant $c$ is independent of $\Xi_0$ and $\e$. 
\end{theorem}

\begin{proof}
    Suppose $\Xi_0=(A_0,\Phi_0,0,b_0,\Psi_0)\in\A_0^{1,p}(\mathfrak{s}_{d,\tilde{f}},\mathfrak{p},\eta)$. Throughout, we abbreviate $\D_{\e}(\Xi_0)$ by $\D_\e$ and $\SW_{\mathfrak{p},\eta,\e}$ by $\SW_{\e}$. We will use $c_0,c_1,$ etc. for constants independent of $\e$ and $\Xi_0$. To construct $\Xi_\e$, we need to solve the equation
    \begin{equation}\label{deformedequation}
        0=\SW_{\e}(\Xi_0) + \D_\e\xi + \mathcal{Q}_\e\xi
    \end{equation}
    for $\xi = \Xi_\e-\Xi_0$. This encodes both $\SW_{\e}(\Xi_\e)=0$ and the gauge-fixing condition \eqref{gaugefixingcondition}. The quadratic term $\mathcal{Q}_\e:W_f^{1,p}\ra W_f^{1,p}$ is independent of $\Xi_0$ and is defined by
    \begin{equation}\label{quadraticerror}
        \mathcal{Q}_\e\xi = \SW_{\e}(\Xi_0 + \xi) - \SW_{\e}(\Xi_0) - D_{\Xi_0}\SW_{\e}(\xi) = \begin{pmatrix}
            i\textnormal{Im}\langle\psi,\phi\rangle\\
            - \langle a^{0,1},\psi\rangle-ic\phi\\
            0\\
            \frac{i}{2}(\e^{-2}|\phi|^2 - |\psi|^2)\\
            \e^{-2}a^{0,1}\phi + ic\psi
        \end{pmatrix}
    \end{equation}
    for $\xi=(a,\phi,v,c,\psi)$ (the second equality follows from expanding out all of the terms). We need to assume $p>3$ here so that $W^{1,p}_f$ is an algebra (by the Sobolev multiplication theorem) and $\mathcal{Q}_\e$ is well-defined. Note that $\Xi_0$ is an approximate zero of $\SW_{\e}$ in the sense that
    $$\|\SW_{\e}(\Xi_0)\|_{0,p,\e} = \big\|(0,0,0,\frac{i}{2}|\Psi_0|^2,i\nabla_t\Psi_0)\big\|_{0,p,\e} \leq c_0\e.$$
    The constant $c_0$ can be chosen independent of $\Xi_0$ because $\M_{0}(\mathfrak{s}_{d,\tilde{f}},\mathfrak{p},\eta)$ is a finite set. To begin the iteration, define $\Xi_1 = \Xi_0 + \xi_0$, where $\xi_0\in W_f^{1,p}$ is the unique solution to the equation
    \begin{equation}\label{firstiteration}
        \D_\e\xi_0=-\SW_\e(\Xi_0).
    \end{equation}
    It is unique because $\D_\e$ is an isomorphism for small $\e>0$. Lemmas \ref{Linfinityestimate} and \ref{modifiedellipticestimate2} imply that 
    \begin{equation}\label{Linfinityestimate1}
        \|\xi_0\|_{1,p,\e} + \e^{1/p}\|\xi_0\|_{\infty,\e}\leq c_1\e\|\SW_{\e}(\Xi_0)\|_{0,p,\e}\leq c_0c_1\e^2,
    \end{equation}
    since the first components of $\SW_{\e}(\Xi_0)$ vanish. Since $\D_\e=G_{\Xi_0}^{*_\e}\oplus D_{\Xi_0}\SW_\e$, equation \eqref{firstiteration} is equivalent to $D_{\Xi_0}\SW_{\e}(\xi_0) = -\SW_{\e}(\Xi_0)$ and $G_{\Xi_0}^{*_\e}(\xi_0)=0$. It follows that
    $$\SW_{\e}(\Xi_1) = \SW_{\e}(\Xi_0+\xi_0) - \SW_{\e}(\Xi_0) - D_{\Xi_0}\SW_{\e}(\xi_0) = \mathcal{Q}_\e\xi_0,$$
    so from \eqref{Linfinityestimate1} and our explicit equation \eqref{quadraticerror} for $\mathcal{Q}_\e$ we obtain
    $$\|\SW_{\e}(\Xi_1)\|_{0,p,\e} \leq c_2\e^{-1}\|\xi_0\|_{0,p,\e}\|\xi_0\|_{\infty,\e}\leq c_0c_1c_2\e^{1-1/p}\|\xi_0\|_{0,p,\e}.$$
    To continue the iteration, suppose we have constructed $\Xi_k$ and define $\Xi_{k+1}=\Xi_k + \xi_k$, where $\xi_k$ is the unique solution to the equation
    \begin{equation}\label{inductionstep}
        \D_\e\xi_k=-\SW_\e(\Xi_k).
    \end{equation}
    Following the proof of \cite[Theorem 5.1]{Dostoglou-Salamon1994}, we use induction to show that
    \begin{align*}
        \|\xi_k\|_{1,p,\e} + \e^{1/p}\|\xi_k\|_{\infty,\e} &\leq 2c_1\|\SW_\e(\Xi_k)\|_{0,p,\e},\\
        \|\xi_k\|_{1,p,\e} + \e^{1/p}\|\xi_k\|_{\infty,\e} &\leq 2^{-k}c_0c_1\e^2,\\
        \|\SW_\e(\Xi_{k+1})\|_{0,p,\e} &\leq c_3\e^{1-1/p}\|\xi_k\|_{0,p,\e},
    \end{align*}
    for $0<\e<\e_0$ provided that $\e_0>0$ is sufficiently small. The first inequality follows from Lemma \ref{modifiedellipticestimate2}, although the first components of $\SW_\e(\Xi_k)$ are no longer zero. The second inequality follows from the first and the previous induction steps:
    \begin{align*}
        \|\xi_k\|_{1,p,\e} + \e^{1/p}\|\xi_k\|_{\infty,\e} &\leq 2c_1\|\SW_\e(\Xi_k)\|_{0,p,\e}\\
        &\leq 2c_1c_3\e^{1-1/p}\|\xi_{k-1}\|_{0,p,\e}\\
        &\leq 2^{-1}\|\xi_{k-1}\|_{0,p,\e}\\
        &\leq 2^{-k}\|\xi_0\|_{0,p,\e}\\
        &\leq 2^{-k}c_0c_1\e^2.
    \end{align*}
    Here, we have chosen $\e_0$ such that $2c_1c_3\e_0^{1-1/p}\leq 1/2$. It follows that
    \begin{equation}\label{Linfinityestimate2}
        \|\Xi_k-\Xi_0\|_{\infty,\e}\leq \sum_{j=0}^{k-1}\|\xi_j\|_{\infty,\e}\leq2c_0c_1\e^{2-1/p}.
    \end{equation}
    The third inequality follows from using equation \eqref{inductionstep} to get
    \begin{align*}
        \SW_\e(\Xi_{k+1}) &= \SW_\e(\Xi_k + \xi_k) - \SW_\e(\Xi_k) - D_{\Xi_0}\SW_\e(\xi_k)\\
        &= \SW_\e(\Xi_k + \xi_k) - \SW_\e(\Xi_k) - D_{\Xi_k}\SW_\e(\xi_k) + (D_{\Xi_k}\SW_\e - D_{\Xi_0}\SW_\e)\xi_k\\
        &= \mathcal{Q}_\e\xi_k + (D_{\Xi_k}\SW_\e - D_{\Xi_0}\SW_\e)\xi_k.
    \end{align*}
    Then we have
    \begin{align*}
        \|\SW_\e(\Xi_{k+1})\|_{0,p,\e} &\leq \|\mathcal{Q}_\e\xi_k\|_{0,p,\e} + \|(D_{\Xi_k}\SW_\e - D_{\Xi_0}\SW_\e)\xi_k\|_{0,p,\e}\\
        &\leq c_4\e^{-1}(\|\xi_k\|_{\infty,\e} + \|\Xi_k-\Xi_0\|_{\infty,\e})\|\xi_k\|_{0,p,\e}\\
        &\leq 3c_0c_1c_4\e^{1-1/p}\|\xi_k\|_{0,p,\e}.
    \end{align*}
    The last estimate follows from \eqref{Linfinityestimate2} and the second to last comes from writing out $\mathcal{Q}_\e$ and $\SW_\e$ explicitly (see \eqref{quadraticerror}). The bound $\|\xi_k\|_{1,p,\e}\leq 2^{-k}c_0c_1\e^2$ shows that $\Xi_k$ converges in $\mathcal{C}^{1,p}_f$ to a limit $\Xi_\e$. The map $\mathcal{T}_\e$ is well-defined because
    $$\|\SW_\e(\Xi_{k+1})\|_{0,p,\e}\leq c_3\e^{1-1/p}\|\xi_{k}\|_{0,p,\e}\leq 2^{-k}c_0c_1c_3\e^2$$
    implies $\SW_\e(\Xi_\e)=0$. Moreover, we have
    \begin{equation}
        \|\Xi_k-\Xi_0\|_{1,p,\e}\leq \sum_{j=0}^{k-1}\|\xi_j\|_{1,p,\e}\leq 2c_0c_1\e^2,
    \end{equation}
    so $\Xi_\e$ satisfies the quadratic estimate \eqref{quadraticconvergence}. Equation \eqref{inductionstep} implies that
    $$\e^{-2}G_{X_0}^{*}x_k + L_{Y_0}^*y_k=0$$
    for all $k$, so the gauge fixing condition \eqref{gaugefixingcondition} holds in the limit $k\ra\infty$.
\end{proof}

\begin{remark}
    The map $\mathcal{T}_{\e}$ is $\G_f^{2,p}$-equivariant and we will also denote by $\mathcal{T}_\e$ the induced map on the moduli spaces.
\end{remark}

\begin{prop}\label{injectivity}
    Suppose $(\mathfrak{p},\eta)\in\mathcal{U}_{\varphi}^{reg}$ and $p>3$. Then there is $\e_0>0$ such that $\mathcal{T}_\e:\M_0(\mathfrak{s}_{d,\tilde{f}},\mathfrak{p},\eta)\ra\M_{\e}(\mathfrak{s}_{d,\tilde{f}},\mathfrak{p},\eta)$ is injective for all $\e\in(0,\e_0)$.
\end{prop}

\begin{proof}
    Suppose not. Then there exist $\Xi_0,\Xi_0'\in\A_0^{1,p}(\mathfrak{s}_{d,\tilde{f}},\mathfrak{p},\eta)$ and sequences $\e_k\ra 0$, $g_k\in\GG_f^{2,p}$ such that for $\Xi_k=\mathcal{T}_{\e_k}(\Xi_0)$ and $\Xi_k'=\mathcal{T}_{\e_k}(\Xi_0')$, we have $g_k^*\Xi_k = \Xi_k'$. If $\Xi_k=(A_k,\Phi_k)$ and $\Xi_k'=(A_k',\Phi_k')$ (for connections and spinors on the 3-manifold $\Sigma_f$), we have
    $$dg_k = g_k(A_k-A_k'),\hspace{3mm}g_k\Phi_k = \Phi_k'.$$
    Rather than working in the $1,p,\e$ norm, we can rescale the sequence as in the proof of Lemmas \ref{technicallemma} and \eqref{Linfinityestimate}, which we denote by the same symbols. By combining Lemma \ref{Linfinityestimate} and the quadratic estimate \eqref{quadraticconvergence}, one can see the rescaled sequences still converge (for $2-1/p>0$) in the standard Sobolev norm on $\Sigma_f$. Taking the $L^p$ norm of the above equation, we have
    $$\|dg_k\|_{L^p}\leq \|g_k\|_{L^\infty}\|A_k-A_k'\|_{L_p}.$$
    Then $g_k$ is uniformly bounded in $W^{1,p}$ since $g_k$ is a $U(1)$ gauge transformation and $A_k,A_k'$ converge. To get a uniform $W^{2,p}$ bound, note that
    $$\|d(g_k(A_k-A_k'))\|_{L^p} \leq \|dg_k\wedge(A_k-A_k')\|_{L^p} + \|g_k d(A_k-A_k')\|_{L^p}$$
    Recall the Sobolev embedding $W^{1,p}\subset L^\infty$ for $p>3$. Then there are uniform bounds on $g_k\in L^\infty$, $dg_k\in L^p$, $A_k-A_k'\in L^\infty$, and $d(A_k-A_k')\in L^p$, so the above inequality shows that $g_k$ is uniformly bounded in $W^{2,p}$. Then there is a subsequence (still denoted by $g_k$) converging to $g_0\in\GG_f^{2,p}$ weakly in $W^{2,p}$ and strongly in $L^\infty$. It also follows that $dg_k$ converges strongly in $L^p$ to $dg_0$, so $dg_0=g_0(A_0-A_0')$. Similarly, $g_0\Phi_0=\Phi_0'$, so $g_0^*\Xi_0=\Xi_0'$.
\end{proof}

Theorem \ref{maintheorem} immediately follows from the results of the last two sections.

\begin{proof}[Proof of Theorem \ref{maintheorem}]
    The set $\mathcal{U}_{\varphi}^{reg}$ was constructed in Proposition \ref{transversality}. By the Lemmas \ref{correspondences} and \ref{kernelidentification}, fixed points of $\Upsilon_{d,\mathfrak{p},\eta}$ for $(\mathfrak{p},\eta)\in\mathcal{U}_{\varphi}^{reg}$ correspond to non-degenerate solutions of the adiabatic limit equation. The injective map $\mathcal{T}_\e$ is constructed by Theorem \ref{ezerotoesmall} and Proposition \ref{injectivity}, which completes the proof. 
\end{proof}

%% file: chapters/counting_multi-monopoles.tex
\section{Examples of multi-monopoles on mapping tori}\label{s5}

In this section, we will apply Theorem \ref{maintheorem} to construct examples of multi-monopoles on mapping tori. We begin by reviewing the complex geometry of the multi-vortex moduli spaces and its relationship with the Jacobian torus. This is used to give a geometric description of the subsets $\Fix_{\tilde{f}}(\Upsilon)$, which is necessary for understanding the spin$^c$ structures of multi-monopoles produced by Theorem \ref{maintheorem}. Our first example constructs multi-monopoles when $d$ is large compared to the genus $g$ of the fibres. In this case, the multi-vortex moduli space is a projective bundle over the Jacobian, which helps in computing the monodromy map $\Upsilon$. The second example concerns mapping tori with genus 1 fibres, where the multi-vortex moduli spaces are discrete and the monodromy can be computed in terms of permutations. This, along with our description of spin$^c$ structures, allows us to prove Corollary \ref{mmexistence}.

\subsection{Multi-vortices and complex geometry}

Fix a complex structure $J\in\mathcal{J}(\Sigma,\omega)$ so that $(\Sigma,J)$ is a genus $g$ Riemann surface (which we denote by just $\Sigma$ from now on). Proposition \ref{vortexmodulispaces} gave a description of the moduli space of framed multi-vortices as
$$\N_\Sigma(d,\E) = \{(\L,\alpha) : \L\in J^d_\Sigma, \alpha\in \P H^0(\Sigma,\L\otimes\E)\}.$$
It depends on $d\in\Z$ and a holomorphic bundle $\E\ra\Sigma$ of degree $D$. From now on, we fix $D=0$ for simplicity. There is a holomorphic map
$$\pi:\N_\Sigma(d,\E)\ra J^d_{\Sigma},\hspace{3mm}(\L,\alpha)\mapsto \L,$$
where $J^d_{\Sigma}$ is the Jacobian torus of $\Sigma$. This is a generalisation of the Abel--Jacobi map $\Sym^d(\Sigma)\ra J_\Sigma^d$, which is recovered when $\E$ is a line bundle. We can sometimes describe $\N_\Sigma$ more explicitly for generic choices of $\E$. Doan proved the first result of this type, describing the multi-vortex moduli space when $\E$ is semi-stable and $d$ is large.

\begin{prop}\cite{Doan_2018}\label{projectivisation}
    Let $g\geq 1$ and $d>2g-2$. Suppose $\E\ra \Sigma$ is a semi-stable vector bundle of rank $N$ and degree $0$. Then $\N_\Sigma(d,\E)$ is regular and biholomorphic to the projectivisation of a rank $N(d+1-g)$ vector bundle over $J_\Sigma^d$. The same result holds when $d\geq 2g-2$ and $\E$ is stable. 
\end{prop}

\begin{proof}
    Since $\E$ is semi-stable, any line bundle admitting a map to $\E$ must have non-positive degree. If there was some $\L\in J_\Sigma^d$ with $h^0(\L^{-1}\otimes\E^*\otimes K)\neq 0$, there would be a map $\beta:\L\otimes K^{-1}\ra \E^*$. But $\deg(\L\otimes K^{-1}) = d-2(g-1)>0$ by assumption, contradicting the semistability of $\E^*$. Then $h^0(\L\otimes\E)=N(d+1-g)$ by Riemann--Roch and the spaces $H^0(\L\otimes\E)$ form a rank $N(d+1-g)$ vector bundle over $J_\Sigma^d$. Explicitly, this bundle is $V_\E:={p_1}_*(P\otimes p_2^*\E)$, where $P\ra J_\Sigma^d\times\Sigma$ is a Poincar\'{e} line bundle and $p_i$ for $i=1,2$ are the projections to the first and second factor of $J_\Sigma^d\times\Sigma$. It is clear from the definitions that $\N_\Sigma(d,\E)=\P V_\E$.
\end{proof}

Another such result follows from Thakar's computation of the cohomology of $\N_\Sigma(d,\E)$.

\begin{prop}\cite{thakar2025modulispacemultimonopolesriemann}\label{moduliproperties}
    Let $\E\ra\Sigma$ be a generic holomorphic bundle of rank $N$ and degree $0$ over a genus $g$ Riemann surface $\Sigma$. Then $\N_\Sigma=\N_\Sigma(d,\E)$ is a smooth projective variety. It is non-empty whenever
    $$\dim_\C\N_\Sigma=Nd-(N-1)(g-1)\geq 0,$$
    and it is connected whenever $\dim_\C\N_\Sigma\geq 1$. Moreover,
    \begin{itemize}
        \item If $\dim_\C\N_\Sigma=0$ then $\N_\Sigma$ consists of $N^g$ positively oriented points.
        \item If $\dim_\C\N_\Sigma=1$ then $\N_\Sigma$ is a smooth curve of genus $N^g(g-1)+1$.
    \end{itemize}
\end{prop}

\subsection{\texorpdfstring{Determining spin$^c$ structures}{Determining spinc structures}}

Throughout this section, fix $d\in\Z$, $(\mathfrak{p},\eta)\in\mathcal{U}_\varphi^{reg}$, and write $\Upsilon=\Upsilon_{d,\mathfrak{p},\eta}$. Even if we can determine fixed points of $\Upsilon$, understanding which spin$^c$ structure they contribute to requires finding lifts in the principal $\G_L$-bundle $\mathcal{Z}_{\tilde{f}}\ra\N_f$. This problem can be rephrased in terms of the Chern class of a universal line bundle. This bundle will be a family version of the following construction: begin with the trivial $\G_L$-equivariant bundle $\mathcal{Z}_\Sigma\times L\ra \mathcal{Z}_\Sigma\times\Sigma$, which descends to the universal Hermitian line bundle
$$\mathscr{U}_\Sigma = (\mathcal{Z}_\Sigma\times L)/\G_L \ra \N_\Sigma\times\Sigma.$$
It is equipped with a connection which restricts to $A$ on $\mathscr{U}_\Sigma|_{\{[A,\Phi]\}\times\Sigma}\simeq L$. This gives $\mathscr{U}_\Sigma$ the structure of a holomorphic line bundle. We now define the family version of this.

\begin{defn}
    Let $\mathfrak{p}\in\mathcal{P}^{reg}_\varphi$ and $\eta=(\sigma,\tau)\in\mathcal{V}_f^+$. Choose a unitary lift $\tilde{f}$ of $f$ and consider the line bundle
    $$\mathcal{Z}_{\tilde{f}}(d,\mathfrak{p},\tau) \times_{S^1}L_{\tilde{f}}\ra\mathcal{Z}_{\tilde{f}}(d,\mathfrak{p},\tau)\times_{S^1}\Sigma_f.$$
    This descends to
    $$\mathscr{U}_f(d,\mathfrak{p},\tau) := (\mathcal{Z}_{\tilde{f}}(d,\mathfrak{p},\tau) \times_{S^1}L_{\tilde{f}})/\G_L \ra \N_f(d,\mathfrak{p},\tau)\times_{S^1}\Sigma_f,$$
    which we call the universal line bundle.
\end{defn}

Note that $\mathscr{U}_f=\mathscr{U}_f(d,\mathfrak{p},\tau)$ is independent of the lift $\tilde{f}$. This is because for any other lift $g^*\tilde{f}$ for $g\in\G_L$, we have
$$[1,z,v]\sim[0,F_{g^*\tilde{f}}(z),g^*\tilde{f}(v)]=[0,g^*F_{\tilde{f}}(z),g^*\tilde{f}(v)]=[0,F_{\tilde{f}}(z),\tilde{f}(v)].$$

\begin{prop}\label{spincprop}
    Let $\mathfrak{p}\in\mathcal{P}^{reg}_{\varphi}$ and $\eta=(\sigma,\tau)\in\mathcal{V}_f^+$, which are omitted from the notation as before. Suppose $\gamma\in\Gamma(\N_f)$ and consider the induced section $\gamma\times_{S^1}1\in\Gamma(\N_f\times_{S^1}\Sigma_f)$. Define the line bundle 
    $$\mathscr{U}_\gamma = (\gamma\times_{S^1}1)^*\mathscr{U}_f\ra \Sigma_f.$$
    Then if $\tilde{f}$ is a unitary lift of $f$, the following are equivalent:
    \begin{enumerate}
        \item $\gamma$ lifts to a section of $\mathcal{Z}_{\tilde{f}}\ra S^1$.
        \item There is an isomorphism $\mathscr{U}_\gamma\simeq L_{\tilde{f}}$.
        \item $c_1(\mathscr{U}_\gamma)\cdot\mathfrak{s}_{can} = \mathfrak{s}_{d,\tilde{f}}$.
    \end{enumerate}
\end{prop}

\begin{proof}
    The equivalence of points 2 and 3 follows from $c_1(L_{\tilde{f}})\cdot\mathfrak{s}_{can} = \mathfrak{s}_{d,\tilde{f}}$ (see Proposition \ref{spincaction}) and that the action of $H^2(\Sigma_f,\Z)$ on $\Spin^c(\Sigma_f)$ is free and transitive. To prove that 1 is equivalent to 2, suppose that $\widetilde{\gamma}\in\Gamma(S^1,\mathcal{Z}_{\tilde{f}})$ is a lift of $\gamma$. This defines a map $L_{\tilde{f}}\ra \mathscr{U}_\gamma$ by
    $$[t,v]\mapsto [t,\widetilde{\gamma}(t),v]$$
    for $v\in L$. This is well-defined because $\widetilde{\gamma}(0) = F_{\tilde{f}}(\widetilde{\gamma}(1))$ and is a unitary isomorphism of line bundles. Conversely, suppose $\theta:L_{\tilde{f}}\ra\mathscr{U}_{\gamma}$ is an isomorphism. For $t\in S^1$, let $\O_{t}$ be the preimage of $\gamma(t)$ under the quotient map $\mathcal{Z}_\Sigma(d,\mathfrak{p}_t,\tau_t)\ra \N_\Sigma(d,\mathfrak{p}_t,\tau_t)$. Restricting $\theta$ to each fibre $\Sigma_t$ of the mapping torus, we obtain isomorphisms $\theta_t:L\ra (\O_{t}\times L)/\G_L$. Fixing $z\in\O_t$, $\theta_t$ is necessarily of the form
    $$[t,v]\mapsto [t,z,u_t(x)v] = [t,u_t^*z,v]$$
    for $v\in L_x$ and some $u_t\in\G_L$. The element $z_t=u_t^*z$ is independent of the choice of $z$, so we can define the smooth path $\widetilde{\gamma}(t) = z_t$. Since $\theta$ is a well-defined isomorphism, we have
    $$[0,z_0,\tilde{f}(v)]=\theta_0([0,\tilde{f}(v)])=\theta_1([1,v]) = [1,z_1,v] = [0,F_{\tilde{f}}(z_1),\tilde{f}(v)].$$
    Then $z_0=F_{\tilde{f}}(z_1)$ and $\widetilde{\gamma}$ is a well-defined lift. 
\end{proof}

It will be convenient to view the mapping torus $\Sigma_f$ together with a family of complex structures $J\in\mathcal{J}_f(\Sigma,\omega)$ as a family of closed Riemann surfaces $\underline{\Sigma}\ra S^1$ with fibre $\Sigma_t=(\Sigma,J_t)$ over $t\in S^1$. The diffeomorphism $f$ becomes a biholomorphism $f:\Sigma_1\ra\Sigma_0$. A family of unitary connections $B\in\A_\varphi(E)$ is equivalent to a family of holomorphic bundles $\underline{\E}\ra\underline{\Sigma}$. This consists of a holomorphic bundle $\E_t\ra\Sigma_t$ for each $t\in S^1$, glued by a holomorphic isomorphism $\varphi:\E_1\ra f^*\E_0$. In this way, a path of parameters $\mathfrak{p}\in\mathcal{P}_\varphi$\footnote{Here, we are abusing notation and writing $\varphi$ for both the holomorphic isomorphism $\E_1\ra f^*\E_0$ and its composition with the pullback map $f^*\E_0\ra\E_0$.} is equivalent to the data $\underline{\E}\ra\underline{\Sigma}$. We call $\underline{\E}\ra\underline{\Sigma}$ a regular family if $\mathfrak{p}\in\mathcal{P}^{reg}_\varphi$. For such a regular family, define the biholomorphism
$$F_\varphi:\N_{\Sigma_0}(d,\E_0)\ra\N_{\Sigma_1}(d,\E_1),\hspace{4mm}[\L,\alpha] \mapsto [f^*\L,(1\otimes\varphi^{-1})f^*\alpha].$$
Emulating the construction of Definition \ref{familymoduli} in this setting, we obtain the family of projective varieties
$$\N_{\underline{\Sigma}}(d,\underline{\E})=\frac{\{(t,v) : v\in\N_{\Sigma_t}(d,\E_t)\}}{(1,F_{\varphi}(v))\sim(0,v)}\ra S^1.$$

The following lemma is a relative version of the see-saw principle over a smooth base.

\begin{lemma}\label{seesaw}
    Let $S$ be a smooth manifold and suppose $X,Y\ra S$ are families of compact complex manifolds such that $Y$ has connected fibres. Let $p:X\times_S Y\ra X$ be the projection and let $Q\ra X\times_S Y$ be a smooth line bundle which is holomorphic along the fibres of $p$. Then if $Q|_{\{x\}\times Y_s}\simeq \O_{Y_s}$ for every $s\in S$ and $x\in X_s$, there is a family of holomorphic line bundles $R\ra X$ such that $Q\simeq p^*R$. 
\end{lemma}

\begin{proof}
    Define $R:=p_*Q$. Since $Y_s$ is connected, we have 
    $$\dim H^0(\{x\}\times Y_s, Q|_{\{x\}\times Y_s})=\dim H^0(\{x\}\times Y_s,\O_{Y_s})= 1$$
    for all $x\in X_s$. Then $R$ is a smooth line bundle over $X$ which is holomorphic over each fibre $X_s$. Now consider the evaluation map $p^*R\ra Q$. On each fibre $p^{-1}(x)=\{x\}\times Y_s$ it is given by the tautological isomorphism
    $$\O_{Y_s}\otimes_{\C} H^0(Y_s,\O_{Y_s})\ra \O_{Y_s},$$
    so we have $Q\simeq p^*R$.
\end{proof}

For $\tau$ with $d-\bar{\tau}<0$, a fibre-wise application of the Hitchin--Kobayashi correspondence (Theorem \ref{vortexmodulispaces}) determines an isomorphism of smooth fibre bundles
$$\kappa:\N_f(d,\mathfrak{p},\tau)\ra \N_{\underline{\Sigma}}(d,\underline{\E}).$$
In what follows, a family of degree $d$ Poincar\'{e} line bundles is a family of holomorphic line bundles $P\ra J_{\underline{\Sigma}}^d\times_{S^1}\underline{\Sigma}$ such that $P|_{\{\L\}\times \Sigma_t}\simeq\L$ for all $t\in S^1$. 


\begin{prop}\label{chernclasscorollary}
    Suppose $\underline{\E}\ra\underline{\Sigma}$ is a regular family of holomorphic bundles which determines $\mathfrak{p}\in\mathcal{P}_\varphi^{reg}$, and suppose $\tau\in\Omega^0_f(\Sigma)$ has constant mean value $\bar{\tau}$ satisfying $d-\bar{\tau}<0$. Then for any family of degree $d$ Poincar\'{e} line bundles $P\ra J_{\underline{\Sigma}}^d\times_{S^1}\underline{\Sigma}$ and any smooth section $\overline{\gamma}\in\Gamma(\N_f(d,\mathfrak{p},\tau))$, we have
    $$c_1(\mathscr{U}_{\overline{\gamma}}) = c_1((\gamma\times_{S^1}1)^*P)\in H^2(\underline{\Sigma},\Z),$$
    where $\gamma:=\pi\circ\kappa\circ\overline{\gamma}\in\Gamma(J_{\underline{\Sigma}}^d)$. 
\end{prop}

\begin{proof}
    Consider the line bundle
    $$Q=(\kappa^{-1}\times_{S^1}1)^*\mathscr{U}_{f}\otimes\big((\pi\times_{S^1}1)^*P\big)^{*}\ra \N_{\underline{\Sigma}}(d,\underline{\E})\times_{S^1}\underline{\Sigma}.$$
    From our description of the universal bundle over $\N_\Sigma\times\Sigma$ at the beginning of this section, we find that
    $$(\kappa^{-1}\times_{S^1}1)^*\mathscr{U}_{f}|_{\{(\L,\alpha)\}\times\Sigma_t}\simeq \L,$$
    so $Q$ is holomorphically trivial along the fibres of the projection
    $$p:\N_{\underline{\Sigma}}(d,\underline{\E})\times_{S^1}\underline{\Sigma}\ra \N_{\underline{\Sigma}}(d,\underline{\E}).$$
    Then Lemma \ref{seesaw} applies and there is a family of holomorphic line bundles $R\ra\N_{\underline{\Sigma}}(d,\underline{\E})$ such that $Q\simeq p^*R$. It suffices to show that the Chern class of $(\kappa\circ\overline{\gamma}\times_{S^1}1)^*p^*R$ vanishes. This holds because $H^2(S^1,\Z)=0$ and the following diagram commutes:
    \[\begin{tikzcd}
	    {\underline{\Sigma}} & {\N_{\underline{\Sigma}}\times_{S^1}\underline{\Sigma}} \\
	    {S^1} & {\N_{\underline{\Sigma}}}
	    \arrow["{\overline{\gamma}\times_{S^1}1}", from=1-1, to=1-2]
	    \arrow[from=1-1, to=2-1]
	    \arrow["p", from=1-2, to=2-2]
	    \arrow["{\overline{\gamma}}"', from=2-1, to=2-2]
    \end{tikzcd}\]
\end{proof}

\begin{prop}\label{sectiontochernclass}
    Let $\pi_0\Gamma(J_{\underline{\Sigma}}^d)$ be the set of homotopy classes of sections of $J_{\underline{\Sigma}}^d$. For a family of degree $d$ Poincar\'{e} line bundles $P\ra J_{\underline{\Sigma}}^d\times_{S^1}\underline{\Sigma}$ and a section $\gamma\in\Gamma(J_{\underline{\Sigma}}^d)$, define $P_\gamma = (\gamma\times_{S^1}1)^*P\ra \underline{\Sigma}$. Then the map of sets
    \begin{equation}\label{homotopyinjection}
        \pi_0\Gamma(J_{\underline{\Sigma}}^d)\ra H^2(\underline{\Sigma},\Z),\hspace{3mm}[\gamma]\mapsto c_1(P_\gamma)
    \end{equation}
    is well-defined and injective. Moreover, it is surjective onto the preimage of $d$ under the restriction map $H^2(\underline{\Sigma},\Z)\ra H^2(\Sigma,\Z)\simeq \Z$. 
\end{prop}

\begin{proof}
    The argument of Corollary \ref{chernclasscorollary} shows that the map \eqref{homotopyinjection} is independent of the choice of $P$. Moreover, it is well-defined because $P_\gamma\simeq P_{\gamma'}$ whenever $[\gamma]=[\gamma']$ \cite[Theorem 1.6]{hatcherVBKT}. To prove injectivity, it suffices to consider the $d=0$ case, where the map \eqref{homotopyinjection} is a group homomorphism (addition of homotopy classes is interpreted fibre-wise). Choosing a model fibre $\Sigma$ of $\underline{\Sigma}$, there is a monodromy map $f:\Sigma\ra \Sigma$ such that $\underline{\Sigma}$ is identified with the mapping torus $\Sigma_f$. The family of Jacobian tori $J_{\underline{\Sigma}}^0$ can then be identified with the mapping torus $J_f$ of
    \begin{equation}\label{jacobianmonodromy}
        (f^{-1})^*:\frac{H^1(\Sigma,\R)}{H^1(\Sigma,\Z)}\ra\frac{H^1(\Sigma,\R)}{H^1(\Sigma,\Z)}.
    \end{equation}
    Any choice of Poincar\'{e} bundle is isomorphic to
    \begin{equation}\label{topologicalpoincare}
        P\simeq (\R\times H^1(\Sigma,\R)\times\Sigma\times\C)/\sim
    \end{equation}
    as a smooth vector bundle, where
    $$(t+1,\alpha,x,z)\sim(t,(f^{-1})^*\alpha,f(x),z),\hspace{3mm}(t,\alpha,x,z)\sim (t,\alpha + m, x,m(x)z),$$
    for all $m\in H^1(\Sigma,\Z)\simeq [\Sigma,U(1)]$. Now consider the following diagram:
    \[\begin{tikzcd}
	    & {\pi_0\Gamma(J_f)} &&& \\
	    0 & {\coker(1-f^*)} & {H^2(\Sigma_f,\Z)} & {H^2(\Sigma,\Z)} & 0
	    \arrow["\ell"', from=1-2, to=2-2]
	    \arrow["{[\gamma]\mapsto c_1(P_{\gamma})}", from=1-2, to=2-3]
	    \arrow[from=2-1, to=2-2]
	    \arrow["\delta", from=2-2, to=2-3]
	    \arrow[from=2-3, to=2-4]
	    \arrow[from=2-4, to=2-5]
    \end{tikzcd}\]
    The short exact sequence on the bottom row comes from the Wang long exact sequence for $\Sigma_f$, with boundary map $\delta$. To define the homomorphism $\ell$, lift a section $\gamma$ to a path $\widetilde{\gamma}:[0,1]\ra H^1(\Sigma,\R)$ satisfying 
    $$\widetilde{\gamma}(1) \equiv f^*\widetilde{\gamma}(0) \mod H^1(\Sigma,\Z),$$
    and take the equivalence class of $\widetilde{\gamma}(1) - f^*\widetilde{\gamma}(0)$ in $\coker(1-f^*)=H^1(\Sigma,\Z)/\im(1-f^*)$. It is easy to check that this is well-defined. Since $\delta$ is injective, it suffices to prove that $\ell$ is an isomorphism and the triangle commutes. To show that $\ell$ is surjective, suppose $\alpha\in H^1(\Sigma,\Z)$ and define $\widetilde{\gamma}(t)=t\alpha$. Then $\widetilde{\gamma}(1)-f^*\widetilde{\gamma}(0)=\alpha$, so $\widetilde{\gamma}$ descends to a section $\gamma\in\Gamma(J_f)$ satisfying $\ell([\gamma]) = [\alpha]$ by construction. To prove injectivity, suppose $\ell([\gamma])=0$. Then for any lift $\widetilde{\gamma}$, there is $\alpha\in H^1(\Sigma,\Z)$ such that
    $$\widetilde{\gamma}(1) - f^*\widetilde{\gamma}(0) = \alpha - f^*\alpha,$$
    and $\widetilde{\gamma}(t)-\alpha$ is another lift of $\gamma$. So without loss of generality, assume that $\widetilde{\gamma}(1) = f^*\widetilde{\gamma}(0)$. For such a lift, define the homotopy $h_s(t)=s\widetilde{\gamma}(t)$, which satisfies
    $$h_0(t)=0,\hspace{3mm}h_1(t)=\widetilde{\gamma}(t),\hspace{3mm}h_s(1)=f^*h_s(0),$$
    for all $s,t$. Then $h_t$ descends to a homotopy between $\gamma$ and the zero section of $J_f$. To show that the triangle commutes, interpret the map $\delta$ in the following way. Given $[\alpha]\in\coker(1-f^*)$, choose a representative $\alpha\in H^1(\Sigma,\Z)$. Interpreting this as a map $\alpha:\Sigma\ra U(1)$ via the isomorphism $H^1(\Sigma,\Z)\simeq [\Sigma,U(1)]$, define the line bundle
    $$L_{\alpha} = (\R\times\Sigma\times\C)/(t+1,x,z)\sim(t,f(x),\alpha(x)z)$$
    over $\Sigma_f$. Then $\delta([\alpha])=c_1(L_\alpha)$, which can be shown to be independent of the choices we made. To relate this to the map $[\gamma]\mapsto c_1(P_{\gamma})$, observe that \eqref{topologicalpoincare} implies
    $$P_{\gamma}=(\gamma\times_{S^1}1)^*P\simeq L_{\alpha},$$
    where $\alpha = \widetilde{\gamma}(1)-f^*\widetilde{\gamma}(0)$ for a lift $\widetilde{\gamma}:[0,1]\ra H^1(\Sigma,\R)$.
\end{proof}

By combining Propositions \ref{spincprop}, \ref{chernclasscorollary} and \ref{sectiontochernclass}, we get the following topological criterion to test if two fixed points of $\Upsilon$ are in the same subset $\Fix_{\tilde{f}}(\Upsilon)$.

\begin{cor}\label{fixedpointspinc}
    Two fixed points of $\Upsilon$ are in the same subset $\Fix_{\tilde{f}}(\Upsilon)$ if and only if they determine homotopic sections of $J_{\underline{\Sigma}}^d$.
\end{cor}

\begin{remark}
    At this point, it is important to emphasise that only the homotopy class of $\Upsilon$ matters for counting fixed points in the components $\Fix_{\tilde{f}}(\Upsilon)$. So even though the adiabatic multi-monopole equations are related to a particular choice of monodromy map, the computation of its fixed points and their spin$^c$ structures is a purely topological problem.
\end{remark}

\subsection{Multi-monopoles for spin\texorpdfstring{$^c$}{c} structures of large degree}

Having understood how spin$^c$ structures are determined by fixed points, we can apply Theorem \ref{maintheorem} to construct some examples of multi-monopoles on mapping tori. The first case to consider is when $d$ is large compared to the genus $g$ of $\Sigma$. Proposition \ref{projectivisation} shows that in this case, $\N_\Sigma$ is compact for all semi-stable $\E$. This implies that all fibre-wise semi-stable families $\underline{\E}\ra\underline{\Sigma}$ are regular. The following proposition counts the fixed points of the monodromy map $\Upsilon$ for such families. 

\begin{prop}\label{projectivebundlemm}
    Suppose $g\geq 1$ and $d>2g-2$. Let $\underline{\E}\ra\underline{\Sigma}$ be a family of semi-stable bundles over a family of genus $g$ Riemann surfaces $\underline{\Sigma}$. Then for any spin$^c$ structure $\mathfrak{s}_{d,\tilde{f}}$ on the underlying mapping torus $\Sigma_f$ of $\underline{\Sigma}$, the signed count of fixed points in $\Fix_{\tilde{f}}(\Upsilon)$ is equal to
    \begin{equation}\label{fixptcount}
        \textnormal{sign}(\det(1-f^*))N(d+1-g),
    \end{equation}
    with the convention that $\textnormal{sign}(0)=0$. 
\end{prop}

\begin{proof}
    Associated to $\Sigma_f$ and the family of complex structures $J\in\mathcal{J}_f(\Sigma,\omega)$ is a family of Jacobian tori, which can be viewed as the mapping torus $J_f:=H^1(\Sigma,S^1)_M$ of 
    $$M:=(f^{-1})^*:H^1(\Sigma,S^1)\ra H^1(\Sigma,S^1),$$
    as in the proof of Proposition \ref{sectiontochernclass}. All of the bundles $V_{\E_t}\ra J_{\Sigma_t}^d$ (constructed in the proof of Proposition \ref{projectivisation}) are isomorphic to a fixed vector bundle $V\ra H^1(\Sigma,S^1)$, so fix an isomorphism of $V_{\E_0}\simeq V$. We may choose the monodromy map $\Upsilon$ of $\N_{\underline{\Sigma}}(d,\underline{\E})$ to be the projectivisation of a bundle map $\phi:V\ra V$ covering $M$. If $f$ has non-degenerate fixed points, then so does $M$. Any fixed points of $\Upsilon=\P\phi$ lie in $\P V_x$ for some $x\in\Fix(M)$ and correspond to eigenvectors of $\phi_x$. Since $\GL(N(d+1-g),\C)$ is connected, $\phi$ is homotopic to some $\phi'$ such that for all $x\in\Fix(M)$, $\phi'_x$ has $N(d+1-g)$ distinct eigenvalues. In this way we obtain $\Upsilon'=\P\phi':\P V\ra\P V$ which is homotopic to $\Upsilon$ and has non-degenerate fixed points.

    Let $\tilde{x}\in\P V_x$ be one of the non-degenerate fixed points of $\Upsilon'$. In a local trivialisation of $\P V\ra H^1(\Sigma,S^1)$ near $x\in\Fix(M)$, the derivative
    $$d_{\tilde{x}}\Upsilon' = \begin{pmatrix}
        1-d_xM & *\\
        0 & 1-d_{\tilde{x}}(\P\phi'_x)
    \end{pmatrix}$$
    is upper triangular. As a fixed point of $\P\phi'_x$, $\tilde{x}$ has a positive index because $\phi'$ is $\C$-linear. Then its index as a fixed point of $\Upsilon'$ is
    \begin{equation}\label{fixptindex}
        \ind_{\Upsilon'}(\tilde{x}) = \textnormal{sign}(\det(1-d_xM)) = \textnormal{sign}(\det(1-f^*)).
    \end{equation}
    The homotopy long exact sequences for the projective bundles show that
    $$\pi_1(\P V)\simeq H^1(\Sigma,\Z),\hspace{5mm}\pi_1(\P V_{\Upsilon'})\simeq \pi_1(J_f).$$
    This shows that the short exact sequences of the fibrations $\P V_{\Upsilon'}\ra S^1$ and $J_f\ra S^1$ are equivalent. Then $\pi_0\Gamma(\P V_{\Upsilon'})\simeq \pi_0\Gamma(J_f)$, which can be further identified with $\coker(1-f^*)$ as in the proof of Proposition \ref{sectiontochernclass}. This means that the homotopy class determined by any fixed point in the fibre over $x\in\Fix(M)$ comes from that of $x$ itself. If $\det(1-f^*)=0$, the index vanishes for all fixed points. On the other hand, if $1-f^*$ is invertible on $H^1(\Sigma,\R)$, we have
    $$\Fix(M)=\frac{(1-f^*)^{-1}H^1(\Sigma,\Z)}{H^1(\Sigma,\Z)}.$$
    The map assigning a fixed point to its associated homotopy class is given by
    $$\Fix(M)\xrightarrow{1-f^*} \coker(1-f^*)\simeq \pi_0\Gamma(J_f).$$
    This is a bijection, so every fixed point of $M$ determines a unique homotopy class and every homotopy class is realised by some fixed point. There are $N(d+1-g)$ fixed points of $\Upsilon'$ for each homotopy class, each with the same index \eqref{fixptindex}, so the total count is given by \eqref{fixptcount}. 
\end{proof}

Theorem \ref{maintheorem} implies the following corollary.

\begin{cor}\label{projcor}
    Suppose $g\geq 1$, $d>2g-2$, and that $1-f^*$ is invertible. Then for any spin$^c$ structure $\mathfrak{s}_{d,\tilde{f}}$, there is a chamber for which there are at least $N(d+1-g)$ multi-monopoles on $\Sigma_f$.
\end{cor}

\subsection{Multi-monopoles on genus 1 mapping tori}

In this section, we use Theorem \ref{maintheorem} to produce multi-monopoles in various chambers on mapping tori with genus 1 fibres (families of elliptic curves). We will restrict to families of semi-stable bundles, although unstable bundles can also be studied (see \cite[Section 8.4]{Doan_2018}). When $d<0$, the multi-vortex moduli space for an elliptic curve is empty, so there are no fixed points. Proposition \ref{projectivebundlemm} treats the $d>0$ case, so it remains to describe what happens for $d=0$. 

\begin{ex}
    Before making our general construction, we motivate it by producing $2$-monopoles on the 3-torus $S^1\times\Sigma$ from families of $2$-vortices on a fixed elliptic curve $\Sigma$. Let $\E\ra\Sigma$ be a semi-stable $\SL(2,\C)$-bundle, which is of the form $\E=A\oplus A^{-1}$ for a line bundle $A\in J_\Sigma^0$. Since $K_\Sigma=\O$, the moduli space $\M_\Sigma(0,\E)$ of 2-vortices consists of triples $(\L,\alpha,\beta)$ satisfying $\beta\alpha=0$ for
    $$\alpha\in \P(H^0(\L\otimes A)\oplus H^0(\L\otimes A^{-1})),$$
    $$\beta\in H^0(\L^{-1}\otimes A^{-1})\oplus H^0(\L^{-1}\otimes A).$$
    Suppose $A^2\neq \O$, which is the generic case. Then the only choices for $\alpha$ are $[1:0]$ with $\L=A^{-1}$ and $[0:1]$ with $\L=A$. In both of these cases we have $\beta=0$, so the moduli space is compact and consists of two points:
    $$\M_\Sigma(0,A\oplus A^{-1}) = \N_\Sigma(0,A\oplus A^{-1}) = \{(A,[0:1]),(A^{-1},[1:0])\}.$$
    When $A^2=\O$ (there are four such line bundles), we must have $\L=A$ and $\alpha\in\P^1$. Since $\beta\in\C^2$ satisfies $\beta\alpha=0$, we have
    $$\M_\Sigma(0,A\oplus A^{-1}) \simeq T^*\P^1.$$
    This is the simplest example of the non-compactness phenomenon for multi-vortex moduli spaces. Now consider the space of all such bundles $\E=A\oplus A^{-1}$ up to isomorphism, which can be identified with the quotient of $J_\Sigma^0$ by the involution $A\mapsto A^{-1}$. This space (which we denote by $\mathcal{R}_\Sigma$) is the $\SU(2)$ representation variety of the torus, also known as the ``pillowcase''. It is homeomorphic to the 2-sphere and has four orbifold points at those bundles $\E=A\oplus A$ with $A^2=\O$. These are the branch points of the double cover $J_\Sigma^0\ra\mathcal{R}_\Sigma$, which has non-trivial monodromy around each one. Denote the complement of the branch set by $\mathcal{R}_\Sigma^*$.

    Now consider the 3-torus $S^1\times\Sigma$, equipped with the product metric induced by the complex structure $J$. A constant family $\E\in\mathcal{R}_\Sigma^*$ corresponds to an $\SU(2)$-connection $B$ on $S^1\times\Sigma$, for which the total count of 2-monopoles is equal to 2 by Theorem \ref{S1invariantstuff}. By considering non-constant loops in $\mathcal{R}_\Sigma^*$, we can produce 2-monopoles in other chambers. Such a loop $\beta(t)=[\E_t]=[A_t\oplus A_t^{-1}]$ induces a monodromy automorphism of the fibre
    $$\{A_0,A_0^{-1}\}\simeq \N_\Sigma(0,\E_0)$$
    over $\beta_0$. Away from the branch set the fibres of $J^0\ra \mathcal{R}_\Sigma$ are discrete, so this monodromy map coincides with the map $\Upsilon$ introduced in Definition \ref{monodromymap}. There are two cases to distinguish:
    \begin{enumerate}
        \item If the winding number of $\beta$ around the branch locus is odd, $\Upsilon$ swaps the two points in $\M_\Sigma(0,\E_0)$ and has no fixed points.
        \item If the winding number of $\beta$ around the branch locus is even, $\Upsilon$ is trivial and fixes both points of $\M_\Sigma(0,\E_0)$.
    \end{enumerate}
    In the first case, there are no fixed points so we cannot conclude that there exist any 2-monopoles in such chambers. In the second case, Theorem \ref{maintheorem} constructs two 2-monopoles on $S^1\times\Sigma$. To complete the picture, we need to describe their spin$^c$ structures. Note that 
    \begin{equation}\label{productspinc}
        H^2(S^1\times\Sigma,\Z)\simeq \Z\oplus H^1(\Sigma,\Z)
    \end{equation}
    is canonically identified with $\Spin^c(S^1\times\Sigma)$, where the first factor is the degree $d$ (which is zero in our case). To determine the other factor, note that the two fixed points correspond to the sections $A_t,A_t^{-1}$ of the trivial fibration over $S^1$ with fibre $J_{\Sigma}^0$. Such sections can be viewed as loops in $J_{\Sigma}^0$, which determine classes $A_+,A_-\in H_1(J_{\Sigma}^0)$. Taking the slant product with the first Chern class of the Poincar\'{e} bundle $P\ra J_\Sigma^0\times\Sigma$ produces classes in $H^1(\Sigma,\Z)$, which we still denote by $A_\pm$. Then Propositions \ref{sectiontochernclass} and \ref{spincprop} imply that the spin$^c$ structures of the 2-monopoles are $(0,A_+)$ and $(0,A_-)$ in the decomposition \eqref{productspinc}. Since the loops are inverses they satisfy $A_+ + A_- = 0$, which is an artefact of working with $\SL(2,C)$-bundles. In future work, we will improve Theorem \ref{maintheorem} to a bijective correspondence. That will complete this example by showing that these are the only 2-monopoles in such chambers, and that there are no 2-monopoles in chambers determined by paths of odd winding number.
\end{ex}

The remainder of this section is concerned with generalising this basic example to construct $N$-monopoles on genus $1$ mapping tori for spin$^c$ structures with $d=0$. We will restrict to degree $0$ semi-stable bundles, which are just direct sums of degree $0$ line bundles \cite{TU19931}. With some minor modifications, the same construction also works for semi-stable $\SL(N,\C)$-bundles.

\begin{defn}
    Let $\Sigma$ be an elliptic curve. For $N\geq 2$, define the following spaces:
    \begin{align*}
        \mathcal{R}_\Sigma &= \{A_1\oplus\cdots\oplus A_N: A_i\in J_\Sigma^0\}/\sim,\\
        \mathcal{R}_{\Sigma}^* &= \{A_1\oplus\cdots\oplus A_N: A_i\in J_\Sigma^0,\hspace{1mm}A_i\neq A_j\textnormal{ for }i\neq j\}/\sim,\\
        \Delta_\Sigma &= \mathcal{R}_\Sigma\setminus \mathcal{R}_{\Sigma}^*,\\
        \widetilde{\mathcal{R}}_{\Sigma} &= \{(\E,A_i) : \E=A_1\oplus\cdots\oplus A_N\in\mathcal{R}_\Sigma,\hspace{1mm} 1\leq i\leq N\},\\
        \widetilde{\mathcal{R}}_{\Sigma}^* &= \{(\E,A_i) : \E=A_1\oplus\cdots\oplus A_N\in\mathcal{R}_{\Sigma}^*,\hspace{1mm} 1\leq i\leq N\},\\
        \widetilde{\Delta}_\Sigma &= \widetilde{\mathcal{R}}_{\Sigma}\setminus\widetilde{\mathcal{R}}_{\Sigma}^*.
    \end{align*}
    The equivalence relation $\sim$ identifies isomorphic holomorphic bundles. If $N$ needs to be specified, we will use the notation $\mathcal{R}_\Sigma(N)$, $\mathcal{R}_\Sigma^*(N)$, and so on. 
\end{defn}

The space $\mathcal{R}_\Sigma$ is nothing but $\Sym^N(J^0_\Sigma)$, viewed as the moduli space of semi-stable holomorphic bundles as described in \cite[Theorem 1]{TU19931}. The natural map $\widetilde{\mathcal{R}}_\Sigma\ra\mathcal{R}_\Sigma$ is an $N$-sheeted covering map, branching along the big diagonal $\Delta_\Sigma$. The monodromy of this branched cover is well-known.

\begin{lemma}
    The fundamental group of $\mathcal{R}_{\Sigma}^*$ is isomorphic to the $N$-strand braid group of $J_\Sigma^0$. Moreover, the monodromy of the covering space $\widetilde{\mathcal{R}}_\Sigma^*\ra\mathcal{R}_\Sigma^*$ is the composition of this isomorphism with the symmetric representation of the braid group. This composition will be denoted by $\rho:\pi_1(\mathcal{R}_\Sigma^*)\ra S_N$, which implicitly depends on an ordering of the fibre over the base-point. 
\end{lemma}


Analogues of these spaces can also be defined for families of elliptic curves over $S^1$.

\begin{defn}
    Let $\underline{\Sigma}\ra S^1$ be a family of elliptic curves over $S^1$ with monodromy map $f:\Sigma_1\ra\Sigma_0$, and write $\Sigma_t$ for the fibre over $t\in S^1$. For $N\geq 2$, define the following spaces:
    \begin{equation*}
    \begin{aligned}[t]
        \mathcal{R}_{\underline{\Sigma}} &= \{(t,\E) : \E\in\mathcal{R}_{\Sigma_t}\}/\sim,\\
        \mathcal{R}_{\underline{\Sigma}}^* &= \{(t,\E) : \E\in\mathcal{R}_{\Sigma_t}^*\}/\sim,\\
        \Delta_{\underline{\Sigma}} &= \{(t,\E) : \E\in\Delta_{\Sigma_t}\}/\sim,
    \end{aligned}
    \qquad
    \begin{aligned}[t]
        \widetilde{\mathcal{R}}_{\underline{\Sigma}}&=\{ (t,\E,A) : (\E,A)\in\widetilde{\mathcal{R}}_{\Sigma_t}\}/\sim,\\
        \widetilde{\mathcal{R}}_{\underline{\Sigma}}^*&=\{(t,\E,A) : (\E,A)\in\widetilde{\mathcal{R}}_{\Sigma_t}^*\}/\sim,\\
        \widetilde{\Delta}_{\underline{\Sigma}} &= \{(t,\E,A) : (\E,A)\in\widetilde{\Delta}_{\Sigma_t}\}/\sim.
    \end{aligned}
    \end{equation*}
    The equivalence relations $\sim$ are given by $(1,f^*\E)\sim(0,\E)$ and $(1,f^*\E,f^*A)\sim(0,\E,A)$.
\end{defn}

As with a fixed curve, there is a relative covering map $\widetilde{\mathcal{R}}_{\underline{\Sigma}}\ra \mathcal{R}_{\underline{\Sigma}}$ with branch locus $\Delta_{\underline{\Sigma}}$. We will also denote its monodromy map by $\rho$. To use these spaces to count multi-monopoles, we need to understand the $N$-vortex moduli spaces for bundles in $\mathcal{R}_\Sigma$. 

\begin{prop}\label{directsumvortexmoduli}
    Let $A_{1},A_{2},\dots,A_m$ be distinct degree 0 line bundles on an elliptic curve $\Sigma$ and let $n_1,n_2,\dots,n_m$ be positive integers summing to $N$. Suppose $\E=\bigoplus_{k=1}^m A_k^{\oplus n_k}$. Then we have
    $$\N_\Sigma(0,\E)\simeq \bigsqcup_{k=1}^m \P^{n_k-1},\hspace{5mm} \M_\Sigma(0,\E)\simeq \bigsqcup_{k=1}^m T^*\P^{n_k-1}.$$
    In particular, if $\E\in\mathcal{R}_{\Sigma}^*$ (so $n_i=1$ and the summands are distinct), then $\M_\Sigma(0,\E)$ is compact and consists of $N$ regular points.
\end{prop}

\begin{proof}
    Suppose first that $m=1$, so $\E=A^{\oplus n}$ for a line bundle $A$. If $[\L,\alpha]\in\N_{\Sigma}(0,\E)$ we require $\alpha\neq 0$, which is only possible when $\L\simeq A^{-1}$. In this case 
    $$\alpha\in \P H^0(A^{-1}\otimes A^{\oplus n})=\P H^0(\O)^{\oplus n}\simeq\P^{n-1},$$
    so $\N_\Sigma(0,\E)\simeq\P^{n-1}$. If we now allow $\beta$ to be non-zero, it also varies in the space $H^0(\O)^{\oplus n}$, subject to the condition $\beta\alpha=0$. We can then identify $\beta$ as a cotangent vector to $\P^{n-1}$ at $\alpha$, so $\M_\Sigma(0,\E)\simeq T^*\P^{n-1}$. 
    For the general case, let $\E=\bigoplus_{k=1}^m A_k^{\oplus n_k}$ for distinct degree 0 line bundles $A_k$. For $\alpha$ to be non-zero, the only possible choices for $\L$ are $A_k^{-1}$. Then $\N_\Sigma(0,\E)$ and $\M_\Sigma(0,\E)$ split as disjoint unions of $\N_\Sigma(0,A_k^{\oplus n_k})$ and $\M_\Sigma(0,A_k^{\oplus n_k})$, respectively.
\end{proof}


The key observation is that these moduli spaces can be assembled into an $N$-sheeted cover of $\mathcal{R}_{\Sigma}^*$, so that the fibre over $\E$ is $\N_\Sigma(0,\E)$. This is isomorphic to the covering space $\widetilde{\mathcal{R}}_{\Sigma}^*\ra \mathcal{R}_{\Sigma}^*$ by identifying $(A_i^{-1},\alpha_i)\in\N_\Sigma(0,\E)$ with $(\E,A_i)\in\widetilde{\mathcal{R}}^*_\Sigma$. The same idea also works for families of elliptic curves.

\begin{prop}\label{coveringisomorphism}
    Let $\underline{\Sigma}\ra S^1$ be a family of elliptic curves. Choose a section $\beta:S^1\ra\mathcal{R}_{\underline{\Sigma}}^*$ and any family of holomorphic bundles $\underline{\E}\ra\underline{\Sigma}$ such that $\beta(t)=[\E_t]$. 
    Then $\underline{\E}$ is regular and $\N_{\underline{\Sigma}}(0,\underline{\E})\ra S^1$ is isomorphic to $\beta^*\widetilde{\mathcal{R}}_{\underline{\Sigma}}^*\ra S^1$. Moreover, the monodromy map $\Upsilon$ can be identified with $\rho(\beta)\in S_N$.
\end{prop}

\begin{proof}
    Any family $\underline{\E}\ra\underline{\Sigma}$ determined by $\beta:S^1\ra\mathcal{R}_{\underline{\Sigma}}^*$ is regular by Proposition \ref{directsumvortexmoduli}. Associated with $\underline{\E}$ is a holomorphic isomorphism $\varphi:\E_1\ra f^*\E_0$. If $\E_t=A_1(t)\oplus\cdots\oplus A_N(t)$, $\varphi$ is given by a matrix
    $$\varphi = \begin{pmatrix}
        \varphi_{11} & \cdots & \varphi_{N1}\\
        \vdots & \ddots & \vdots\\
        \varphi_{1N} & \cdots & \varphi_{NN}
    \end{pmatrix}.$$
    The map $\varphi_{ij}:A_i(1)\ra f^*A_j(0)$ is zero unless $j=\rho(i)$, where $\rho=\rho(\beta)\in S_N$. Define
    $$\mathcal{F}:\N_{\underline{\Sigma}}(0,\underline{\E})\ra \beta^*\widetilde{\mathcal{R}}_{\underline{\Sigma}}^*,\hspace{4mm}[t,(A_i(t)^{-1},\alpha)]\mapsto [t,(\E_t,A_i(t))].$$
    Since $\alpha$ is uniquely determined up to scaling, we will drop it from the notation. Then $F_\varphi$ simply acts as pullback by $f$ and $\mathcal{F}$ is a well-defined isomorphism. Since the fibres are discrete, parallel transport $P_t$ in $\N_{\underline{\Sigma}}(0,\underline{\E})$ must be given by $[0,A_i(0)^{-1}]\mapsto [t,A_i(t)^{-1}]$. If $j=\rho(i)$,
    $$\Upsilon:[0,A_i(0)^{-1}]\overset{P_1}{\mapsto}[1,A_i(1)^{-1}] = [1,f^*A_j(0)^{-1}]\overset{F_\varphi^{-1}}{\mapsto}[0,A_j(0)^{-1}].$$
    Passing through $\mathcal{F}$ identifies this with $\rho$. 
\end{proof}
    

The next corollary summarises our description of the chamber structure and fixed points for families of elliptic curves. It follows directly from Proposition \ref{coveringisomorphism}, Corollary \ref{fixedpointspinc}, and Theorem \ref{maintheorem}. We then use it to prove Theorem \ref{mmexistence}.

\begin{cor}\label{genus1multimonopoles}
    Let $\underline{\Sigma}\ra S^1$ be a family of elliptic curves and let $\beta:S^1\ra\mathcal{R}_{\underline{\Sigma}}^*$ be a section. Then each fixed point of $\rho(\beta)$ determines a unique multi-monopole on the underlying 3-manifold of $\underline{\Sigma}$. The chamber is determined by the homotopy class $[\beta]\in\pi_0\Gamma(\mathcal{R}_{\underline{\Sigma}}^*)$ and the family of complex structures defining $\underline{\Sigma}$. Moreover, by identifying $\mathcal{R}_{\underline{\Sigma}}$ with the family of (fibre-wise) symmetric products of $J_{\underline{\Sigma}}^{0}$, the section $\beta$ can be viewed as a link in $J_{\underline{\Sigma}}^{0}$. Each fixed point of $\rho(\beta)$ determines a component of this link, which is the image of a section $\gamma:S^1\ra J_{\underline{\Sigma}}^0$. The spin$^c$ structure of the multi-monopole is determined by the homotopy class $[\gamma]\in\pi_0\Gamma(J_{\underline{\Sigma}}^0)$ and the condition $d=0$.
\end{cor}

\begin{proof}[Proof of Theorem \ref{mmexistence}.]
    Let $\underline{\Sigma}\ra S^1$ be the family of elliptic curves determined by $J$. By Corollary \ref{genus1multimonopoles}, it suffices to find a section $\beta$ of $\mathcal{R}_{\underline{\Sigma}}^*(N)\ra S^1$ such that:
    \begin{itemize}
        \item $\rho(\beta)$ has (at least) $k:=\sum_{i=1}^m k_i$ fixed points with corresponding sections $\gamma_{j_i}\in\Gamma(J_{\underline{\Sigma}}^0)$ for $1\leq j_i\leq k_i$.
        \item For each $i$, the sections $\gamma_{j_i}$ have a prescribed homotopy class $h_i\in\pi_0\Gamma(J_{\underline{\Sigma}}^0)$.
    \end{itemize}
    To construct $\beta$, choose $k_i$ representatives $\gamma_{j_i}\in\Gamma(J_{\underline{\Sigma}}^0)$ of each homotopy class $h_i$. These can be perturbed so that their images do not intersect, which determines a section $\beta_0$ of $\mathcal{R}_{\underline{\Sigma}}^*(k)\ra S^1$ such that $\rho(\beta_0)$ has $k$ fixed points. Note that $\beta_0$ is certainly not unique, since any of these sections can be homotoped to link any other. Viewing $\beta_0$ as a $k$-strand braid in $J_{\underline{\Sigma}}^0$, add $N-k$ strands to complete it to an $N$-strand braid (there are also many ways to do this). This determines a section $\beta$ of $\mathcal{R}_{\underline{\Sigma}}^*(N)\ra S^1$ with the required properties. 
\end{proof}

%% file: chapters/appendix.tex
\appendix

\section{Operator identities}
\label{appendix:a}

In this appendix, we state and prove several identities between the operators $N,G,S,L,M$ appearing in the linearisation \eqref{blocklinearisation}. Throughout, fix a spin$^c$ structure $\mathfrak{s}_{d,\tilde{f}}$ on $\Sigma_f$. For a lift $\varphi:E\ra E$ of $f$, the family of parameters $\mathfrak{p}\in\mathcal{P}_\varphi^{reg}$ will be given by $\mathfrak{p}_t=(J_t,B_t)$. The perturbation $\eta$ will be determined by $\sigma,\tau$ as in Definition \ref{perturbation}. 

\begin{lemma}\label{operatoridentities}
    Suppose $\Xi=(X,0,Y)=(A,\Phi,0,b,\Psi)$ belongs to either $\A_\e(\mathfrak{s}_{d,\tilde{f}},\mathfrak{p},\eta)$ for $\e>0$ or $\A_0(\mathfrak{s}_{d,\tilde{f}},\mathfrak{p},\eta)$. Then the following identities hold:
    \begin{equation}\label{identity0}
        L_{Y}^*M_{Y} = i\textnormal{Re}\langle\nabla_t\Psi,\cdot\rangle,
    \end{equation}
    \begin{equation}\label{identity1}
        N_{Y}G_{X} + S_{X}^*L_{Y} = 0,
    \end{equation}
    \begin{equation}\label{identity2}
        N_{Y}S_{X}^* + G_{X}L_Y^* + S_{X}^*M_{Y}=R_{\Xi}:=\begin{pmatrix}
            -\star\dot{\star}d & -2i\textnormal{Im}\langle\cdot,\nabla_{A,p}\Psi\rangle\\
            -2\nabla_t\Phi & -i[\nabla_t,\delbar_{A,p}^*]
        \end{pmatrix}.
    \end{equation}
\end{lemma}

\begin{proof}
    Since $(A,\Phi,0,b,\Psi)$ is allowed to solve either the multi-monopole equations \eqref{eq:epsilonSW} for $\e>0$ or the adiabatic limit equations \eqref{eq:epsilonzeroSW1}, we only make use of the $\e$-independent equations of \eqref{eq:epsilonSW} which survive in the adiabatic limit:
    $$-i\nabla_t\Phi + \delbar_{A,\mathfrak{p}}^*\Psi=0,\hspace{4mm}\star(\dot{A}-db-i\sigma) - i\Im\langle\Psi,\Phi\rangle=0.$$
    Throughout the proof we omit subscripts from $N,G,S,L,$ and $M$. The first identity \eqref{identity0} is the simplest:
    \begin{align*}
        L^*M\begin{pmatrix}
            c\\
            \psi
        \end{pmatrix} &= \begin{pmatrix}
            \del_t & -i\Im\langle\Psi,\cdot\rangle
        \end{pmatrix}\begin{pmatrix}
            i\Re\langle\Psi,\psi\rangle\\
            -ic\Psi-i\nabla_t\psi
        \end{pmatrix}\\
        &= i\Re\langle\nabla_t\Psi,\psi\rangle + i\Re\langle\Psi,\nabla_t\psi\rangle + i\Im\langle\Psi,i\nabla_t\psi\rangle\\
        &=i\Re\langle\nabla_t\Psi,\psi\rangle.
    \end{align*}
    To prove the second identity \eqref{identity1}, expand the compositions $NG$ and $S^*L$:
    \begin{align*}
        NG(v) = \begin{pmatrix}
            \star\del_t & -i\Im\langle\Psi,\cdot\rangle\\
            -\langle\Psi,(\cdot)^{0,1}\rangle & i\nabla_t
        \end{pmatrix}\begin{pmatrix}
            -dv\\
            v\Phi
        \end{pmatrix} = \begin{pmatrix}
            -\star d\dot{v} - i\Im\langle\Psi,v\Phi\rangle\\
            \langle\Psi,\delbar v\rangle + i\dot{v}\Phi + iv\nabla_t\Phi
        \end{pmatrix}
    \end{align*}
    \begin{align*}
        S^*L(v) = \begin{pmatrix}
            -\star d & -i\Im\langle\cdot,\Phi\rangle\\
            i(\cdot)\Phi & -\delbar_{A,\mathfrak{p}}^*
        \end{pmatrix}\begin{pmatrix}
            -\dot{v}\\
            v\Psi
        \end{pmatrix} = \begin{pmatrix}
            \star d\dot{v} - i\Im\langle v\Psi,\Phi\rangle\\
            -i\dot{v}\Phi - v\delbar_{A,\mathfrak{p}}^*\Psi - \langle\Psi,\delbar v\rangle
        \end{pmatrix}
    \end{align*}
    Above, we applied $\delbar_{A,\mathfrak{p}}^*(v\Psi) = v\delbar_{A,\mathfrak{p}}^*\Psi + \langle\Psi,\delbar v\rangle$. Adding the terms together we get
    $$(NG+S^*L)(v) = \begin{pmatrix}
        0\\
        -v(i\nabla_t\Phi - \delbar_{A,p}^*\Psi)
    \end{pmatrix},$$
    which vanishes when $(A,\Phi,0,b,\Psi)$ solves equation \eqref{eq:epsilonSW}. Similarly, we expand the compositions $NS^*$, $GL^*$ and $S^*M$ appearing in the third identity \eqref{identity2}:
    \begin{align*}
        NS^*\begin{pmatrix}
            c\\
            \psi
        \end{pmatrix} &= \begin{pmatrix}
            \star\del_t & -i\Im\langle\Psi,\cdot\rangle\\
            -\langle\Psi,(\cdot)^{0,1}\rangle & i\nabla_t
        \end{pmatrix}\begin{pmatrix}
            -\star dc -i\Im\langle\psi,\Phi\rangle\\
            ic\Phi -\delbar_{A,\mathfrak{p}}^*\psi
        \end{pmatrix}\\
        &= \begin{pmatrix}
            -\star\dot{\star}dc + d\dot{c} - i\Im\langle i\nabla_t\psi,\Phi\rangle + i\Im\langle \psi,i\nabla_t\Phi\rangle -i\Im\langle\Psi,ic\Phi\rangle + i\Im\langle\Psi,\delbar_{A,\mathfrak{p}}^*\psi\rangle\\
            -i\langle\Psi,\delbar c\rangle + \langle\Psi,\langle\psi,\Phi\rangle\rangle - \dot{c}\Phi - c\nabla_t\Phi - i\nabla_t\delbar_{A,\mathfrak{p}}^*\psi
        \end{pmatrix}
    \end{align*}
    \begin{align*}
        GL^*\begin{pmatrix}
            c\\
            \psi
        \end{pmatrix} &= \begin{pmatrix}
            -d\\
            (\cdot)\Phi
        \end{pmatrix}\begin{pmatrix}
            \dot{c} & -i\Im\langle\Psi,\psi\rangle
        \end{pmatrix}\\
        &= \begin{pmatrix}
            -d\dot{c} + id\Im\langle\Psi,\psi\rangle\\
            \dot{c}\Phi - i\Im\langle\Psi,\psi\rangle\Phi
        \end{pmatrix}
    \end{align*}
    \begin{align*}
        S^*M\begin{pmatrix}
            c\\
            \psi
        \end{pmatrix} &= \begin{pmatrix}
            -\star d & -i\Im\langle\cdot,\Phi\rangle\\
            i(\cdot)\Phi & -\delbar_{A,\mathfrak{p}}^*
        \end{pmatrix}\begin{pmatrix}
            i\textnormal{Re}\langle\Psi,\psi\rangle\\
            -ic\Psi -i\nabla_t\psi
        \end{pmatrix}\\
        &= \begin{pmatrix}
            -i\star d\Re\langle\Psi,\psi\rangle +i\Im\langle ic\Psi,\Phi\rangle + i\Im\langle i\nabla_t\psi,\Phi\rangle\\
            -\Re\langle\Psi,\psi\rangle\Phi + ic\delbar_{A,\mathfrak{p}}^*\Psi + i\langle\Psi,\delbar c\rangle + i\delbar_{A,\mathfrak{p}}^*\nabla_t\psi
        \end{pmatrix}
    \end{align*}
    Before adding everything together, note that
    \begin{align*}
        -i\star d\Re\langle\Psi,\psi\rangle + id\Im\langle\Psi,\psi\rangle &= -2i\Im(\del\langle\psi,\Psi\rangle)\\
        &= -2i\Im\langle\del_{A,p}\psi,\Psi\rangle - 2i\Im\langle\psi,\delbar_{A,\mathfrak{p}}\Psi\rangle.
    \end{align*}
    Here, the last equality can be deduced from a computation in local holomorphic coordinates. A similar local computation using $\delbar^*=-\star\del\star$ and $\del^*=-\star\delbar\star$ shows that 
    $$\langle\del_{A,\mathfrak{p}}\psi,\Psi\rangle = -\frac{1}{2}\langle\delbar_{A,\mathfrak{p}}^*\psi,\Psi\rangle,\hspace{4mm}\langle\psi,\delbar_{A,\mathfrak{p}}\Psi\rangle = -\frac{1}{2}\langle\psi,\del_{A,\mathfrak{p}}^*\Psi\rangle.$$
    Here, the factor $1/2$ comes from the fact that $|d\overline{z}|^2=1$ in the rescaled metric (see \eqref{rescalings}). The first of these yields
    $$-i\star d\Re\langle\Psi,\psi\rangle + id\Im\langle\Psi,\psi\rangle = -i\Im\langle\Psi,\delbar_{A,\mathfrak{p}}^*\psi\rangle - 2i\Im\langle\psi,\delbar_{A,\mathfrak{p}}\Psi\rangle,$$
    and the second is used below to get our final result. 
    \begin{align*}
        (NS^*+GL^*+S^*M)\begin{pmatrix}
            c\\
            \psi
        \end{pmatrix} &= \begin{pmatrix}
            -\star\dot{\star}dc + i\Im\langle\psi,i\nabla_t\Phi\rangle - i \star d\Re\langle\Psi,\psi\rangle + id\Im\langle\Psi,\psi\rangle + i\Im\langle\Psi,\delbar_{A,\mathfrak{p}}^*\psi\rangle\\
            -c\nabla_t\Phi + ic\delbar_{A,\mathfrak{p}}^*\Psi + i\delbar_{A,\mathfrak{p}}^*\nabla_t\psi - i\nabla_t\delbar_{A,\mathfrak{p}}^*\psi
        \end{pmatrix}\\
        &= \begin{pmatrix}
            -\star\dot{\star}dc + i\Im\langle\psi,\delbar_{A,\mathfrak{p}}^*\Psi\rangle - 2i\Im\langle\psi,\delbar_{A,\mathfrak{p}}\Psi\rangle \\
            -2c\nabla_t\Phi - i[\nabla_t,\delbar_{A,\mathfrak{p}}^*]\psi
        \end{pmatrix}\\
        &= \begin{pmatrix}
            -\star\dot{\star}dc - 2i\Im\langle\psi,\nabla_{A,\mathfrak{p}}\Psi\rangle  \\
            -2c\nabla_t\Phi - i[\nabla_t,\delbar_{A,\mathfrak{p}}^*]\psi
        \end{pmatrix}
    \end{align*}
\end{proof}